\documentclass[11pt,oneside]{amsart}
\usepackage{fouriernc} 

\usepackage{sty/Packages}       
\usepackage{sty/Alphabets}      
\usepackage{sty/Hyphenation}    
\usepackage{sty/PageSetup}      
\usepackage{sty/NewEnv}         
\usepackage{sty/Definitions}    
\usepackage{sty/TikzSetup}      
\usepackage{sty/AuthorInfo}     

\hypersetup{
 pdfkeywords={permutative category,symmetric monoidal bicategory},
 pdfauthor={Nick Gurski,Niles Johnson,Angélica Osorno},
}


\newcommand\genatop[2]{\genfrac{}{}{0pt}{}{#1}{#2}}

\newcommand{\bilin}{\mathpzc{Bilin}}
\newcommand{\permcat}{\mathpzc{PermCat}}
\newcommand{\permcats}{\mathpzc{PermCat}_s}

\newcommand{\txcoprod}{{\textstyle\coprod}}
\newcommand{\bincoprod}{\mathbin{\txcoprod}}

\let\oldboxplus\boxplus
\renewcommand{\boxplus}{\mathbin{\scalebox{.75}{$\mathbf{\oldboxplus}$}}}

\newcommand{\internalStrict}[2]{\bigl[#1 , #2\bigr]} 
\NewDocumentCommand \strict
{ > { \SplitArgument { 1 } { , } } m }
{ \internalStrict #1 }

\newcommand{\internalBHOM}[2]{\Hom\bigl(#1 , #2\bigr)} 
\NewDocumentCommand \BHom
{ > { \SplitArgument { 1 } { , } } m }
{ \internalBHOM #1 }

\newcommand{\andspace}{\quad \text{and} \quad}

\newcommand{\ifspace}{\quad \text{if} \quad}
\newcommand{\inspace}{\quad \text{in} \quad}
\newcommand{\forspace}{\quad \text{for} \quad}
\newcommand{\withspace}{\quad \text{with} \quad}

\newcommand{\To}{\Rightarrow}
\newcommand{\tsum}{\textstyle\sum}
\newcommand{\ssum}{\scriptstyle\sum}

\newcommand{\tlprod}[2]{{\textstyle\prod\limits_{#1}^{#2}}}
\newcommand{\cpre}{c_{\mathrm{pre}}}
\newcommand{\cpost}{c_{\mathrm{post}}}

\newcommand{\Mdot}{M^\bullet}
\newcommand{\Wdot}{W^\bullet}

\newcommand{\patil}{\wt{\pa}}

\newcommand{\bigT}{\mathbf{T}}
\newcommand{\bigP}{\mathbf{P}}
\newcommand{\grtimes}{\mathbin{\scalebox{.7}{$\boxtimes$}}} 

\newcommand{\Phidot}{\Phi^\bullet}

\newcommand{\dL}{\mathpzc{L}}
\newcommand{\dLdot}{\mathpzc{L}^\bullet}
\newcommand{\dR}{\mathpzc{R}}
\newcommand{\dRdot}{\mathpzc{R}^\bullet}

\newcommand{\dA}{\mathpzc{A}}
\newcommand{\dAdot}{\mathpzc{A}^\bullet}

\newcommand{\etaa}{\eta}
\newcommand{\etaainv}{\eta^{-1}}
\newcommand{\epza}{\epz}
\newcommand{\epzainv}{\epz^{-1}}

\newcommand{\dB}{\mathpzc{B}}

\newcommand{\B}{\mathsf{B}} 
\newcommand{\C}{\mathsf{C}} 
\newcommand{\syl}{\nu} 
\newcommand{\omdot}{\om^\bullet}

\newcommand{\smallnodes}{\tikzset{every node/.style={scale=.6}}}
\newcommand{\smallzero}{.7}

\tikzset{anno/.style={color=green!60!black}}
\newcommand{\annoarg}[1]{{\color{green!60!black}#1}}
\newcommand{\natA}{\text{nat\ }\dA}
\newcommand{\natAdot}{\text{nat\ }\dAdot}
\newcommand{\natB}{\text{nat\ }\dB}

\newcommand{\uni}{u} 

\newcommand{\II}{\mathcal{I}}
\newcommand{\JJ}{\mathcal{J}}

\title{The symmetric monoidal 2-category of permutative categories}
\authorGurski
\authorJohnson
\authorOsorno
\date{15 November 2023}

\hypersetup{bookmarksdepth=2}

\begin{document}

\begin{abstract}
  We define a tensor product for permutative categories and prove a number of key properties.
We show that this product makes the 2-category of permutative categories closed symmetric monoidal as a bicategory.

\end{abstract}

\maketitle
\tableofcontents 

\section{Introduction}
The category of commutative monoids is complete and cocomplete, it has an internal hom given by pointwise operations on monoid homomorphisms, and it has a tensor product constructed in analogy to that of abelian groups. This tensor product can be defined explicitly via formal sums of simple tensors modulo distributivity relations, or via a universal property using bilinear maps. 
A category with these features (complete, cocomplete, closed symmetric monoidal) is called a B\'enabou cosmos \cite{Str74Cosmoi}.

Symmetric monoidal categories admit a categorified version of this structure.
Such structure has been investigated at varying levels of detail \cite{HP02Pseudo, Sch07Tensor, GGN15Universality, Bou17Skew}.
As with most higher-categorical analogues of well-known algebraic structures, the construction of this categorified B\'enabou cosmos structure is more work than one usually anticipates, often tedious, but nevertheless layered with curious surprises. 

The primary goal of this paper is to provide a detailed and comprehensive construction of the closed symmetric monoidal 2-category of permutative categories, symmetric monoidal functors, and monoidal transformations. 
We have also included an exploration of direct sums, in the bicategorical sense, in this 2-category, as the interaction between tensors and sums is an essential ingredient in classic constructions and in arguments from the theory of abelian groups.

\subsection*{Previous work on such tensor products}

In many closed monoidal categories of interest, the internal hom has a simpler construction than the tensor product: it is often given by equipping the hom-sets of the underlying category with some additional structure. 
The same is true for the 2-category of permutative categories, but proceeding any further into the theory reveals as many differences as similarities. 

There is a notion of bilinear or multilinear symmetric monoidal functor, going back to work of May \cite{may1980pairings} and Elmendorf and Mandell \cite{EM2006Rings}. In both cases, their multilinear functors are strictly unital because they have stable homotopy theoretic applications in mind, so basepoints need to be preserved strictly. Furthermore, neither work discusses the existence of a universal object for bilinear functors.

Hyland and Power \cite{HP02Pseudo} describe bilinearity and a corresponding universal object in the more general context of $T$-algebras over a 2-monad, including the closed, sometimes symmetric monoidal structure that arises. In this current work, we give a direct definition in the context of permutative categories. In particular, we provide many details omitted from that paper, and in several instances the resulting descriptions or arguments look very little like those hinted at by Hyland and Power. 

Later work by Bourke \cite{Bou17Skew} rigorously completed the project begun by Hyland-Power, but had a different aim from what we pursue here and so, unsurprisingly, sidestepped the detailed constructions we present. 
Having those details worked out explicitly is useful for applications, necessary for extending the theory to other related structures like Picard categories or graded versions, and often reveals interesting features that would otherwise be hidden. 
The structure we detail has the same unit and internal hom as Bourke's, and we believe the rest of our structure agrees with his, but we have not checked the remaining details.

We take this opportunity to comment on two other related approaches to the same or similar problems. 
The first appears in a preprint by Schmitt \cite[Section~12]{Sch07Tensor} and only addresses the case of \emph{lax} symmetric monoidal functors. 
The resulting structure has more in common with a skew monoidal category \cite{Szl2012Skew}, and only yields a symmetric monoidal \emph{category} after quotienting out all 2-cells. 

The second approach is from the $(\infty, 1)$-categorical literature in \cite{GGN15Universality}. 
Once again, the methods employed there were aimed at a different problem that had a natural home in the world of quasicategories, whereas our interests lie in studying objects and constructions of a 2-dimensional (or perhaps 3-dimensional) nature. 
Given that $(\infty, 1)$- and 2-categorical universal properties do not, as a general principle, coincide, we believe it best to view our work and that of \cite{GGN15Universality} as related in spirit but fundamentally different in practice.

\subsection*{Our main result}

Throughout this paper, we say that $\B$ is a \emph{symmetric monoidal 2-category} if it is a 2-category that is symmetric monoidal as a bicategory (see \cref{sec:smb-defn} for the definition of symmetric monoidal bicategory).
We define closed symmetric monoidal 2-category also in the bicategorical sense, as follows. 
\begin{defn}\label{defn:closed-smb}
  A symmetric monoidal bicategory $\bigl(\B, \otimes\bigr)$ is \emph{closed} if it is equipped with a pseudofunctor
  \[
    \BHom{-,-}\cn \B^\op \times \B \to \B
  \]
  and equivalences
  \[
    \B \bigl( X \otimes Y, Z \bigr) \fto{\hty} \B \bigl( X, \BHom{Y,Z} \bigr)
  \]
  that are pseudonatural in $X$, $Y$, and $Z$.
  We say that $\B$ is a \emph{closed symmetric monoidal 2-category} if it is a 2-category that is closed symmetric monoidal as a bicategory.
\end{defn}
Thus, \cref{thm:biadj,thm:permcat-smb} show the following.
\begin{thm}\label{cor:permcat-clsmb}
  The 2-category $\permcat$, consisting of small permutative categories, strong symmetric monoidal functors, and monoidal transformations, is a closed symmetric monoidal 2-category.
\end{thm}
The relevant data are defined in \cref{sec:unit,sec:B,sec:A,sec:Adot,sec:2Ddata} and are not entirely trivial, but still considerably simpler than those of a general symmetric monoidal 2-category.
Each of the associator $\dA$, left unitor $\dL$, and right unitor $\dR$ is an equivalence, not an isomorphism, and there is a nontrivial left 2-unitor $\la$.
But the other 2-dimensional data, $\pi$, $\mu$, $\rho$, $R_{-|--}$, $R_{--|-}$, and $\syl$, are all identities.
A key technical ingredient for our proof of \Cref{thm:permcat-smb} is a coherence theorem~\ref{thm:iterated-tensor-coherence} for structure morphisms in iterated tensor products.

This notion of symmetric monoidal structure for 2-categories is weaker and generally distinct from the $\Cat$-enriched notion described in, e.g., \cite[Section~1.5]{JYringIII}.
For example, it has been noted in \cite[p.~166]{EM2006Rings}, that small permutative categories do not appear to support the structure of a closed symmetric monoidal 1-category.
In \cite[5.7.23 and~10.2.17]{JYringIII} it is shown that neither the smash product of pointed multicategories, nor that on the subcategory of $\cM\underline{1}$-modules, restricts to a symmetric monoidal 1-category structure for small permutative categories.
In both cases, the corresponding monoidal unit is a multicategory that is not equivalent to a permutative category.
Since a $\Cat$-enriched symmetric monoidal structure has an underlying symmetric monoidal 1-category, negative results about 1-categorical symmetric monoidal structure imply negative results about $\Cat$-enriched variants.

\subsection*{A discussion of choices: permutative categories and symmetric monoidal functors}

From the perspective of symmetric monoidal bicategory theory, restricting to permutative categories rather than studying all symmetric monoidal categories makes little difference: the inclusion of the 2-category of permutative categories into the 2-category of symmetric monoidal categories is a biequivalence, so long as the 1-cells in both consist of all symmetric monoidal functors.
See \cite[Section~5.1]{Gur2012Biequivalences} for transport of monoidal structure along biequivalences.

The question of which morphisms to focus on is more subtle.
Beginning in \cref{defn:smfunctor}, and throughout the paper, we use the term \emph{symmetric monoidal functor} for the strong case, with invertible monoidal constraints.
In \cref{defn:bilinear} we introduce a corresponding notion of bilinear maps, also of the strong variety, and prove a characterization in \cref{prop:curry}.
For both symmetric monoidal functors and bilinear maps, there are generalizations to a lax variant, where the monoidal and unit constraints are not assumed invertible.
In these cases, one has a generalization of \cref{prop:curry} corresponding to lax symmetric monoidal functors.
However, our further constructions and results depend in several ways on the restriction to the strong case.

Firstly, the comparison with functors out of a tensor product in \cref{prop:otimes-bilin-iso} depends on the definition of bilinear maps with invertible monoidal constraints, because these correspond to the (invertible) structure morphisms in the tensor product.
Invertibility of these structure morphisms is used in a number of places, notably:
\begin{itemize}
\item for the unit and counit of the adjoint equivalence $(\dA,\dAdot)$ in \cref{prop:AAdot} and
\item for the 2-unitor $\la$ in \cref{defn:lambda}.
\end{itemize} 

Secondly, our proof of the $\otimes-\Hom$ biadjunction in \cref{thm:biadj} depends, via \cref{prop:tensor-cofib-adj-equiv}, on invertibility of the monoidal constraints for symmetric monoidal functors.
See \cref{rmk:strong-needed} for further details of this point.

Thirdly, and independently of the above, the pseudonaturality constraints for the left unitor $\dL$ in \cref{prop:Lpsnat} have components that depend on the monoidal and unit constraints of symmetric monoidal functors $F$.  If these constraints are not invertible, then the corresponding lax naturality constraint for $\dL$ will also not be invertible.  A similar conclusion holds for $\dR$ in \cref{defn:R,lem:R=LB}. Thus we see that considering lax rather than strong symmetric monoidal functors forces a skew-type monoidal structure \cite{Szl2012Skew} rather than a genuine one.

A further, and we believe important, feature of our attention to detail is the recognition that the structure we construct is far easier to work with than a generic symmetric monoidal bicategory.
This ease comes as a result of many of the functors and transformations satisfying stronger-than-expected properties. 
As an example, the tensor product we construct always produces what would be called cofibrant permutative categories in the homotopy-theoretic literature (see \cref{prop:tensor-cofib-adj-equiv}) and can therefore be used to construct 2-monadic pseudomorphism classifiers \cite{BKP1989Two}.

\subsection*{A list of main results}
We list our main results below.
\begin{enumerate}
  \renewcommand{\labelenumi}{(\arabic{enumi})}
\item The tensor product is 2-functorial: \cref{prop:tensor-functor}.
\item The tensor product corepresents bilinear symmetric monoidal functors: \cref{prop:otimes-bilin-iso}.
\item The categories of strong, respectively strict, symmetric monoidal functors out of a tensor product are equivalent: \cref{prop:tensor-cofib-adj-equiv}.
\item The tensor product is left biadjoint to the internal hom: \cref{thm:biadj}.
\item Equipped with the tensor product and internal hom, the 2-category of permutative categories and strong symmetric monoidal functors is closed symmetric monoidal as a bicategory: \cref{thm:biadj,thm:permcat-smb}.
\item The free permutative category construction is symmetric monoidal, as a 2-functor between symmetric monoidal bicategories, with respect to the cartesian product of categories and the tensor product of permutative categories: \cref{thm:P-sm2fun}.
\item There is a direct sum for permutative categories that is both a bicategorical coproduct and bicategorical product: \cref{thm:oplus=bicat_directsum}.
\item The tensor product distributes, via a strict symmetric monoidal equivalence, over direct sums: \cref{cor:oplus-otimes-dist}.
\end{enumerate}

\section{Background}
\label{sec:defs}

This section recalls basic definitions and results for permutative categories and general 2-/bicategorical notions.

\subsection*{Permutative categories}
Here we fix terminology and notation for permutative categories and related notions.

\begin{defn}\label{defn:pc}
  A \textit{permutative category} $C$ consists of a strict monoidal
  category $(C, \oplus, e)$ together with a natural isomorphism,
  \[
    \begin{tikzpicture}[x=18mm,y=10mm]
      \draw[0cell] 
      (0,0) node (xy) {C \times C}
      (1,-1.25) node (yx) {C \times C}
      (2,0) node (xyagain) {C}
      ;
      \draw[1cell] 
      (xy) edge node {\oplus} (xyagain)
      (xy) edge[swap] node {\tau} (yx)
      (yx) edge[swap] node {\oplus} (xyagain)
      ;
      \draw[2cell]
      (1,-.5) node[rotate=0] {\be \Downarrow};
    \end{tikzpicture}
  \]
  where $\tau \cn C \times C \to C \times C$ is the symmetry
  isomorphism in $\Cat$, such that the following axioms hold for all
  objects $x,y,z$ of $C$.
  \begin{itemize}
  \item $\beta_{y,x} \beta_{x,y} = 1_{x \oplus y}$
  \item $\beta_{x, y \oplus z} = (1_y \oplus \beta_{x,z}) \circ (\beta_{x,y} \oplus 1_z)$
  \end{itemize}
  Note that these imply $\beta_{e,x} = 1_{x} = \beta_{x,e}$ for each object $x$ in $C$.
\end{defn}

\begin{notn}\label{notn:plus-zero}
  Here and throughout \cref{sec:bilin} we use $\oplus$ and $e$ to denote the monoidal sum and unit of a permutative category.
  Beginning in \cref{sec:otimes} and continuing through the rest of the document, we use the simpler notation $+$ and $0$, reserving $\oplus$ for the direct sum of permutative categories, defined in \cref{sec:oplus}.
\end{notn}

\begin{defn}\label{defn:smfunctor}
 Let $(C,e)$ and $(C',e')$ be permutative categories.
 A \emph{symmetric monoidal functor} between them consists of a functor $F\colon C \to C'$, a natural isomorphism $F_2$ called the \emph{monoidal constraint} with components
 \[
 \bigl(F_2\bigr)_{x,y}\colon Fx \oplus Fy \to F(x \oplus y), 
 \]
 and an isomorphism
 \[F_0\colon e' \to Fe,\]
 called the \emph{unit constraint}.
 These data satisfy the following associativity, unity, and symmetry axioms for all objects $x,y,z$ in $C$:
 \begin{align*}
   \bigl(F_2\bigr)_{x,y \oplus z} \, \circ \,  \bigl( 1_{Fx} \oplus \bigl(F_2\bigr)_{y,z}\bigr)
   & = \bigl(F_2\bigr)_{x \oplus y,z} \, \circ \,  \bigl( \bigl(F_2\bigr)_{x,y} \oplus 1_{Fz}\bigr),\\
   \bigl( F_2 \bigr)_{e,x} \,\circ\, \bigl( F_0 \oplus 1_{Fx} \bigr)
   & = 1_{Fx}, \\
   \bigl( F_2 \bigr)_{x,e} \,\circ\, \bigl( 1_{Fx} \oplus F_0 \bigr)
   & = 1_{Fx}, \andspace \\
   \bigl( F_2 \bigr)_{y,x} \,\circ\, \beta_{Fx,Fy}
   & = F\bigl(\beta_{x,y}\bigr) \,\circ\, \bigl( F_2 \bigr)_{x,y}.
 \end{align*}
 Note that in the presence of symmetries, either of the unit axioms implies the other.

 In the literature, what we have called a symmetric monoidal functor is sometimes also called a \emph{strong symmetric monoidal functor}, indicating that the monoidal and unit constraints are isomorphisms.
 If $F_2$ and $F_0$ are identities, we say $F$ is a \emph{strict symmetric monoidal functor}.
\end{defn}

\begin{convention}\label{convention:F2}
  For a symmetric monoidal functor $(F,F_2,F_0)$, the coherence theorem for monoidal functors \cite{JS1993Braided} shows that any two parallel composites of morphisms constructed using the monoidal product $\oplus$ of instances of $F_2$ and identity morphisms are equal.
  We will often write $F_2$ to denote any such composite.
\end{convention}

\begin{defn}\label{defn:montransf}
  Let $F,G \cn A \to B$ be symmetric monoidal functors between permutative categories.
  A \emph{monoidal transformation} $\phi \cn F \To G$ consists of a natural transformation $\phi$ such that the following monoidal naturality and unit equalities hold for each pair of objects $x,y \in A$:
  \begin{align*}
    \phi_{x \oplus y} \circ \bigl(F_2\bigr)_{x,y}
    & = \bigl(G_2\bigr)_{x,y} \circ \bigl( \phi_x \oplus \phi_y \bigr) \andspace \\
    \phi_{e} \circ F_0
    & = G_0.
  \end{align*}
  A monoidal transformation between strict symmetric monoidal functors is defined in the same way.
\end{defn}

\begin{notn}\label{notn:permcat}
  We write $\permcat$ for the 2-category of permutative categories, symmetric monoidal functors, and monoidal transformations.
  We write $\permcats$ for the sub 2-category whose 1-cells are strict symmetric monoidal functors, and
  \begin{equation}\label{eq:incl-j}
    j \cn \permcats \to \permcat
  \end{equation}
  for the inclusion 2-functor that is the identity on 0-, 1-, and 2-cells.
  In \cref{sec:bilin-otimes} we also write
  \[
    \strict{A,B} = \permcats(A,B)
  \]
  for the category of strict symmetric monoidal functors and monoidal transformations between $A$ and $B$.
\end{notn}

\begin{defn}\label{defn:oplus_fun_trans}
  For permuative categories $A$ and $B$, the \emph{pointwise permutative structure} on
  \[
    \permcat(A,B)
  \]
  is defined as follows.

  For symmetric monoidal functors 
  \[
    F,F'\cn A \to B
  \]
  we define a symmetric monoidal functor $F \oplus F' \cn A \to B$ with
  \begin{itemize}
  \item $(F \oplus F')(a) = Fa \oplus F'a$,
  \item $(F \oplus F')_2 =  (F_2\oplus F'_2)\circ (1 \oplus \beta \oplus 1)$, \andspace
  \item $(F \oplus F')_0 = F_0 \oplus F'_0$.
  \end{itemize}
  To verify that $F \oplus F'$ is a symmetric monoidal functor, one uses the symmetric monoidal functor axioms of $F$ and $F'$ together with the structure axioms for the symmetry $\beta$ in $B$.
  
  Given monoidal transformations $\phi \cn F \impl G$ and $\phi' \cn F'\impl G'$, the monoidal transformation $\phi \oplus \phi'\cn F\oplus F'\impl G \oplus G'$ is defined by
  \[(\phi \oplus \phi')_a=\phi_a\oplus \phi'_a.\]
  The monoidal naturality axioms for $\phi$ and $\phi'$ imply those for $\phi \oplus \phi'$.
\end{defn}
 
The following proposition is standard.
We note that the strictness of the monoidal structure on $\permcat(A,B)$ follows directly from the corresponding strictness on $B$.
\begin{prop}\label{prop:permcat_SM}
  For permutative categories $A$ and $B$, the category of symmetric monoidal functors $\permcat(A,B)$ has the structure of a permutative category under the pointwise sum of \cref{defn:oplus_fun_trans}.
  The monoidal unit is the constant functor at the monoidal unit of $B$.
  The symmetry isomorphism is given pointwise by that of $B$.
\end{prop}

\subsection*{2-categorical background}

We begin by recalling two forms of the notion of biadjunction for 2-categories.
These are special cases of more general notions of biadjunction in a general tricategory $\mathsf{T}$.

\begin{defn}\label{defn:biadjunction}
  Consider the tricategory $\iicatps$ consisting of 2-categories, pseudofunctors, pseudonatural transformations, and modifications.
  A \emph{biadjunction} in $\iicatps$ between 2-categories $K$ and $L$ is a tuple $(F,G,\eta,\epz,\Ga,\De)$ consisting of
  \begin{itemize}
  \item pseudofunctors $F \cn K \to L$ and $G \cn L \to K$,
  \item pseudonatural transformations $\eta\cn 1_{L} \To GF$ and $\epz \cn FG \To 1_{K}$, and
  \item invertible modifications $\Ga\cn (\epz * F)\circ(F*\eta) \To 1_F$ and $\De \cn (G * \epz)\circ(\eta * G) \To 1_G$.
  \end{itemize}
  These data satisfy the two \emph{swallowtail equalities}, which require that certain pastings depicted in 
  \cite[Figures~1 and ~2]{Gur2012Biequivalences} are equal to identities.
\end{defn}

The details of the swallowtail equalities in \cref{defn:biadjunction} will not be needed here.
Instead, we will use the following result to recognize the tensor-hom biadjunction in \cref{thm:biadj}.
\begin{prop}\label{prop:biadj-loceq}
  Suppose given 2-categories $K$ and $L$ together with 2-functors $F\cn K \to L$ and $G \cn L \to K$.
  Then $(F,G)$ are members of a biadjunction in the tricategory $\iicatps$ if and only if there is an equivalence of categories
  \begin{equation}\label{eq:biadj-loceq}
    a_{X,Y} \cn L(FX,Y) \fto{\hty} K(X,GY) \forspace X \in K, Y \in L,
  \end{equation}
  such that $a_{X,Y}$ is pseudonatural in $X$ and $Y$.
\end{prop}
\begin{rmk}
  \Cref{prop:biadj-loceq} is well-known to experts, and similar formulations appear in \cite[Section~I,7]{Gra74Formal}.
  Local equivalences such as \cref{eq:biadj-loceq} are taken as a definition in, e.g., \cite{street1980fb,Kel89Elementary}.
  The introduction to Chapter 9 of \cite{Fio06Pseudo} gives further overview of the literature.
  For the specific formulation of \cref{prop:biadj-loceq}, one direction follows from \cite[Lemma~3.13]{Pst22Dualizable} and the other follows from \cite[Theorems~9.5 and~9.16]{Fio06Pseudo} together with \cite[Theorem~3.14]{Pst22Dualizable}.
\end{rmk}

In \cref{sec:oplus} we define a direct sum for permutative categories and show that it is a bicategorical coproduct and product.
Here we recall these general terms.
\begin{defn}\label{defn:bicat-2cat-prod}
  Let $a, b$ be objects of a 2-category $K$.
  \begin{enumerate}
  \item A \emph{bicategorical product} of $a$ and $b$ consists of an object $a \times b$ and projection 1-cells
    \[
      a \times b \fto{\pi_a} a \andspace a \times b \fto{\pi_b} b,
    \]
    such that each of the functors
    \[
      K(x, a \times b) \to K(x, a) \times K(x,b),
    \]
    given by composition and whiskering with $\pi_a$ and $\pi_b$, is an equivalence of categories, pseudonaturally in $x$.

  \item A \emph{bicategorical coproduct} of $a$ and $b$ is a bicategorical coproduct in $K^{op}$, the 2-category in which the direction of the 1-cells (and 1-cells only) is reversed.

  \item A \emph{2-categorical product} of $a$ and $b$ consists of an object $a \times b$ and projection 1-cells
    \[
      a \times b \fto{\pi_a} a \andspace a \times b \fto{\pi_b} b,
    \]
    such that each of the functors
    \[
      K(x, a \times b) \to K(x, a) \times K(x,b),
    \]
    given by composition and whiskering with $\pi_a$ and $\pi_b$, is an isomorphism of categories, 2-naturally in $x$.

  \item A \emph{2-categorical coproduct} of $a$ and $b$ is a 2-categorical coproduct in $K^{op}$, the 2-category in which the direction of the 1-cells (and 1-cells only) is reversed.
  \end{enumerate}
\end{defn}

\begin{rmk}\label{rmk:2cat-to-bicat}
  We note that any 2-categorical product or coproduct is a bicategorical one, by direct comparison of definitions.
  This is not true for all types of limits or colimits, but is true when the weight is appropriately cofibrant, as is the case here.
  We also note that a 2-category $K$ might admit bicategorical products or coproducts without admitting 2-categorical ones.
\end{rmk}

The cartesian product of permutative categories, with symmetric monoidal structure induced componentwise, is an example of 2-categorical products \cite{BKP1989Two,LS2012Enhanced}.
\begin{prop}\label{prop:prods-exist}
  Let $A$ and $B$ be permutative categories. Then the product of the underlying categories, $A \times B$, admits the structure of a permutative category pointwise. It is a 2-categorical product in both $\permcats$ and $\permcat$.
\end{prop}

Bicategorical direct sums are defined as follows.
\begin{defn}\label{defn:bicat-direct-sum}
Let $K$ be a 2-category.
\begin{description}
\item[Initial and terminal objects] An object $i$ is \emph{initial} if $K(i, a)$ is the terminal category for all $a \in K$. An object $t$ is \emph{terminal} if $K(a, t)$ is the terminal category for all $a \in K$.
\item[Zero objects] A \emph{zero object} is an object $0$ which is both initial and terminal.
\item[The comparison map] Assume that $K$ admits bicategorical products $a \times b$,  bicategorical coproducts $a \bincoprod b$, and a zero object $0$. Then $a \simeq a \times 0$ and $b \simeq 0 \times b$, giving an equivalence $a \bincoprod b \fto{\hty} \bigl( a \times 0 \bigr) \bincoprod  \bigl( 0 \times b \bigr)$. The \emph{comparison map} $a \bincoprod b \to a \times b$ is the composite
\[
  a \bincoprod b \fto{\hty} \bigl( a \times 0 \bigr) \bincoprod  \bigl( 0 \times b \bigr) \to a \times b,
\]
where the second map is induced by
\[
  a \times 0 \fto{1_a \times !} a \times b
  \andspace
  0 \times b \fto{! \times 1_b} a \times b.
\]
\item[Bicategorical direct sums] We say that $K$ admits \emph{bicategorical direct sums} if the comparison map is an equivalence.
\end{description}
\end{defn}

We note that both $\permcat$ and $\permcats$ have a zero object, namely, the terminal category equipped with its unique permutative structure.

\section{Bilinearity and the internal homs}
\label{sec:bilin}

In this section we define bilinear maps
\[
  H \cn A,B \to C
\]
and transformations between them for permutative categories $A$, $B$, and $C$.
\Cref{prop:curry} gives an isomorphism of permutative categories between the one given by the collection of such bilinear maps and bilinear transformations and the one given by symmetric monoidal functors
\[
  F \cn A \to \permcat(B,C)
\]
and monoidal transformations. 

\begin{defn}\label{defn:bilinear}
  Let $(A, \oplus, e_A)$, $(B,\oplus,e_B)$ and $(C,\oplus,e_C)$ be permutative categories. A \emph{bilinear map} $H\cn A,B \to C$ consists of the following.
  \begin{itemize}
  \item It has an \emph{underlying functor} $H\cn A \times B \to C$.
  \item It has \emph{bilinearity constraint} isomorphisms
    \begin{align*}
      H_a(b,b')\cn & H(a,b) \oplus H(a,b') \fto{\iso} H(a,b\oplus b'),\\
      H_b(a,a')\cn & H(a,b) \oplus H(a',b) \fto{\iso} H(a\oplus a',b).
    \end{align*}
    for $a, a' \in A$ and $b,b' \in B$. Both isomorphisms are natural with respect to all three of their variables.
  \item It has \emph{unit constraint} isomorphisms
    \begin{align*}
      H_{a,0}\cn & e_C \fto{\iso} H(a,e_B) \andspace \\
      H_{0,b}\cn & e_C \fto{\iso} H(e_A,b)
    \end{align*}
    that are natural with respect to $a \in A$ and $b \in B$.
  \end{itemize}
  These data are subject to the following axioms.
  \begin{enumerate}
  \item\label{BIL1} For fixed $a \in A$ and $b \in B$, the functors
    $H(a,-)$ and $H(-,b)$ are symmetric monoidal, with monoidal, respectively unit, constraints  
    \begin{align*}
      H(a,-)_2  = H_a, \quad
      & H(a,-)_0 = H_{a,0}, \andspace\\
      H(-,b)_2 = H_b, \quad
      & H(-,b)_0 = H_{0,b}.
    \end{align*}

  \item\label{BILii} The following \emph{interchange diagram} commutes for each $a,a' \in A$ and $b,b' \in B$.
    \begin{equation}\label{BIL2}
      \begin{tikzpicture}[x=30mm,y=20mm,vcenter]
        \draw[0cell] 
        (0,0) node (p0) {H(a,b) \oplus H(a,b') \oplus H(a',b) \oplus H(a',b')}
        (-1,-1.5) node (p2) {H(a,b\oplus b') \oplus H(a',b\oplus b')}
        (1,-1) node (p1) {H(a,b) \oplus H(a',b) \oplus H(a,b') \oplus H(a',b')}
        (1,-2) node (p3) {H(a\oplus a',b) \oplus H(a\oplus a',b')}
        (0,-3) node (p4) {H(a\oplus a',b\oplus b')}
        ;
        \draw[1cell] 
        (p0) edge node {1 \oplus \beta \oplus 1} (p1) 
        (p0) edge[swap] node {H_a \oplus H_{a'}} (p2)
        (p2) edge[swap] node {H_{b\oplus b'}} (p4)
        (p1) edge node {H_b\oplus H_{b'}} (p3)
        (p3) edge node {H_{a\oplus a'}} (p4)
        ;
      \end{tikzpicture}
    \end{equation}
  \item\label{BILiii} The following diagram commutes.
    \begin{equation}\label{BIL3}
      \begin{tikzpicture}[x=35mm,y=20mm,vcenter]
        \draw[0cell] 
        (0,0) node (0) {e_C}
        (1,0) node (1) {H(e_A,e_B)}
        ;
        \draw[1cell] 
        (0) edge[bend left=0, transform canvas={shift={(0,3pt)}}] node {H_{e_A,0}} (1)
        (0) edge[swap,bend right=0, transform canvas={shift={(0,-3pt)}}] node {H_{0,e_B}} (1)
        ;
      \end{tikzpicture}
    \end{equation}
  \item\label{BILiv} The following diagrams commute for each $a,a' \in A$ and $b,b' \in B$.
    \begin{equation}\label{BIL6a}
      \begin{tikzpicture}[x=55mm,y=15mm,vcenter]
        \draw[0cell] 
        (0,0) node (0p0) {e_C\oplus e_C}
        (0,-1) node (HpH) {H(a,e_B) \oplus H(a',e_B)}
        (1,0) node (0) {e_C}
        (1,-1) node (H) {H(a\oplus a',e_B)}
        ;
        \draw[1cell] 
        (0p0) edge node {=} (0)
        (HpH) edge['] node {H_{e_B}(a,a')} (H)
        (0p0) edge['] node {H_{a,0} \oplus H_{a',0}} (HpH)
        (0) edge node {H_{a\oplus a',0}} (H)
        ;
      \end{tikzpicture}
    \end{equation}
    \begin{equation}\label{BIL6b}
      \begin{tikzpicture}[x=55mm,y=15mm,vcenter]
        \draw[0cell] 
        (0,0) node (0p0) {e_C\oplus e_C}
        (0,-1) node (HpH) {H(e_A,b) \oplus H(e_A,b')}
        (1,0) node (0) {e_C}
        (1,-1) node (H) {H(e_A,b\oplus b')}
        ;
        \draw[1cell] 
        (0p0) edge node {=} (0)
        (HpH) edge['] node {H_{e_A}(b,b')} (H)
        (0p0) edge['] node {H_{0,b} \oplus H_{0,b'}} (HpH)
        (0) edge node {H_{0,b\oplus b'}} (H)
        ;
      \end{tikzpicture}
    \end{equation}
  \end{enumerate}
  This finishes the definition of a bilinear map.
\end{defn}

\begin{defn}[Bilinear transformation]\label{bilin-transf}
  A \emph{bilinear transformation}  
  \[
  \phy\cn H \Rightarrow J \cn A,B \to C
  \]
  is a natural transformation 
    \[
  \phy\cn H \Rightarrow J \cn A \times B \to C
  \]
  such that
  \begin{align*}
      \phy_{(a,-)} & \cn H(a,-) \Rightarrow J(a,-) \cn B \to C \text{ and }\\
      \phy_{(-,b)} & \cn H(-,b) \Rightarrow J(-,b) \cn A \to C
  \end{align*}
  are monoidal transformations for each $a \in A$ and $b \in B$.
\end{defn}

\begin{prop}\label{prop:bilin-perm}
 Given permutative categories $A$, $B$, and $C$, the collection of bilinear maps $A,B \to C$ together with bilinear transformations between them forms a permutative category denoted $\bilin(A,B;C)$.
\end{prop}
\begin{proof}
  Identity morphisms and composition are inherited from those of $\Cat(A \times B, C)$, noting that the identity transformation is always bilinear and the composite of two bilinear transformations is again bilinear.
  The permutative structure is determined pointwise by that of $C$, as follows.
  Given bilinear maps $H,J \in \bilin(A,B;C)$ the pointwise sum is
  \[
    (H \oplus J)(a,b) = H(a,b) \oplus J(a,b).
  \]
  The bilinearity constraints of $H \oplus J$ are given by those of $H$ and $J$ combined with the symmetry of $C$, as in the following composite:
  \begin{equation}\label{eq:HplusJsymm}
    \begin{tikzpicture}[x=30mm,y=18mm,vcenter]
      \draw[0cell] 
      (0,0) node (a) {
        H(a,b) \oplus J(a,b) \oplus H(a,b') \oplus J(a,b')
      }
      (1,-1) node (b) {
        H(a,b) \oplus H(a,b') \oplus J(a,b) \oplus J(a,b')
      }
      (2,0) node (c) {
        H(a,b \oplus b') \oplus J(a,b \oplus b') 
      }
      ;
      \draw[1cell] 
      (a) edge['] node[pos=.4] {1 \oplus \beta \oplus 1} (b)
      (b) edge['] node[pos=.6] {H_a \oplus J_a} (c)
      ;
    \end{tikzpicture}
  \end{equation} 
  and similarly for the other variable.
  The unit constraints for $H \oplus J$ are given by adding those of $H$ and $J$.
  The bilinear map axioms for $H \oplus J$ follow from those of $H$ and $J$ separately, together with the permutative structure of $C$.

  The pointwise sum of bilinear transformations $\phy$ and $\mu$ is given by components
  \[
    (\phy \oplus \mu)_{(a,b)} = \phy_{(a,b)} \oplus \mu_{(a,b)} \forspace a \in A,\ b \in B,
  \]
  and the axioms of \cref{bilin-transf} hold componentwise by those of $\phy$ and $\mu$.
  Functoriality of the pointwise sum follows because units and composites of bilinear transformations are determined componentwise.
  
  Strict associativity of the pointwise sum follows from that of $C$. The monoidal unit in $\bilin(A,B;C)$ is the constant functor at $e_C$.
  The symmetry isomorphism is given pointwise by that of $C$.
\end{proof}

\begin{lem}\label{lem:bilin-compose-sm}
  Consider permutative categories
  \[
    A_1,\quad A_2,\quad \ol{A}_1,\quad \ol{A}_2,\quad C, \andspace \ol{C}.
  \]
  Precomposition and postcomposition define functors
  \[\cpre\cn\bilin\bigl(A_1,A_2;C\bigr) \times \tlprod{i=1}{2} \permcat\bigl(\ol{A}_i,A_i\bigr) \to \bilin\bigl(\ol{A}_1,\ol{A}_2;C\bigr)\]
  and
  \[\cpost\cn\permcat\bigl(C,\ol{C}\bigr) \times \bilin\bigl(A_1,A_2;C\bigr) \to \bilin\bigl(A_1,A_2;\ol{C}\bigr)\]
  such that the following axioms hold.
  Below, we abbreviate $\zP = \permcat$ and $\zB = \bilin$.
  \begin{description}
  \item[Compatibility] The following diagram commutes.
    \[
      \begin{tikzpicture}[x=60mm,y=18mm]
        \draw[0cell=.8] 
        (0,0) node (a) {
          \zP \bigl( {C},\ol{C} \bigr)
          \times
          \zB \bigl( A_1,A_2;C \bigr)
          \times
          \tlprod{i}{} \zP \bigl( \ol{A}_i,{A}_i \bigr)
        }
        (a)+(1,0) node (b) {
          \zP \bigl( {C},\ol{C} \bigr)
          \times
          \zB \bigl( \ol{A}_1,\ol{A}_2;C \bigr)
        }
        (b)+(0,-1) node (c) {
          \zB \bigl( \ol{A}_1,\ol{A}_2;\ol{C} \bigr)
        }
        (a)+(0,-1) node (a') {
          \zB \bigl( A_1,A_2;\ol{C} \bigr)
          \times
          \tlprod{i}{} \zP \bigl( \ol{A}_i,{A}_i \bigr)
        }
        ;
        \draw[1cell=.9] 
        (a) edge node {1 \times \cpre} (b)
        (b) edge node {\cpost} (c)
        (a) edge['] node {\cpost \times 1} (a')
        (a') edge['] node {\cpre} (c)
        ;
      \end{tikzpicture}
    \]
  \item[Associativity] The following two diagrams commute for permutative categories $\wt{A}_1$, $\wt{A}_2$, and $\wt{C}$.
    Here, $\circ$ denotes composition in $\zP = \permcat$.
    \[
      \begin{tikzpicture}[x=57mm,y=18mm]
        \draw[0cell=.8] 
        (0,0) node (a) {
          \zP \bigl( \ol{C},\wt{C} \bigr)
          \times
          \zP \bigl( {C},\ol{C} \bigr)
          \times
          \zB \bigl( A_1,A_2;C \bigr)
        }
        (a)+(1,0) node (b) {
          \zP \bigl( \ol{C},\wt{C} \bigr)
          \times
          \zB \bigl( A_1,A_2;\ol{C} \bigr)
        }
        (a)+(0,-1) node (c) {
          \zP \bigl( {C},\wt{C} \bigr)
          \times
          \zB \bigl( A_1,A_2;C \bigr)
        }
        (c)+(1,0) node (d) {
          \zB \bigl( A_1,A_2;\wt{C} \bigr)
        }
        ;
        \draw[1cell=.9] 
        (a) edge node {1 \times \cpost} (b)
        (b) edge node {\cpost} (d)
        (a) edge['] node {\circ \times 1} (c)
        (c) edge['] node {\cpost} (d)
        ;
      \end{tikzpicture}
    \]
    \[
      \begin{tikzpicture}[x=57mm,y=18mm]
        \draw[0cell=.8] 
        (0,0) node (a) {
          \zB \bigl( A_1,A_2;C \bigr)
          \times
          \tlprod{i}{} \zP \bigl( \ol{A}_i,A_i \bigr)
          \times
          \tlprod{i}{} \zP \bigl( \wt{A}_i,\ol{A}_i \bigr)
        }
        (a)+(1,0) node (b) {
          \zB \bigl( \ol{A}_1,\ol{A}_2;C \bigr)
          \times
          \tlprod{i}{} \zP \bigl( \wt{A}_i,\ol{A}_i \bigr)
        }
        (a)+(0,-1) node (c) {
          \zB \bigl( A_1,A_2;C \bigr)
          \times
          \tlprod{i}{} \zP \bigl( \wt{A}_i,A_i \bigr)
        }
        (c)+(1,0) node (d) {
          \zB \bigl( \wt{A}_1,\wt{A}_2;C \bigr)
        }
        ;
        \draw[1cell=.9] 
        (a) edge node {\cpre \times 1} (b)
        (b) edge node {\cpre} (d)
        (a) edge['] node {1 \times \tlprod{i}{} \circ} (c)
        (c) edge['] node {\cpre} (d)
        ;
      \end{tikzpicture}
    \]
  \item[Unity] The following two quadrilaterals commute, where $*$ denotes the terminal category and $\id$ denotes the functors that pick out identities.
    \[
      \begin{tikzpicture}[x=20mm,y=11mm]
        \draw[0cell=.8] 
        (0,0) node (a) {
          \zB \bigl( A_1,A_2;C \bigr)
        }
        (a)+(1,1) node (as) {
          \zB \bigl( A_1,A_2;C \bigr) \times *
        }
        (as)+(2,0) node (b) {
          \zB \bigl( A_1,A_2;C \bigr)
          \times 
          \tlprod{i}{} \zP \bigl( A_i,A_i \bigr)
        }
        (b)+(1,-1) node (c) {
          \zB \bigl( A_1,A_2;C \bigr)
        }
        (a)+(1,-1) node (sa) {
          * \times \zB \bigl( A_1,A_2;C \bigr)
        }
        (sa)+(2,0) node (b') {
          \zP \bigl( C,C \bigr)
          \times \zB \bigl( A_1,A_2;C \bigr)
        }
        ;
        \draw[1cell=.9] 
        (a) edge node {\iso} (as)
        (as) edge node {1 \times \id} (b)
        (b) edge node[pos=.7] {\cpre} (c)
        (a) edge['] node {\iso} (sa)
        (sa) edge['] node {\id \times 1} (b')
        (b') edge['] node[pos=.7] {\cpost} (c)
        (a) edge node {1} (c)
        ;
      \end{tikzpicture}
    \]
  \end{description}
\end{lem}

\begin{proof}
  To ease notation in this proof, we let
  \[
    A = A_1 \andspace B = A_2.
  \]
  At the level of underlying functors and underlying transformations, these operations are given by usual composition.
  More precisely, $\cpre$ sends
  \[H\colon A,B \to C, \quad P \colon \ol{A} \to A, \quad \text{and} \quad Q \colon \ol{B} \to B\]
  to the bilinear map with underlying functor given by $H(P-,Q-)$.
  The bilinearity constraints for $a,a'\in \ol{A}$ and $b,b'\in \ol{B}$ are given by the composites
  \[H(Pa,Qb)\oplus H(Pa,Qb') \fto{H_{Pa}} H(Pa,Qb\oplus Qb') \fto{H(1,Q_2)} H(Pa,Q(b\oplus b')),\]
  \[H(Pa,Qb)\oplus H(Pa',Qb) \fto{H_{Qb}} H(Pa\oplus Pa',Qb) \fto{H(P_2,1)} H(P(a\oplus a') ,Qb).\]
  Similarly, the unit constraints are given by
  \[e_C \fto{H_{Pa,0}} H(Pa,e_B) \fto{H(1,Q_0)} H(Pa,Qe_{\ol{B}}),\]
  \[e_C \fto{H_{0,Qb}} H(e_A,Qb) \fto{H(P_0,1)} H(Pe_{\ol{A}},Qb).\]
  The fact that these constraints satisfy the axioms of \cref{defn:bilinear} follow from the analogous axioms for $H$, the axioms for the symmetric monoidal functors, and naturality.
  Note that, when fixing one variable, this composition recovers the standard composition of symmetric monoidal functors.
  This observation can be used to prove that the composite of a bilinear transformation with monoidal transformations is bilinear, thus showing that $\cpre$ is well-defined on transformations and is functorial. 

  The second functor, $\cpost$, sends 
  \[R \colon C \to \ol{C} \quad \text{and} \quad H \colon A,B \to C\]
  to the bilinear map with underlying functor $RH$.
  For $a,a'\in A$ and $b,b'\in B$, the bilinearity constraints are given by
  \[RH(a,b)\oplus RH(a,b') \fto{R_2} R(H(a,b)\oplus H(a,b') \fto{R(H_a)} RH(a,b\oplus b'),\]
  \[RH(a,b)\oplus RH(a',b) \fto{R_2} R(H(a,b)\oplus H(a',b) \fto{R(H_b)} RH(a\oplus a',b),\]
  and the unit constraints are given by
  \[e_{\ol{C}} \fto{R_0} Re_C \fto{R(H_{a,0})} RH(a,e_B),\]
  \[e_{\ol{C}} \fto{R_0} Re_C \fto{R(H_{0,b})} RH(e_A,b).\]
  Again, the axioms for $H$ and $R$, including naturality of the constraints, are key to showing that this satisfies the axioms of \cref{defn:bilinear}.
  The same observation as above shows that $\cpost$ is well-defined on transformations and is functorial.

  The compatibility, associativity, and unity axioms follow from similar considerations about composition of symmetric monoidal functors, or can be verified directly.
  This completes the proof.
\end{proof}

\begin{rmk}\label{rmk:bilin-multicat}
  \Cref{lem:bilin-compose-sm} can be extended to define a multicomposition for $\bilin$, satisfying the axioms of a symmetric multicategory.
  We note that this is subtly different from the multicategory structure for permutative categories developed in \cite{EM2006Rings} and  \cite[Sections~6.5,6.6]{JYringIII}.
  There, the multilinear functors of permutative categories are required to be strictly unital.
  That condition is necessary for the applications to $K$-theory spectra in those and related works, but is too restrictive for our purposes below.
\end{rmk}

\begin{prop}\label{prop:curry}
 There are isomorphisms of permutative categories
 \begin{equation}\label{eq:bilin-curry}
   \permcat\bigl(A,\permcat(B,C)\bigr)\cong \bilin\bigl(A,B;C\bigr)
 \end{equation}
 that are
 \begin{itemize}
 \item 2-natural in $A$, $B$, and $C$, and
 \item strict symmetric monoidal with respect to the pointwise monoidal sums.  \end{itemize} 
\end{prop}
\begin{proof}
  The isomorphism of categories
  \begin{equation}\label{eq:cat-curry}
    \Cat\bigl(A,\Cat(B,C)\bigr) \iso \Cat\bigl(A \times B, C\bigr)
  \end{equation} 
  determines a correspondence of functors $F \leftrightarrow H$ where
  \[
    F(a)(b) = H(a,b).
  \]
  We first verify that this correspondence extends to an isomorphism
  \[ \permcat\bigl(A,\permcat(B,C)\bigr) \iso \bilin(A,B;C).\]
   The additional data for a symmetric monoidal functor $F \cn A \to \permcat(B, C)$ corresponds to the additional data of a bilinear map $H \cn A, B \to C$ as follows.
  \begin{itemize}
  \item The component of the coherence isomorphism $H_a$ at $b,b'$ corresponds to the isomorphism
    \[
      F(a)_2 \cn F(a)(b) \oplus F(a)(b') \cong F(a)(b \oplus b'),
    \]
    which is the monoidal constraint for the symmetric monoidal functor $F(a)$.
  \item The component of the coherence isomorphism $H_b$ at $a, a'$ corresponds to the component of $(F_2)_{a,a'}$ at $b$, where $F_2$ is the monoidal constraint of $F\cn A \to \permcat(B,C)$.
  \item The isomorphism $H_{a,0}$ corresponds to $F(a)_0$, which is the unit constraint $F(a)_0 \cn e_C \cong F(a)(e_B)$ for the monoidal functor $F(a)$. 
  \item The isomorphism $H_{0,b}$ corresponds to the component of $F_0$ at $b$, where $F_0$ is the unit constraint of $F$.
  \end{itemize}

  One then verifies that the axioms for symmetric monoidal structure on each $F(a)$, and on $F$ itself, correspond to the bilinearity axioms of \cref{defn:bilinear}.
  For example, the monoidal constraint has components
  \begin{equation}\label{eq:bilinF2}
    F_2\cn F(a) \oplus F(a') \cong F(a \oplus a') \forspace a,a' \in A
  \end{equation}
  that are morphisms in $\permcat(B,C)$, and thus are monoidal transformations.
  The unit axiom of \cref{eq:bilinF2} corresponds to commutativity of \cref{BIL6a}.
  The monoidal naturality axiom for \cref{eq:bilinF2}, at a pair of objects $b,b' \in B$, is commutativity of the square below.
  \[
    \begin{tikzpicture}[x=75mm,y=20mm,vcenter]
      \draw[0cell] 
      (0,0) node (a) {\Bigl( F(a) \oplus F(a') \Bigr) (b) \oplus \Bigl( F(a) \oplus F(a') \Bigr) (b')}
      (1,0) node (b) {\Bigl(F(a) \oplus F(a')\Bigr)(b \oplus b')}
      (0,-1) node (c) {F(a \oplus a')(b) \oplus F(a \oplus a')(b')}
      (1,-1) node (d) {F(a \oplus a')(b \oplus b')}
      ;
      \draw[1cell]
      (a) edge node {\bigl(F(a) \oplus F(a')\bigr)_2} (b)
      (b) edge node {F_2} (d)
      (a) edge[swap] node {F_2 \oplus F_2} (c)
      (c) edge[swap] node {F(a \oplus a')_2} (d)
      ;
    \end{tikzpicture}
  \]
  Under the correspondence $F \leftrightarrow H$, commutativity of the above diagram corresponds to that of \cref{BIL2}.
  The top right composite above is the right hand composite of \cref{BIL2}, because the monoidal constraint for $F(a) \oplus F(a')$ involves the symmetry isomorphism in $C$.

  Similarly, the unit constraint for $F$
  \begin{equation}\label{eq:bilinF0}
    F_0: e_{\permcat(B,C)} \cong F(e_A)
  \end{equation} 
  is also a monoidal transformation, where the left hand side denotes the constant functor at $e_C$.
  The monoidal naturality axiom for \cref{eq:bilinF0} corresponds to commutativity of \cref{BIL6b} for each $b,b' \in B$.
  The unit axiom for \cref{eq:bilinF0} is the commutative triangle below.
  \[
    \begin{tikzpicture}[x=50mm,y=20mm,vcenter]
      \draw[0cell]
      (0,0) node (a) {e_C}
      (1,.5) node (b) {e_{\permcat(B,C)}(e_B)}
      (1,-.5) node (c) {F(e_A)(e_B)}
      ;
      \draw[1cell]
      (a) edge node[pos=.8] {(e_{\permcat(B,C)})_0} (b)
      (b) edge node {F_0} (c)
      (a) edge['] node[pos=.6] {F(e_A)_0} (c)
      ;
    \end{tikzpicture}
  \]
  The 1-cell $(e_{\permcat(B,C)})_0$ is the identity at $e_C$ and therefore, under the correspondence $F \leftrightarrow H$, commutativity of the above diagram corresponds to the equality in \cref{BIL3}.

  The remaining axioms for $(F, F_2, F_0)$ correspond to various symmetric monoidal functor axioms \cref{BIL1} for $H$.  This verifies a bijection on objects for \cref{eq:bilin-curry}.

  The isomorphism of categories \cref{eq:cat-curry} also determines a correspondence of natural transformations $\al \leftrightarrow \phy$ with
  \[
    \al_{a,b} = \phy_{(a,b)},
  \]
  where we write $\al_{a,b}$ for the component of $\al_a$ at $b$. 
  For each $a$, the monoidal naturality axioms of $\al_a$ correspond to those of $\phy_{(a,-)}$.
  The monoidal naturality axioms for $\al$ itself correspond, componentwise at each $b$, to those axioms for $\phy_{(-,b)}$.
  This verifies a bijection on morphisms for \cref{eq:bilin-curry}. 

  Functoriality of the bijection
  \begin{equation}\label{eq:F-H-al-phi}
    F \leftrightarrow H \andspace \al \leftrightarrow \phy
  \end{equation}
  is immediate because the respective identity transformations correspond and composition is determined componentwise on both sides.  Therefore, \cref{eq:bilin-curry} is an isomorphism of categories.  The description of bilinearity and unit constraints for $H \oplus J$ in the proof of \cref{prop:bilin-perm} shows that the permutative structure on each side is strictly preserved by the bijection \cref{eq:F-H-al-phi}.

  To verify 2-naturality of \cref{eq:bilin-curry}, recall \cref{lem:bilin-compose-sm} and consider symmetric monoidal functors
  \[
    P \cn \ol{A} \to {A}, \quad
    Q \cn \ol{B} \to {B}, \andspace
    R \cn {C} \to \ol{C}.
  \]
  For each symmetric monoidal functor $F \in \permcat\bigl(A, \permcat(B,C)\bigr)$ and the corresponding bilinear map $H \in \bilin(A,B;C)$, let
  \begin{align*}
    \ol{F} & = R_*Q^*P^* F = R \bigl( F(P-)(Q-)  \bigr)
             \inspace \permcat\Bigl(\ol{A}, \permcat\bigl(\ol{B},\ol{C}\bigr)\Bigr)
             \andspace \\
    \ol{H} & = R_*(P \times Q)^* = R \bigl( H(\,P- , Q-\,)  \bigr)
             \inspace \bilin\Bigl(\ol{A},\ol{B}; \ol{C}\Bigr).
  \end{align*}
  The underlying functors correspond to each other, and one can verify that the monoidal, respectively bilinearity, and unit constraints also agree, so $\ol{F} \leftrightarrow \ol{H}$.
  Likewise, for monoidal transformations
  \[
    \phi \cn P \To P', \quad \psi \cn Q \To Q', \andspace \ga \cn R \To R',
  \]
  the correspondence $\al \leftrightarrow \phy$ for morphisms of \cref{eq:bilin-curry} is preserved by horizontal composition with $\phi \times \psi$ and $\ga$.
\end{proof}

\section{The tensor product of permutative categories}
\label{sec:otimes}

In this section we define a tensor product 2-functor
\[
  \otimes \cn \permcat \times \permcat \to \permcat.
\]
As it will be seen in \cref{sec:bilin-otimes}, the tensor product has the universal property that \emph{strict} symmetric monoidal functors with source $A\otimes B$ are in one-to-one correspondence with bilinear maps out of $A,B$.
This motivates the definition of $\otimes$.

The definitions of $\otimes$ on objects, 1-, and 2-cells are 
\cref{defn:AtensorB,,defn:FtensorG,,defn:phi-tensor-psi}, respectively.  The 2-functoriality of $\otimes$ is proved in \cref{prop:tensor-functor}. \begin{notn}
  Beginning in this section, and continuing throughout the document, we use the generic notation $(+,0)$ for the monoidal sum and monoidal unit in a permutative category.
  We will use $\oplus$ for the direct sum of permutative categories, defined in \cref{sec:oplus}.
\end{notn}

\begin{defn}\label{notn:freeT}
  Given a category $C$, let
  \[
    \bigT C = \coprod_{n\geq 0} C^n
  \]
  denote the free strict monoidal category on $C$.
\end{defn}

\begin{defn}\label{defn:free-adjoin}
  Let $A$ be a strict monoidal category.
  We define two constructions using colimits in the category of strict monoidal categories and strict monoidal functors.
  Let $\{0,1\}$ denote the discrete category with two objects.  
  Let $\II$ denote the category with two objects and a single isomorphism $f$ between them.
  \begin{description}
    \item[Adjoin an Isomorphism]
      Let $x,y$ be a pair of objects in  $A$.
      Let $A+\II$ be the pushout below.
      \[
        \begin{tikzpicture}[x=30mm,y=20mm]
          \draw[0cell] 
          (0,0) node (a) { \bigT \{0,1\} }
          (0,-1) node (b) { \bigT  \II }
          (1,0) node (c) {A}
          (1,-1) node (d) {A+\II}
          ;
          \draw[1cell] 
          (a) edge node {} (b)
          (a) edge node {} (c)
          (b) edge node {} (d)
          (c) edge node {} (d)
          ;
          \draw[2cell] 
          (d)+(135:.25) node[scale=1.5] {\ulcorner}
          ;
        \end{tikzpicture}
      \]
      Here, the top arrow sends $0,1$ to $x,y$, respectively, and the left vertical arrow sends $0,1$ to the source and target of $f$, respectively.
      We say that $A+\II$ is obtained by \emph{adjoining the isomorphism $f\cn x \to y$}.
    \item[Impose a Relation]
      Let $p$ and $q$ be parallel morphisms in $A$.
      Let $\JJ$ denote the category with two objects and a single non-identity morphism.
      Let $A/(p=q)$ be the coequalizer below.
      \[
        \begin{tikzpicture}[x=20mm,y=20mm]
          \draw[0cell] 
          (0,0) node (a) { \bigT \JJ }
          (1,0) node (c) {A}
          (2.125,0) node (d) {A/(p=q)}
          ;
          \draw[1cell] 
          (a) edge[transform canvas={shift={(0,.7ex)}}] node {} (c)
          (a) edge[transform canvas={shift={(0,-.7ex)}}] node {} (c)
          (c) edge node {} (d)
          ;
        \end{tikzpicture}
      \]
      Here, the two arrows at left are the free symmetric monoidal functors that send the non-identity morphism of $\JJ$ to $p$ and $q$, respectively.
      We say that $A/(p=q)$ is obtained by \emph{imposing the relation $p = q$}.
    \end{description}
\end{defn}

\begin{defn}\label{defn:AtensorB}
Let $A = (A,+,0)$ and $B = (B,+,0)$ be permutative categories.  The
permutative category $A \otimes B$ is defined by the following process.  

\begin{description}
\item[Step 1] Consider the category $\bigT \bigl(A \times B\bigr)$.
  We write objects or morphisms of $A\times B$ as $a.b$ or $f.g$ for pairs of objects or morphisms, with $f \cn a \to a'$ in $A$ and $g \cn b \to b'$ in $B$.
  We write an element of $(A\times B)^n$ as a formal sum of $n$ terms, with the unique $0$-tuple being denoted by $0$.
  Objects and morphisms of the form $a.b$ or $f.g$  are called \emph{simple} objects and morphisms.

\item[Step 2] Adjoin the following isomorphisms for all $a,a',b,b'$:
\[
a.b + a.b' \fto{\ \de^L_{a}(b, b')\ } a.(b + b')
\quad \mathrm{and} \quad
a.b + a'.b \fto{\ \de^R_{b}(a, a')\ } (a+a').b
\]
\[
a.b + a'.b' \fto{\ \beta_{a.b, \, a'.b'}\ } a'.b' + a.b
\]
\[
0 \fto{\ \ze^L_{a}\ } a.0
\quad \mathrm{and} \quad
0 \fto{\ \ze^R_{b}\ } 0.b.
\]
These are called the \emph{adjoined morphisms} in $A \otimes B$.
We will sometimes suppress the superscripts $L$ and $R$.

\item[Step 3] Impose the following relations.

\begin{enumerate}
\item\label{TENnat} The isomorphisms $\delta, \beta, \zeta$, in all of their variants, are natural in every variable.
\item\label{TENsmc} The isomorphisms $\be$, together with $+$ and $0$, give the structure of a permutative category.
\item\label{TEN1} For fixed $a \in A$, the \emph{left multiplication} functor $a.(-)$ is symmetric monoidal with coherence isomorphisms given by $\de^L_a$ and $\ze^L_{a}$. Similarly, for fixed $b\in B$ the \emph{right multiplication} functor $(-).b$ is symmetric monoidal with coherence isomorphisms given by $\de^R_b$ and $\ze^R_{b}$.

\item\label{TEN2}
  The following composites are equal for each $a,a'\in A$ and $b,b' \in B$.
\[
  \begin{tikzpicture}[x=53mm,y=15mm]\
    \draw[0cell] 
    (0,1) node (01) {a.b + a'.b + a.b' + a'.b'}
    (1,1) node (11) {a.b + a.b' + a'.b + a'.b'}
    (0,0) node (00) {(a + a').b + (a + a').b'}
    (1,0) node (10) {a.(b+b') + a'.(b+b')}
    (.5,-.8) node (e) {(a + a').(b + b')}
    ;
    \draw[1cell] 
    (01) edge node {1 + \be + 1} (11)
    (01) edge[swap] node {\de^R_b + \de^R_{b'}} (00)
    (11) edge node {\de^L_a + \de^L_{a'}} (10)
    (00) edge[swap] node {\de^L_{a+a'}} (e)
    (10) edge node {\de^R_{b+b'}} (e)
    ;
  \end{tikzpicture}
\]

\item\label{TEN3}
  The following morphisms are equal.
\[
\begin{tikzpicture}[x=36mm,y=15mm]
  \draw[0cell] 
  (0,0) node (00) {0}
  (1,0) node (01) {0.0}
  ;
  \draw[1cell] 
  (00) edge[bend left=0, transform canvas={shift={(0,3pt)}}] node{\ze^L_{0}} (01)
  (00) edge[swap, bend right=0, transform canvas={shift={(0,-3pt)}}] node{\ze^R_{0}} (01)  
  ;
\end{tikzpicture}
\]

\item\label{TEN6}
  The following two composites are equal for each $a,a'\in A$ and $b,b' \in B$.
\[
\begin{tikzpicture}[x=36mm,y=15mm]
  \draw[0cell] 
  (0,0) node (00) {a.0 + a'.0}
  (1,0) node (10) {(a+a').0}
  (0,1) node (01) {0 + 0}
  (1,1) node (11) {0}
  ;
  \draw[1cell] 
  (00) edge[swap] node{\de^R_0} (10)
  (01) edge node{=} (11)
  (01) edge[swap] node{\zeta^L + \zeta^L} (00)
  (11) edge node{\zeta^L} (10)  
  ;
\end{tikzpicture}
\hspace{2pc}
\begin{tikzpicture}[x=36mm,y=15mm]
  \draw[0cell] 
  (0,0) node (00) {0.b + 0.b'}
  (1,0) node (10) {0.(b+b')}
  (0,1) node (01) {0 + 0}
  (1,1) node (11) {0}
  ;
  \draw[1cell] 
  (00) edge[swap] node{\de^L_0} (10)
  (01) edge node{=} (11)
  (01) edge[swap] node{\zeta^R + \zeta^R} (00)
  (11) edge node{\zeta^R} (10)
  ;
\end{tikzpicture}
\]

\end{enumerate}
\end{description}

\end{defn}

\begin{lem}\label{lem:tensor-bilinear}
There is a bilinear map
\[
  \uni\cn A,B \to A \otimes B
\]
whose underlying functor is the composite 
 \[A\times B \hookrightarrow \bigT(A\times B) \to A\otimes B,\]
and whose bilinearity constraints are given by $\delta_a^L$ and $\delta_b^R$, and whose unit constraints are given by $\zeta_a^L$ and $\zeta_b^R$. 
\end{lem}
\begin{proof}
Indeed, note that the axioms for a bilinear map in \cref{defn:bilinear} mirror the relations imposed in the constructions of $A\otimes B$ above.
\end{proof}

\begin{rmk}\label{rmk:T-vs-P}
  In \cref{defn:AtensorB}, one can equivalently use the free \emph{permutative} category construction, $\bigP$ (see \cref{defn:freeP}), instead of $\bigT$, and adjoin only the morphisms $\de^L$, $\de^R$, $\ze^L$, and $\ze^R$.
\end{rmk}

\begin{defn}\label{defn:FtensorG}
  Given symmetric monoidal functors
  \[
    F\cn A \to A' \andspace
    G\cn B \to B',
  \]
  we define a strict symmetric monoidal functor
  \[
    F \otimes G \cn A \otimes B \to A' \otimes B'
  \]
  as follows.
  \begin{description}
  \item[Step 1] Consider the strict monoidal functor
    \[
      \bigT \bigl(F \times G\bigr) \cn \bigT \bigl(A \times B\bigr) \to \bigT \bigl(A' \times B'\bigr).
    \]
    For the monoidal unit 0 or for a formal sum of pairs, this functor is given explicitly by
    \begin{align*}
      \bigT \bigl( F \times G \bigr) (0) & = 0, \\
      \bigT \bigl( F \times G \bigr) (\tsum a_i . b_i)
                                    & = \tsum Fa_i . Gb_i, \andspace\\
      \bigT \bigl( F \times G \bigr) (\tsum f_i . g_i)
                                    & = \tsum Ff_i . Gg_i.
    \end{align*}
    By the definition there is the canonical strict monoidal functor
    \[
      \bigT (A' \times B') \to A' \otimes B'.
    \]
    We will define $F \otimes G$ by descending through the pushouts and coequalizers defining $A \otimes B$.
  \item[Step 2] For each morphism $\ka$ adjoined in Step 2 of \cref{defn:AtensorB}, we define $\bigl( F \otimes G \bigr)(\ka)$ as the indicated composite in the corresponding diagram below.
    \[
      \begin{tikzpicture}[x=50mm,y=15mm]
        \draw[0cell] 
        (0,0) node (a) {Fa.Gb + Fa.Gb'}
        (a)+(.5,-1) node (b) {Fa.(Gb + Gb')}
        (b)+(.5,1) node (c) {Fa.G(b+b')}
        ;
        \draw[1cell] 
        (a) edge['] node {\de^L_{Fa}} (b)
        (b) edge['] node {1.G_2} (c)
        (a) edge[dashed] node {\bigl(F \otimes G\bigr)\bigl(\de^L_a\bigr)} (c)
        ;
      \end{tikzpicture}
    \]
    \[
      \begin{tikzpicture}[x=50mm,y=15mm]
        \draw[0cell] 
        (0,0) node (a) {Fa.Gb + Fa'.Gb}
        (a)+(.5,-1) node (b) {(Fa + Fa').Gb}
        (b)+(.5,1) node (c) {F(a + a').Gb}
        ;
        \draw[1cell] 
        (a) edge['] node {\de^R_{Gb}} (b)
        (b) edge['] node {F_2.1} (c)
        (a) edge[dashed] node {\bigl(F \otimes G\bigr)\bigl(\de^R_b\bigr)} (c)
        ;
      \end{tikzpicture}
    \]
    \[
      \begin{tikzpicture}[x=50mm,y=18mm]
        \draw[0cell] 
        (0,0) node (a) {Fa.Gb + Fa'.Gb'}
        (a)+(0,-.5) node (b) {Fa.Gb + Fa'.Gb'}
        (1,0) node (c) {Fa'.Gb' + Fa.Gb}
        (c)+(0,-.5) node (d) {Fa'.Gb' + Fa.Gb}
        ;
        \draw[1cell] 
        (a) edge[dashed] node {\bigl(F \otimes G\bigr)\bigl(\be\bigr)} (c)
        (b) edge['] node {\be} (d)
        (a) edge[equal] node {} (b)
        (c) edge[equal] node {} (d)
        ;
      \end{tikzpicture}
    \]
    \[
      \begin{tikzpicture}[x=50mm,y=15mm]
        \draw[0cell] 
        (0,0) node (a) {0}
        (a)+(.5,-1) node (b) {Fa.0}
        (b)+(.5,1) node (c) {Fa.G0}
        ;
        \draw[1cell] 
        (a) edge['] node {\ze^L_{Fa}} (b)
        (b) edge['] node {1.G_0} (c)
        (a) edge[dashed] node {\bigl(F \otimes G\bigr)\bigl(\ze^L_{a}\bigr)} (c)
        ;
      \end{tikzpicture}
    \]
    \[
      \begin{tikzpicture}[x=50mm,y=15mm]
        \draw[0cell] 
        (0,0) node (a) {0}
        (a)+(.5,-1) node (b) {0.Gb}
        (b)+(.5,1) node (c) {F0.Gb}
        ;
        \draw[1cell] 
        (a) edge['] node {\ze^R_{Gb}} (b)
        (b) edge['] node {F_0.1} (c)
        (a) edge[dashed] node {\bigl(F \otimes G\bigr)\bigl(\ze^R_{b}\bigr)} (c)
        ;
      \end{tikzpicture}
    \]
  \item[Step 3] One verifies that $F \otimes G$ is well-defined on $A \otimes B$ by checking that the relations imposed in Step 3 of \cref{defn:AtensorB} are preserved by the assignments above.
    For each relation, this involves the corresponding relation in $A' \otimes B'$, the monoidal and unit axioms of $F$ and $G$, functoriality of the cartesian product, and naturality of the data $\de$, $\be$, and $\ze$.
    For example, to verify that relation \cref{TEN3} is preserved, one uses naturality of $\ze^L$ and $\ze^R$ in $A' \otimes B'$ together with functoriality of the cartesian product to show that the following diagram commutes.
    \[
      \begin{tikzpicture}[x=70mm,y=50mm]
        \draw[0cell] 
        (0,0) node (0) {0}
        (0)+(.5,0) node (00) {0.0}
        (00)+(.5,0) node (f0g0) {F0.G0}
        (00)+(.25,.25) node (f00) {F0.0}
        (00)+(.25,-.25) node (0g0) {0.G0}
        ;
        \draw[1cell] 
        (0) edge[bend left=0, transform canvas={shift={(0,3pt)}}] node[pos=.6] {\ze^L_{0}} (00)
        (0) edge[',bend right=0, transform canvas={shift={(0,-3pt)}}] node[pos=.6] {\ze^R_{0}} (00)
        (00) edge node[pos=.4] {F_0.1} (f00)
        (00) edge['] node[pos=.4] {1.G_0} (0g0)
        (f00) edge node[pos=.6] {1.G_0} (f0g0)
        (0g0) edge['] node[pos=.6] {F_0.1} (f0g0)
        ;
        \draw[1cell,rounded corners=3mm]
        (0.55) -- node {\ze^L_{F0}} ($(f00)+(-.5,0)$) -- (f00)
        ;
        \draw[1cell,rounded corners=3mm]
        (0.-55) -- node['] {\ze^R_{G0}} ($(0g0)+(-.5,0)$) -- (0g0)
        ;
      \end{tikzpicture}
    \]





  \end{description}
  Constructed as the universal morphism out of the colimits in \cref{defn:AtensorB}, $F \otimes G$ is a strict monoidal functor, even if $F$ and $G$ are not.  It is symmetric monoidal because
  \[
    \bigl( F \otimes G \bigr) (\beta) = \beta.
  \]
\end{defn}

\begin{defn}\label{defn:phi-tensor-psi}
  Given monoidal transformations
  \[
    \phi \cn F \To F' \andspace \psi \cn G \To G',
  \]
  we define a monoidal transformation
  \[
    \phi \otimes \psi \cn F \otimes G \To F' \otimes G'
  \]
  with components
  \begin{align*}
    \bigl( \phi \otimes \psi \bigr)_{\ssum a_i.b_i}
    & = \tsum \phi_{a_i} . \psi_{b_i} \andspace \\
    \bigl( \phi \otimes \psi \bigr)_0
    & = 1_0.
  \end{align*}
  Naturality with respect to morphisms $f.g$ in $A \otimes B$ follows from naturality of $\phi$ and $\psi$.  Naturality with respect to the additional morphisms $\de$, $\ze$, $\beta$ follows from the naturality of those morphisms and the monoidal transformation axioms for $\phi$ and $\psi$.
\end{defn}

\begin{prop}\label{prop:tensor-functor}
  The tensor product of \cref{defn:AtensorB,defn:FtensorG,defn:phi-tensor-psi} defines a 2-functor
  \[
    \otimes \cn \permcat \times \permcat \to \permcat.
  \]
\end{prop}
\begin{proof}
  Given two pairs of composable functors $F,F'$ and $G,G'$, one computes from the definitions
  \begin{align*}
    \Bigl(\bigl( F'\otimes G'\bigr)
    \circ
    \bigl( F \otimes G \bigr)\Bigr)\bigl(\de_a(b,b')\bigr)
    & = \Bigl( F' \otimes G' \Bigr) \bigl( 1.G_2 \circ \de_{Fa}(Gb,Gb') \bigr) \\
    & = 1.G'(G_2) \circ 1.G_2' \circ \de_{F'Fa}(G'Gb, G'Gb') \\
    & = 1.(G'G)_2 \circ \de_{F'Fa}(G'Gb, G'Gb') \\
    & = \Bigl( F'F \otimes G'G  \Bigr)\bigl( \de_{a}(b,b')  \bigr).
  \end{align*}
  This shows that $\bigl( F'\otimes G'\bigr) \circ \bigl( F \otimes G \bigr)$ and $\bigl( F'F \bigr) \otimes \bigl( G'G  \bigr)$ give the same assignment on $\de_a$.
  The corresponding calculations for the other adjoined morphisms are similar.
  Thus, $\otimes$ preserves composites of symmetric monoidal functors. One can similarly show that $\otimes$ preserves identity symmetric monoidal functors.
  One verifies componentwise that $\otimes$ preserves identities, vertical, and horizontal composites of monoidal transformations.
\end{proof}

\section{Bilinearity and the tensor product}
\label{sec:bilin-otimes}

This section shows that \emph{strict} symmetric monoidal functors out of a tensor product $A \otimes B$ correspond to bilinear maps with source $A, B$ in the sense of \cref{defn:bilinear}.
The main result is \cref{prop:otimes-bilin-iso}.

\begin{convention}\label{conv:bilin-generic}
  Throughout this section, we use the following notation for generic objects and morphisms.
  \begin{center}
    \renewcommand{\arraystretch}{1.25}
    \begin{tabular}{ccc}
      category & objects & morphisms\\
      \hline
      $A$ & $a,a',a_i$ & $f, f_i$
      \\
      $B$ & $b,b',b_i$ & $g, g_i$
      \\
      $C$ & $c,c',c_i$ & $h, h_i$
    \end{tabular}
    \\ \ 
  \end{center}
\end{convention}

\begin{defn}\label{defn:ABstrict}
  For permutative categories $A$ and $B$, we let
  \[
    \strict{A,B} = \permcats(A,B)
  \]
  denote the category of strict symmetric monoidal functors $A \to B$ and monoidal transformations.
  Note that this is generally \emph{not} a monoidal category with pointwise sum.
  Recalling \cref{defn:oplus_fun_trans}, for strict monoidal $F$ and $G$, the pointwise sum $F + G$ is generally not a strict monoidal functor, unless the target $B$ has trivial symmetry.
\end{defn}

\begin{prop}\label{prop:otimes-bilin-iso}
  Composition with the bilinear map $\uni$ of \cref{lem:tensor-bilinear} defines an isomorphism of categories
  \begin{equation}\label{eq:otimes-bilin}
    \strict{A \otimes B, C} \cong \bilin(A,B; C)
  \end{equation} 
  that is 2-natural with respect to strict symmetric monoidal functors $C \to \ol{C}$ and not-necessarily-strict symmetric monoidal functors $\ol{A} \to A$ and $\ol{B} \to B$.
\end{prop}
\begin{proof}
  By \cref{lem:bilin-compose-sm}, composition with $\uni\cn A,B \to A\otimes B$ defines a functor
  \[
    \strict{A \otimes B, C} \to \bilin \bigl( A,B; C \bigr)
  \]
  that sends a strict symmetric monoidal functor $F \cn A\otimes B \to C$ to the bilinear map
  \[A,B \fto{\uni} A\otimes B \fto{F} C.\]
  
  Given a bilinear map $H \cn A,B \to C$, we can construct a strict symmetric monoidal functor $F\cn A\otimes B \to C$ as follows. Applying the universal property of $\bigT$ to the underlying functor $H \cn A\times B \to C$ gives rise to a strict monoidal functor
  \[F\cn \bigT(A\times B) \to C.\]
  We further extend this to a strict symmetric monoidal functor with domain $A\otimes B$ by the following assignments on adjoined morphisms:
  \begin{align*}
    F(\de^L_a) & = H_a,
    \quad F(\de^R_b) = H_b,
    \quad F(\be_{a.b,a'.b'}) = \be_{H(a,b),H(a',b')},
    \\
    F(\ze^L_a) & = H_{a,0}, 
    \quad F(\ze^R_b) = H_{0,b}.
  \end{align*}

  To show that $F$ is well-defined, we verify that $F$ preserving the relations imposed in \cref{defn:AtensorB} for $A \otimes B$ corresponds to the axioms that $H$ must satisfy from \cref{defn:bilinear}. 
  Because $F$ is strictly symmetric monoidal, the requirement that $F$ preserve conditions~\ref{defn:AtensorB}.\cref{TENnat} and~\ref{defn:AtensorB}.\cref{TEN1} is equivalent to the requirement that each of
  \[
    \bigl( F(a.-)\,,\, F(\de^L_a)\,,\, F(\ze^L_a) \bigr)
    \andspace
    \bigl( F(-.b)\,,\, F(\de^R_b)\,,\, F(\ze^L_b) \bigr)
  \]
  is a symmetric monoidal functor.
  These correspond to axiom~\ref{defn:bilinear}.\cref{BIL1} for $H$, which requires that
  \[
    \bigl( H(a,-)\,,\, H_a\,,\, H_{a,0} \bigr) \andspace
    \bigl( H(-,b)\,,\, H_b\,,\, H_{0,b} \bigr)
  \]
  are symmetric monoidal functors.
  Preservation of~\ref{defn:AtensorB}.\cref{TENsmc} automatically holds because $F$ is required to be symmetric monoidal.  
  Preservation of the pentagon~\ref{defn:AtensorB}.\cref{TEN2} corresponds to the interchange axiom~\ref{defn:bilinear}.\cref{BILii} for $H$.
  Similarly, the requirement that $F$ preserve relations in~\ref{defn:AtensorB}.\cref{TEN3}, respectively~\ref{defn:AtensorB}.\cref{TEN6}, corresponds to axiom~\ref{defn:bilinear}.\cref{BILiii}, respectively~\ref{defn:bilinear}.\cref{BILiv} for $H$.
  This completes the proof that \cref{eq:otimes-bilin} is a bijection on objects.
  
  Given a bilinear transformation
  \begin{equation}\label{eq:phiab-alab}
    \phy \cn H \to H' \inspace \bilin(A,B;C),
  \end{equation} 
  the two-dimensional universal property of $\bigT$ gives rise to a monoidal transformation
    \[
    \al \cn F \to F' \inspace \strict{\bigT(A \times B), C}   \withspace  \phy_{(a,b)} = \al_{a.b}.
  \]
  One can further verify that $\al$ is natural with respect to the adjoined morphisms of $A \otimes B$ as follows.
  Naturality of $\al$ with respect to the adjoined morphisms $\de$ and $\ze$ corresponds to the requirement that $\phy$ is monoidal natural in each variable separately.
  Naturality of $\al$ with respect to the symmetry in $A \otimes B$ automatically holds by naturality of the symmetry in $C$, because $F$ and $F'$ are strict symmetric monoidal.

  By definition, the constructions of $F$ and $\al$ from $H$ and $\phy$ above are inverse to composition with $\uni$.
  Therefore, \cref{eq:otimes-bilin} is an isomorphism of categories.

  The associativity condition for $\cpost$ in \cref{lem:bilin-compose-sm} implies 2-naturality of \cref{eq:otimes-bilin} with respect to strict monoidal functors in $C$.
  For symmetric monoidal functors and monoidal transformations
  \[
    P \cn \ol{A} \to {A} \andspace
    Q  \cn \ol{B} \to {B},
  \]
  one can check directly from the constructions in \cref{lem:bilin-compose-sm,lem:tensor-bilinear,defn:FtensorG} that the diagram
  \[
    \begin{tikzpicture}[x=20mm,y=15mm]
      \draw[0cell] 
      (0,0) node (a) {\ol{A},\ol{B}}
      (1,0) node (b) {\ol{A} \otimes \ol{B}}
      (0,-1) node (a') {{A},{B}}
      (1,-1) node (b') {{A} \otimes {B}}
      ;
      \draw[1cell] 
      (a) edge node {\uni} (b)
      (a') edge node {\uni} (b')
      (a) edge['] node {P,Q} (a')
      (b) edge node {P \otimes Q} (b')
      ;
    \end{tikzpicture}
  \]
  commutes, in the sense that the bilinear maps given by the depicted composites are equal. This verifies that \cref{eq:otimes-bilin} is natural with respect to symmetric monoidal functors in $A$ and $B$.
  Similarly, 2-naturality in $A$ and $B$ is verified using the composition of monoidal transformations with bilinear maps in \cref{lem:bilin-compose-sm} and the formulas for $\phi \otimes \psi$ from \cref{defn:phi-tensor-psi}.
  This completes the proof.
\end{proof}

\begin{rmk}\label{rmk:R-not-strict}
  In \cref{eq:otimes-bilin} we only consider 2-naturality in $C$ with respect to \emph{strict} symmetric monoidal functors.
  This is because the construction $\strict{- \otimes -, -}$ is only functorial with respect to strict symmetric monoidal functors in the last variable, since composing with a non-strict symmetric monoidal functor rarely gives a strict one.
\end{rmk}

Combining \cref{prop:curry,prop:otimes-bilin-iso} gives two isomorphisms
\begin{equation}\label{eq:otimes-curry}
  \strict{A \otimes B,C} \cong \bilin(A,B;C)
  \cong \permcat\bigl(A , \permcat(B , C)\bigr)
\end{equation}
that are 2-natural with respect to symmetric monoidal functors in $A$ and $B$, and strict symmetric monoidal functors in $C$.
This implies the following.
\begin{cor}\label{cor:tensor-2adj}
  For each permutative category $B$, there is a 2-adjunction
  \[
    - \otimes B \cn \permcat \lradj \permcats \bacn \permcat(B,-).
  \]
\end{cor}
Note that this is subtly different from a closed structure, because
we must restrict to $\strict{A\otimes B,C} = \permcats(A \otimes B, C)$ on the left side of \cref{eq:otimes-curry}.
See \cref{thm:biadj} below for a tensor-hom biadjunction.

Recall from \cref{prop:bilin-perm} that bilinear maps have a pointwise monoidal structure.
Thus, \Cref{prop:otimes-bilin-iso} induces a symmetric monoidal structure on $\strict{A \otimes B, C}$ as follows.
We will use this in \cref{prop:tensor-cofib-adj-equiv} below.
\begin{defn}\label{defn:sum-for-strict-AB-C}
  Suppose $A$, $B$, and $C$ are permutative categories.
  We define the \emph{bilinear permutative structure} on $\strict{A \otimes B,C}$ to be the permutative structure induced by the isomorphism \cref{eq:otimes-bilin}.
  Explicitly, the sum $F \boxplus G$ of two strict symmetric monoidal functors
  \[
    F,G \cn A \otimes B \to C
  \]
  is defined strictly on sums of objects by
  \[
    \bigl( F \boxplus G \bigr)\bigl( \tsum a_i . b_i \bigr)
    = \tsum \bigl( F(a_i.b_i) + G(a_i.b_i) \bigr)
    \andspace
    \bigl( F \boxplus G \bigr)\bigl( 0 \bigr) = 0.
  \]
  On sums of morphisms, $\tsum \ka_i$, the sum $F \boxplus G$ is also defined to be strict monoidal, with the following assignments on simple and adjoined morphisms of $A \otimes B$:
  \begin{align*}
    \bigl( F \boxplus G \bigr) \bigl( f.g \bigr)
    & = F(f.g) + G(f.g), \\
    \bigl( F \boxplus G \bigr) \bigl( \de^L_a(b,b') \bigr)
    & = \bigl( F(\de^L_a) + G(\de^L_a) \bigr)
      \circ \bigl( 1_{F(a.b)} + \be_{G(a.b),F(a.b')} + 1_{G(a.b')} \bigr), \\
    \bigl( F \boxplus G \bigr) \bigl( \de^R_b(a,a') \bigr)
    & =  \bigl( F(\de^R_b) + G(\de^R_b) \bigr)
      \circ \bigl( 1_{F(a.b)} + \be_{G(a.b),F(a'.b)} + 1_{G(a'.b)} \bigr), \\
    \bigl( F \boxplus G \bigr) \bigl( \be \bigr)
    & = \be, \\
    \bigl( F \boxplus G \bigr) \bigl( \ze^L_a \bigr)
    & = F(\ze^L_a) + G(\ze^L_a), \andspace \\
    \bigl( F \boxplus G \bigr) \bigl( \ze^R_b \bigr)
    & = F(\ze^R_b) + G(\ze^R_b).
  \end{align*}
  The appearance of morphisms $1 + \be + 1$ above corresponds to their appearance in the sum of bilinear maps $H + J$ in \cref{eq:HplusJsymm}.

  The sum of two monoidal transformations in $\strict{A\otimes B, C}$,
  \[
    \phi \cn F \to F' 
    \andspace
    \psi \cn G \to G'
  \]
  is given componentwise by
  \[
    \bigl( \phi \boxplus \psi \bigr)_{\ssum a_i.b_i} = \tsum (\phi_{a_i.b_i} + \psi_{a_i.b_i})
    \andspace
    \bigl( \phi \boxplus \psi \bigr)_0 = 1_0.
  \]
  This completes the description of the bilinear monoidal sum.
\end{defn}

\begin{prop}\label{prop:box-iso-o}
Let $F,G \cn A \otimes B \to C$ be strict symmetric monoidal functors. Then $F \boxplus G \cong F \oplus G$.
\end{prop}
\begin{proof}
By definition of the bilinear permutative structure, we have
\[
    \bigl( F \boxplus G \bigr)\bigl( \tsum a_i . b_i \bigr)
    = \tsum \bigl( F(a_i.b_i) + G(a_i.b_i) \bigr)
    \andspace
    \bigl( F \boxplus G \bigr)\bigl( 0 \bigr) = 0,
  \]
while by the definition of the pointwise permutative structure (\cref{defn:oplus_fun_trans}), we have
\[
    \bigl( F \oplus G \bigr)\bigl( \tsum a_i . b_i \bigr)
    =  F\bigl( \tsum a_i.b_i \bigr) + G\bigl(\tsum a_i.b_i \bigr)
    \andspace
    \bigl( F \oplus G \bigr)\bigl( 0 \bigr) = 0.
  \]
  We leave it to the reader to check that the composite
  \[
  \tsum \bigl( F(a_i.b_i) + G(a_i.b_i) \bigr) \cong \bigl( \tsum F(a_i.b_i) \bigr) + \bigl( \tsum G(a_i.b_i) \bigr) = F\bigl( \tsum a_i.b_i \bigr) + G\bigl(\tsum a_i.b_i \bigr),
  \]
  given by a permutation isomorphism and strictness of $F$ and $G$, defines the component at the object $\tsum a_i . b_i$ of a monoidal natural isomorphism from $F \boxplus G$ to $F \oplus G$.
\end{proof}

\section{The tensor-hom biadjunction}
\label{sec:otimes-biadj}

In this section we establish a tensor-hom biadjunction for permutative categories, \cref{thm:biadj}.
The result depends on \cref{prop:curry,prop:otimes-bilin-iso}, together with a strictification result, \cref{prop:tensor-cofib-adj-equiv}, showing that the
local inclusion
\[
  j\cn \permcats(A\otimes B, C) \hookrightarrow \permcat(A \otimes B, C) 
\]
is an equivalence of categories, for each triple of permutative categories $A$, $B$, $C$.
The proof of this strictification result depends on the representation of objects in $A \otimes B$ as formal sums, and does not hold generally if $A \otimes B$ is replaced by an arbitrary permutative category $X$.

\begin{convention}\label{conv:tensor-hom-generic}
  Throughout this section, we continue to use the following notation from \cref{conv:bilin-generic} for generic objects and morphisms.
  \begin{center}
    \renewcommand{\arraystretch}{1.25}
    \begin{tabular}{ccc}
      category & objects & morphisms\\
      \hline
      $A$ & $a,a',a_i$ & $f, f_i$
      \\
      $B$ & $b,b',b_i$ & $g, g_i$
      \\
      $C$ & $c,c',c_i$ & $h, h_i$
    \end{tabular}
    \\ \ 
  \end{center}
\end{convention}

We begin with the definition of a strictification for symmetric monoidal functors.

\begin{defn}\label{defn:Fs}
  Let
  \[
    F \cn A \otimes B \to C
  \]
  be a symmetric monoidal functor.
  Define a strict symmetric monoidal functor
  \[
    F^s \cn A \otimes B \to C
  \]
  via the following assignments.
  \begin{description}
  \item[Objects] Define
    \[
      F^s(0) = 0 \andspace
      F^s\bigl(\tsum a_i.b_i\bigr) = \tsum F(a_i.b_i).
    \]
  \item[Morphisms] On simple and adjoined morphisms, define
    \begin{align*}
      F^s \bigl( f.g \bigr)
      & = F \bigl( f.g \bigr),
      & F^s \bigl( \be \bigr)
      & = \be, \\
      F^s \bigl( \de^L_a \bigr)
      & = F\bigl( \de^L_a \bigr) \circ F_2, 
      & F^s \bigl( \ze^L_a \bigr)
      & = F\bigl( \ze^L_a \bigr) \circ F_0, \\
      F^s \bigl( \de^R_b \bigr)
      & = F\bigl( \de^R_b \bigr) \circ F_2,
      & F^s \bigl( \ze^R_b \bigr)
      & = F\bigl( \ze^R_b \bigr) \circ F_0.
    \end{align*}
    For a formal sum of morphisms $\ka_i$ in $A \otimes B$, define
    \[
      F^s\tsum \ka_i = \tsum F^s\ka_i
    \]
    so that $F^s$ is strictly monoidal.
  \end{description} 
  This completes the definition of $F^s$ as an assignment on objects and morphisms.
\end{defn}

\begin{lem}\label{lem:Fs}
  In the context of \cref{defn:Fs}, $F^s$ is a well-defined strict symmetric monoidal functor.
\end{lem}
\begin{proof}
  To show that $F^s$ is well-defined on the imposed relations of $A \otimes B$, one uses symmetric monoidal functor axioms and coherence for $F$, together with the assumption that $F$ preserves these relations.
  For example, preservation of the pentagon~\ref{defn:AtensorB}.\cref{TEN2} is commutativity of the outer diagram below.
  \[
    \begin{tikzpicture}[x=20mm,y=18mm]
      \draw[0cell=.8] 
      (0,0) node (a) {\genatop{F(a.b) + F(a'.b)}{ + F(a.b') + F(a'.b')}}
      (4,0) node (a') {\genatop{F(a.b) + F(a.b')}{ + F(a'.b) + F(a'.b')}}
      (a)+(0,-1) node (b) {\genatop{F(a.b + a'.b)}{ + F(a.b' + a'.b')}}
      (b)+(0,-1) node (c) {\genatop{F((a + a').b)}{ + F((a + a').b')}}
      (c)+(.5,-1) node (d) {F((a + a').b + (a + a').b')}
      (d)+(1.5,-1) node (e) {F((a + a').(b + b'))}
      (a')+(0,-1) node (b') {\genatop{F(a.b + a.b')}{ + F(a'.b + a'.b')}}
      (b')+(0,-1) node (c') {\genatop{F(a.(b + b'))}{ + F(a'.(b + b'))}}
      (c')+(-.5,-1) node (d') {F(a.(b + b') + a'.(b + b'))}
      (b)+(1.2,-1) node (x) {F(X)}
      (b')+(-1.2,-1) node (y) {F(Y)}
      ;
      \draw[1cell]
      (a) edge node {1 + \be + 1} (a')
      (a) edge['] node {F_2 + F_2} (b)
      (b) edge['] node {F(\de^R_b) + F(\de^R_{b'})} (c)
      (c) edge['] node {F_2} (d)
      (d) edge['] node {F(\de^L_{a+a'})} (e)
      (a') edge node {F_2 + F_2} (b')
      (b') edge node {F(\de^L_a) + F(\de^L_{a'})} (c')
      (c') edge node {F_2} (d')
      (d') edge node {F(\de^R_{b+b'})} (e)
      ;
      \draw[1cell]
      (b) edge node {F_2} (x)
      (b') edge['] node {F_2} (y)
      (x) edge node {F(1+\be+1)} (y)
      (x) edge node[pos=.4] {F(\de^R_b + \de^R_{b'})} (d)
      (y) edge['] node[pos=.4] {F(\de^L_a + \de^L_{a'})} (d')
      ;
    \end{tikzpicture}
  \]
  In the above diagram,
  \[
    X = a.b + a'.b + a.b' + a'.b'
    \andspace
    Y = a.b + a.b' + a'.b + a'.b'.
  \]
  The two quadrilaterals commute by naturality of $F_2$, the upper hexagon commutes by coherence for $F$, and the lower pentagon is $F$ applied to~\ref{defn:AtensorB}.\cref{TEN2}.
  Verification that $F^s$ preserves the other imposed relations is similar.

  Functoriality of $F^s$ follows from functoriality of $F$ together with naturality of the constraints $F_2$ and $F_0$ and naturality of the adjoined morphisms in $A \otimes B$.
  Finally, $F^s$ is strict symmetric monoidal by its definition on objects and morphisms.
\end{proof}

\begin{defn}\label{defn:als}
  Let
  \[
    \al \cn F \to F' \cn A \otimes B \to C
  \]
  be a monoidal transformation between symmetric monoidal functors.
  Define a monoidal transformation
  \[
    \al^s \cn F^s \to (F')^s
  \]
  via components
  \begin{align}\label{eq:als-components}
    \al^s_0 = 1_{0} \andspace \al^s_{\ssum a_i.b_i} = \tsum \al_{a_i.b_i}.
  \end{align}
\end{defn} 
\begin{lem}\label{lem:als}
  In the context of \cref{defn:als}, $\al^s$ is a monoidal transformation.
\end{lem} 
\begin{proof}
  Since both $F^s$ and $(F')^s$ are strict monoidal,
  it suffices by definition \cref{eq:als-components} to verify naturality of $\al^s$ with respect to the simple and adjoined morphisms of $A \otimes B$.
  The monoidal and unit axioms of \cref{defn:montransf} are likewise immediate from \cref{eq:als-components}.
  
  Naturality of $\al^s$ with respect to a simple morphism $f.g$ reduces to naturality of $\al$.
  Naturality of $\al^s$ with respect to adjoined morphisms $\ka$ follows from the formulas for $F^s(\ka)$ and $(F')^s(\ka)$ above, together with monoidal naturality of $\al$.
  For example, naturality of $\al^s$ with respect to $\ka = \de^L_a$ is commutativity of the outer diagram below, using the monoidal naturality of $\al$ at left and naturality with respect to $\de^L_a$ at right.
  \[
    \begin{tikzpicture}[x=40mm,y=15mm]
      \draw[0cell] 
      (0,0) node (a) {F(a.b) + F(a.b')}
      (1,0) node (b) {F(a.b + a.b')}
      (2,0) node (c) {F(a.(b+b'))}
      (0,-1) node (a') {F'(a.b) + F'(a.b')}
      (1,-1) node (b') {F'(a.b + a.b')}
      (2,-1) node (c') {F'(a.(b+b'))}
      ;
      \draw[1cell] 
      (a) edge node {F_2} (b)
      (b) edge node {F(\de^L_a)} (c)
      (a') edge['] node {F'_2} (b')
      (b') edge['] node {F'(\de^L_a)} (c')
      (a) edge['] node {\al_{a.b} + \al_{a.b'}} (a')
      (b) edge node {\al_{a.b + a.b'}} (b')
      (c) edge node {\al_{a.(b+b')}} (c')
      ;
    \end{tikzpicture}
  \]
  Naturality with respect to the other adjoined morphisms $\ka$ is similar.
\end{proof}

\begin{defn}\label{defn:strictification}
  Define a strict symmetric monoidal functor
  \[
    (-)^s\cn \permcat(A \otimes B, C) \to \permcats(A \otimes B, C)
  \]
  via \cref{defn:Fs,defn:als}.
  Functoriality of $(-)^s$ follows from \cref{eq:als-components} because identities and composites of monoidal transformations are determined componentwise.
\end{defn}

\begin{prop}\label{prop:tensor-cofib-adj-equiv}
  Suppose $A$, $B$, and $C$ are permutative categories.
  There is an adjoint equivalence
  \begin{equation}\label{eq:tensor-cofib}
    j\cn \permcats(A \otimes B, C) \lradjequiv \permcat(A \otimes B, C) \bacn (-)^s
  \end{equation}
  where $j$ is the fully faithful inclusion of $\permcats$ into $\permcat$ and $(-)^s$ is given by \cref{defn:strictification}.
  Moreover, the following statements hold.
  \begin{itemize}
  \item Each of $j$ and $(-)^s$ is 2-natural with respect to strict symmetric monoidal functors $C \to \ol{C}$ and not-necessarily-strict symmetric monoidal functors $\ol{A} \to A$ and $\ol{B} \to B$.
  \item With respect to the bilinear permutative structure at left and the pointwise permutative structure at right, 
the inclusion $j$ is symmetric monoidal while $(-)^s$ is strict symmetric monoidal.
\end{itemize}
\end{prop}
\begin{proof}
  Let
  \[
    F \cn A \otimes B \to C
  \]
  be a symmetric monoidal functor.
  Recall from \cref{defn:Fs} that we have
  \[
    F^s(0) = 0 \andspace
    F^s \bigl( \tsum a_i.b_i \bigr) = \tsum F(a_i.b_i).
  \]
  Therefore, there is a natural isomorphism
  \begin{equation}\label{eq:Fs-to-F}
   \mu_F \cn F^s \To F
  \end{equation}
  given by (composites of) the monoidal and unit constraints $F_2$ and $F_0$.
  Similarly to \cref{convention:F2}, the coherence theorem for monoidal functors \cite{JS1993Braided} ensures that any two such composites with the same source and target are equal, thus implying that $\mu_F$ is monoidal.
  Naturality of \cref{eq:Fs-to-F} with respect to monoidal transformations $\al$ follows from the componentwise definition of $\al^s$ together with the monoidal naturality and unit equalities of \cref{defn:montransf} for $\al$.
  Thus, \cref{eq:Fs-to-F} provides a natural isomorphism
  \[
    \mu\cn j \circ (-)^s \To 1.  
  \]
  Using the formulas of \cref{defn:Fs,defn:als}, one verifies that the composite $(-)^s \circ j$ is the identity and both whiskerings $\mu * j$ and $(-)^s * \mu$ are identities.
  This completes the proof that $(j, (-)^s)$ is an adjoint equivalence.

  Next we verify the 2-naturality assertion.
  For $j$, this follows from 2-functoriality of the inclusion
  \[
    j\cn \permcats \hookrightarrow \permcat.
  \]
  For $(-)^s$, consider symmetric monoidal functors
  \[
    P \cn \ol{A} \to A, \quad Q \cn \ol{B} \to B, \andspace R \cn C \to \ol{C},
  \] 
  with $R$ assumed strict.
  Let
  \[
    \ol{F} = R_*(P \otimes Q)^* F \inspace \permcat(\ol{A} \otimes \ol{B}, \ol{C}), \forspace F \in \permcat(A \otimes B, C).
  \]
  Since both $\bigl( \ol{F} \bigr)^s$ and $\ol{(F^s)}$ are strict symmetric monoidal, it suffices to verify that they agree on generating objects and morphisms of $\ol{A} \otimes \ol{B}$.
  For simple objects $a.b$ and simple morphisms $f.g$ there is nothing to check.
  For the adjoined morphisms $\de^L_a$, the formulas above together with the formulas for monoidal constraints of a composite give the following because $R$ is assumed strict:
  \begin{align*}
    \bigl( \ol{F} \bigr)^s (\de^L_a)
    & = \ol{F}(\de^L_a) \circ \ol{F}_2 \\
    & = R\bigl( F \bigl( 1.Q_2 \circ \de^L_{Pa} \bigr)\bigr) \circ R\bigl(\bigl(F \circ (P \otimes Q) \bigr)_2\bigr) \\
    & = R\bigl( F^s \bigl( (P \otimes Q)(\de^L_a) \bigr)\bigr) \\
    & = \ol{(F^s)} (\de^L_a).
  \end{align*}
  A similar computation shows
  \[
    \bigl( \ol{F} \bigr)^s (\ka)
    = \ol{(F^s)} (\ka)
  \]
  for each of the other adjoined morphisms $\ka$.
  This shows that the naturality diagram
  \[
    \begin{tikzpicture}[x=45mm,y=15mm]
      \draw[0cell] 
      (0,0) node (a) {\permcat(A \otimes B, C)}
      (1,0) node (b) {\permcats(A \otimes B, C)}
      (0,-1) node (a') {\permcat(\ol{A} \otimes \ol{B}, \ol{C})}
      (1,-1) node (b') {\permcats(\ol{A} \otimes \ol{B}, \ol{C})}
      ;
      \draw[1cell] 
      (a) edge node {(-)^s} (b)
      (a') edge['] node {(-)^s} (b')
      (a) edge['] node {R_* (P \otimes Q)^*} (a')
      (b) edge node {R_* (P \otimes Q)^*} (b')
      ;
    \end{tikzpicture}
  \]
  commutes on objects.
  Consider monoidal transformations
  \[
    \al\cn F \To F' \inspace \permcat(A \otimes B, C),
  \]
  together with
  \[
    \phi \cn P \To P', \quad
    \psi \cn Q \To Q', \andspace
    \ga \cn R \To R'.
  \]
  It is straightforward to verify the equality
  \[
    \bigl( \ga * \al * (\phi \otimes \psi) \bigr)^s
    = \ga * \al^s * (\phi \otimes \psi)
  \]
  of monoidal transformations using \cref{eq:als-components}. 
  Taking each of $\phi, \psi, \ga$ to be the identity, the resulting equation shows that the naturality square above also commutes on morphisms.
  Taking $\al$ to be the identity shows that this naturality square is actually a 2-naturality square.
  This completes the proof that $(-)^s$ is 2-natural.
  
  We only prove that $(-)^s$ is strict symmetric monoidal, and the claim that $j$ is symmetric monoidal follows by doctrinal adjunction \cite[Section~2.1]{Kelly1974Doctrinal}.
  Let $F, G \cn A \otimes B \to C$ be symmetric monoidal functors. 
  We will show that 
  \[
  \bigl( F \oplus G \bigr)^s = F^s \boxplus G^s,
  \]
  where $\oplus$ is the pointwise permutative structure of \cref{defn:oplus_fun_trans} and $\boxplus$ is the bilinear permutative structure of \cref{defn:sum-for-strict-AB-C}.
  Both of these functors are strict symmetric monoidal by construction, so we only need to verify that they are equal as functors.
  As this is accomplished by checking both definitions on all of the objects and morphisms, we give some representative calculations and leave the rest for the reader.
  For objects, we have the following equalities by the definitions of $(-)^s$, $\oplus$, and $\boxplus$.
  \begin{align*}
    \bigl( F \oplus G \bigr)^s \bigl( \tsum a_i . b_i \bigr)
    & = \tsum \bigl( F \oplus G \bigr)(a_i. b_i)
      = \tsum \bigl( F(a_i.b_i) + G(a_i.b_i) \bigr)\\
    \bigl( F^s \boxplus G^s \bigr) \bigl( \tsum a_i . b_i \bigr)
    & = \tsum \bigl( F(a_i.b_i) + G(a_i.b_i) \bigr)
  \end{align*}
  Therefore, these functors agree on objects.
  The same calculation suffices for sums of morphisms $\tsum f_i.g_i$.
  On an adjoined morphism of the form $\de^L_a$, we have the following equalities.
  \begin{align*}
    \bigl( F \oplus G \bigr)^s \bigl( \de^L_a \bigr)
    & = \bigl( F \oplus G \bigr) \bigl( \de^L_a \bigr) \circ \bigl( F \oplus G \bigr)_2\\
    & = \bigl( F \oplus G \bigr) \bigl( \de^L_a \bigr) \circ \bigl( F_2 + G_2 \bigr) \circ \bigl( 1 + \be + 1 \bigr)\\
    \bigl( F^s \boxplus G^s \bigr) \bigl( \de^L_a \bigr)
    & = \Bigl( F^s\bigl(\de^L_a\bigr) + G^s\bigl(\de^L_a\bigr) \Bigr) \circ \bigl( 1 + \be + 1 \bigr)\\
    & = \Bigl( \bigl(F\bigl(\de^L_a\bigr)\circ F_2\bigr) + \bigl(G\bigl(\de^L_a\bigr)\circ G_2\bigr) \Bigr) \circ \bigl( 1 + \be + 1 \bigr)
  \end{align*}
  By the functoriality of $+$ in $C$, these functors agree on the adjoined morphisms $\de^L_a$.
  Similar calculations suffice for the other adjoined morphisms.
\end{proof}

\begin{rmk}\label{rmk:strong-needed}
  Note that the natural isomorphim $F^s \To F$ \cref{eq:Fs-to-F} in the proof of \cref{prop:tensor-cofib-adj-equiv} has components given by the monoidal and unit constraints of $F$.
  Thus, invertibility of these components depends on the assumption that $F$ is a symmetric monoidal functor in the strong sense of \cref{defn:smfunctor}.
\end{rmk}

The following result establishes a biadjunction between the 2-functors
\[
  - \otimes B \cn \permcat \lradj \permcat \bacn \permcat\bigl( B, - \bigr)
\]
by \cref{prop:biadj-loceq}.
\begin{thm}[Tensor-hom biadjunction]\label{thm:biadj}
  For permutative categories $A$, $B$, and $C$, there is an adjoint equivalence
  \begin{equation}\label{eq:tensor-hom-equiv}
    W \cn \permcat\bigl(A\otimes B, C\bigr) \lradjequiv \permcat\bigl(A, \permcat(B,C)\bigr) \bacn \Wdot
  \end{equation}
  with the following properties.
  \begin{itemize}
  \item The functor $W \cn \permcat\bigl(A\otimes B, C\bigr) \to \permcat\bigl(A, \permcat(B,C)\bigr)$ is 2-natural with respect to symmetric monoidal functors in $A$, $B$, and $C$ and is strict symmetric monoidal with respect to the pointwise structures.
  \item The functor $\Wdot \cn \permcat\bigl(A, \permcat(B,C)\bigr) \to \permcat\bigl(A\otimes B, C\bigr)$ is 2-natural with respect to symmetric monoidal functors in $A$ and $B$, is 2-natural with respect to strict symmetric monoidal functors and pseudonatural with respect to symmetric monoidal functors in $C$, and is symmetric monoidal with respect to the pointwise structures.
  \end{itemize}
\end{thm}
\begin{proof}
  Combining \cref{prop:curry,prop:otimes-bilin-iso,prop:tensor-cofib-adj-equiv} gives the following, which defines the two functors of \cref{eq:tensor-hom-equiv}
  \[
    \begin{tikzpicture}[x=40mm,y=15mm]
      \draw[0cell] 
      (0,0) node (a) {
        \permcat\bigl(A \otimes B, C\bigr) 
      }
      (1,0) node (b) {
        \strict{A \otimes B, C}
      }
      (b) +(.6,-1) node (c) {
        \bilin(A,B; C)
      }
      (b) +(1.2,0) node (d) {
        \permcat\bigl(A , \permcat(B, C)\bigr).
      }
      ;
      \draw[1cell] 
      (a) edge[transform canvas={shift={(0,3pt)}}] node {(-)^s} (b)
      (b) edge[transform canvas={shift={(0,-3pt)}}] node {j} (a)
      (b) edge node {\iso} node['] {\ref{prop:otimes-bilin-iso}} (c)
      (c) edge node {\iso} node['] {\ref{prop:curry}} (d)
      ;
    \end{tikzpicture}
  \]
  We write $W$ for the composite from left to right (using $(-)^s$), and $\Wdot$ for the composite from right to left (using $j$).
  Those results establish the adjoint equivalence and 2-naturality in $A$ and $B$ with respect to all symmetric monoidal functors, but only establish 2-naturality with respect to strict monoidal functors $C \to \ol{C}$.
  We will verify 2-naturality of \cref{eq:tensor-hom-equiv} with respect to not-necessarily-strict symmetric monoidal functors in $C$ below.
  $W$ is strict symmetric monoidal and $\Wdot$ is symmetric monoidal by \cref{prop:curry}, \cref{defn:sum-for-strict-AB-C}, and \cref{prop:tensor-cofib-adj-equiv}.
  
  Now we verify 2-naturality in $C$. Since the isomorphism of \cref{prop:curry} is 2-natural in $C$, is suffices to check 2-naturality of the composite of the first two maps.
  For not-necessarily-strict symmetric monoidal functors
  \[
    F \cn A \otimes B \to C
    \andspace
    R \cn C \to \ol{C},
  \]
  let
  \[
    H \in \bilin(A,B;C) \andspace
    \ol{H} \in \bilin(A,B;\ol{C})
  \]
  be the bilinear maps corresponding to
  \[
    F^s \in \strict{A \otimes B, C}
    \andspace
    (R_*F)^s \in \strict{A \otimes B, \ol{C}},
  \]
  respectively.
  Then one verifies that $R_*H$ and $\ol{H}$ are equal as bilinear maps.  On objects, both send $(a,b) \in A \times B$ to $RF(a.b)$. A similar calculation works for morphisms, thus showing the underlying functors are equal.  On bilinearity constraints, for example, for $a \in A$ and $b,b' \in B$ we have
  \begin{align*}
    \bigl( R_*H \bigr)_a(b,b')
    & = R \bigl( F(\de^L_a) \circ F_2 \bigr) \,\circ\, R_2, \andspace\\
    \bigl( \ol{H} \bigr)_a(b,b')
    & = \bigl( RF \bigr)(\de^L_a) \,\circ\, \bigl( RF \bigr)_2\\
    & = \bigl( RF \bigr)(\de^L_a) \,\circ\, R(F_2) \circ R_2.
  \end{align*} 
  A similar verification shows that the other bilinearity constraints agree.
  This shows that the following two composites involving $(-)^s$ are equal on objects.
  \[
    \begin{tikzpicture}[x=40mm,y=20mm]
      \draw[0cell] 
      (0,0) node (a) {\permcat(A\otimes B, C)}
      (1,0) node (b) {\strict{A \otimes B, C}}
      (2,0) node (c) {\bilin(A,B;C)}
      (0,-1) node (a') {\permcat(A\otimes B, \ol{C})}
      (1,-1) node (b') {\strict{A \otimes B, \ol{C}}}
      (2,-1) node (c') {\bilin(A,B;\ol{C})}
      ;
      \draw[1cell] 
      (a) edge[transform canvas={shift={(0,0pt)}}] node {(-)^s} (b)
      (b) edge node {\iso} (c)
      (a') edge[transform canvas={shift={(0,0pt)}}] node {(-)^s} (b')
      (b') edge node {\iso} (c')
      (a) edge['] node {R_*} (a')
      (c) edge node {R_*} (c')
      ;
    \end{tikzpicture}
  \]
  
  For the remainder of the verification of 2-naturality, we proceed as in the proof of \cref{prop:tensor-cofib-adj-equiv}. 
  Consider the following monoidal transformations between symmetric monoidal functors:
  \[
    \al \cn F \To F' \cn A \otimes B \to C
    \andspace
    \ga \cn R \To R' \cn C \to \ol{C}.
  \]
  Writing $\widehat{F} \cn A, B \to C$ for the bilinear map associated to $F$ and $\widehat{\al} \cn \widehat{F} \To \widehat{F'}$ for the bilinear transformation associated to $\al$, then using \cref{eq:als-components} one verifies
  \[
    \widehat{(\ga * \al)^s} = \ga * \widehat{(\al^s)}.
  \]
  Taking $\ga$ to be the identity shows that the naturality square involving $(-)^s$ commutes.
  Taking $\al$ to be the identity shows that this naturality square is actually a 2-naturality square.
  Thus the composite
  \[
    \begin{tikzpicture}[x=40mm,y=20mm]
      \draw[0cell] 
      (0,0) node (a) {\permcat(A\otimes B, C)}
      (1,0) node (b) {\strict{A \otimes B, C}}
      (2,0) node (c) {\bilin(A,B;C)}
      ;
      \draw[1cell] 
      (a) edge[transform canvas={shift={(0,0pt)}}] node {(-)^s} (b)
      (b) edge node {\iso} (c)
      ;
    \end{tikzpicture}
  \]
  is 2-natural with respect to symmetric monoidal functors and monoidal transformations, so composing with the 2-natural isomorphism from \cref{prop:curry} proves that $W$ is 2-natural in all variables.
  
  We now turn to the pseudonaturality of $\Wdot$ in not-necessarily-strict symmetric monoidal functors in $C$.
  Let $R \cn C \to \ol{C}$ be a symmetric monoidal functor.
  We will produce the natural isomorphism labeled $\Wdot_R$ in the square below, and then check that these assemble to give a pseudonatural transformation.
   \[
    \begin{tikzpicture}[x=50mm,y=20mm]
      \draw[0cell] 
      (0,0) node (a) {\permcat \bigl(A, \permcat(B, C) \bigr)}
      (1,0) node (b) {\permcat \bigl(A, \permcat(B, \ol{C}) \bigr)}
      (0,-1) node (a') {\permcat\bigl( A\otimes B, C \bigr)}
      (1,-1) node (b') {\permcat\bigl( A\otimes B, \ol{C} \bigr)}
      (.5,-.5) node (m) {\Downarrow \Wdot_R}
      ;
      \draw[1cell] 
      (a) edge node {R_*} (b)
      (b) edge[] node {\Wdot} (b')
      (a) edge[swap] node {\Wdot} (a')
      (a') edge['] node {R_*} (b')
      ;
    \end{tikzpicture}
  \]
In order to define a natural isomorphism $\Wdot_R \cn \Wdot \circ R_* \To R_* \circ \Wdot$, we must give a component at each object $F$ in $\permcat \bigl(A, \permcat(B, C) \bigr)$.
Such a component is a morphism of $\permcat\bigl( A\otimes B, \ol{C} \bigr)$.
These morphisms are monoidal transformations and therefore themselves have components at objects $\tsum a_i . b_i$.
Let $F \cn (A, +) \to \bigl( \permcat(B,C), \oplus \bigr)$ be a symmetric monoidal functor, and let $\tsum a_i . b_i$ be an object of $A \otimes B$.
Omitting both $F$ and $\tsum a_i . b_i$ from the notation below to avoid clutter, the component of $\Wdot_R$ is given by
\[
R_2 \cn \tsum R \bigl( F(a_i)(b_i) \bigr) \cong R \bigl( \tsum F(a_i)(b_i) \bigr),
\]
where we follow our convention that this is actually $R_0$ for a sum of length zero, the identity for a sum of length one, the morphisms $R_2$ for a sum of length two, and a composite of $R_2$'s for longer sums (all of which are equal by coherence).

We verify the axioms as follows.
\begin{itemize}
\item These components define a natural isomorphism 
\[
\Wdot \bigl( R_* F \bigr) \to R_* \Wdot(F)
\]
using the naturality of $R_2$ for morphisms of the form $\tsum f_i.g_i$, the definition of $\Wdot$ for adjoined morphisms of the form $\de$ or $\ze$, and the fact that $R$ is a symmetric monoidal functor for adjoined morphisms of the form $\be$.
\item This is a monoidal natural isomorphism using the naturality of $R_2$. Thus the components of $\Wdot_R$ are morphisms in $\permcat\bigl( A\otimes B, \ol{C} \bigr)$.
\item Given a monoidal transformation $\ga \cn R \To R'$, the equality
\[
\Wdot_{R'} \circ \bigl( \Wdot * \ga_* \bigr) = \bigl( \ga_* * \Wdot \bigr) \circ \Wdot_R
\]
holds using monoidal transformation axioms for $\ga$. Thus the 2-cell $\Wdot_R$ is natural in $R$.
\item When $R$ is the identity symmetric monoidal functor, $\Wdot_1$ is the identity by definition.
\item For a composable pair $\fto{R} \fto{S}$ of symmetric monoidal functors, the equality
   \[
    \begin{tikzpicture}[x=17.5mm,y=10mm]
      \draw[0cell] 
      (0,0) node (a) {}
      (1,0) node (b) {}
      (0,-1) node (a') {}
      (1,-1) node (b') {}
      (.5,-.5) node (m) {\Downarrow \Wdot_R}
       (2,0) node (c) {}
       (2,-1) node (c') {}
      (1.5,-.5) node (m) {\Downarrow \Wdot_S}
      (2.75,-.5) node (=) {=}
      (3.5,0) node (ta) {}
      (4.5,0) node (tb) {}
      (3.5,-1) node (ta') {}
      (4.5,-1) node (tb') {}
      (4,-.5) node (tm) {\Downarrow \Wdot_{SR}}
      ;
      \draw[1cell] 
      (a) edge node {R_*} (b)
      (b) edge[] node {} (b')
      (a) edge[swap] node {\Wdot} (a')
      (a') edge['] node {R_*} (b')
      (b) edge node {S_*} (c)
       (b') edge['] node {S_*} (c')
       (c) edge[] node {\Wdot} (c')
       (ta) edge node {(SR)_*} (tb)
      (tb) edge[] node {\Wdot} (tb')
      (ta) edge[swap] node {\Wdot} (ta')
      (ta') edge['] node {(SR)_*} (tb')
      ;
    \end{tikzpicture}
  \]
  holds by the definition of $(SR)_2$ as $S(R_2) \circ S_2$. Thus the cells $\Wdot_R$ satisfy the axioms of a pseudonatural transformation.
\end{itemize}
  \end{proof}

\section{Coherence of the adjoined morphisms}
\label{sec:coherence}

In this section we prove Coherence Theorems~\ref{thm:tensor-coherence} and~\ref{thm:iterated-tensor-coherence} for the adjoined morphisms in (iterated) tensor products.
Our approach uses a free construction and a corresponding biadjunction from \cite{BKP1989Two}. This coherence result will be used to quickly check axioms for the symmetric monoidal 2-category structure on $\permcat$, but it is also the first step in showing that the free permutative category 2-functor is itself symmetric monoidal.

\begin{defn}\label{defn:freeP}
  Let
  \[
    \bigP\cn \Cat \to \permcats
  \]
  denote the free permutative category 2-functor, and 
  \[
    U\cn \permcats \to \Cat
  \]
  denote the underlying category 2-functor.
  For a category $A$, the permutative category $\bigP A$ has objects finite sums $\tsum a_i$, for $a_i$ objects of $A$, and morphisms generated by sums $\tsum f_i$ and symmetries $\be$.
  This functor has an alternative description via the categorical Barratt-Eccles operad \cite[Lemmas~4.3-4.5]{May1974Einfty}.
  There is a 2-adjunction between $\bigP$ and $U$ giving a 2-natural isomorphism 
  \begin{equation}\label{eq:bigP-2adj}
    \permcats(\bigP X, B) \iso \Cat(X, UB),
  \end{equation} 
  where $B$ is a permutative category and $X$ is a category.
\end{defn}

\begin{defn}\label{defn:Phi}
  Suppose $X$ and $Y$ are categories.
  Define a strict symmetric monoidal functor
  \[
    \Phi\cn \bigP X \otimes \bigP Y \to
    \bigP \bigl( X \times Y \bigr)
  \]
  as follows.
  \begin{itemize}
  \item For objects and morphisms $f_i\cn x_i \to x_i'$ in $X$ and $g_j\cn y_j \to y_j'$ in $Y$,
    \begin{align*}
      \Phi\Bigl(\bigl( \tsum_i x_i \bigr).\bigl( \tsum_j y_j \bigr)\Bigr)
      & = \tsum_i \tsum_j (x_i,y_j) \andspace\\
      \Phi\Bigl(\bigl( \tsum_i f_i \bigr).\bigl( \tsum_j g_j \bigr)\Bigr)
      & = \tsum_i \tsum_j (f_i,g_j).
    \end{align*} 
    We note in the two special cases involving an empty sum, $\bigl( \tsum_i x_i \bigr).0$ and $0. \bigl( \tsum_i y_i \bigr)$, $\Phi$ takes the value $0$.
  \item For the adjoined morphism $\be\cn x_1 +x_2 \to x_2 +x_1$ in $\bigP X$ and an object $\tsum_j y_j$ in $\bigP Y$,
  \[ 
  \Phi(\be.1)=\be \cn 
  \tsum_j (x_1,y_j) + \tsum_j (x_2,y_j) \to \tsum_j (x_2,y_j) + \tsum_j (x_1,y_j).
  \]
  \item For the adjoined morphism $\be\cn y_1 +y_2 \to y_2 +y_1$ in $\bigP Y$ and an object $\tsum_i x_j$ in $\bigP X$,
  \[ 
  \Phi(1.\be)=\tsum_i \be \cn 
  \tsum_i \bigl((x_i,y_1) + (x_i,y_2)\bigr) \to \tsum_i \bigl((x_i,y_2) + (x_i,y_1)\bigr).
  \]
  \item For the adjoined morphisms
    \[
      \de^L_{\ssum x_i}\cn
      \bigl( \tsum x_i \bigr).\bigl( \tsum_j y_j\bigr) + \bigl( \tsum x_i \bigr).\bigl(\tsum_k y'_k \bigr)
      \to
      \bigl( \tsum x_i \bigr).\bigl( \tsum_j y_j + \tsum_k y'_k \bigr)
    \]
    in $\bigP X \otimes \bigP Y$,
    \begin{equation}\label{eq:Phi-deL}
      \Phi\bigl(\de^L_{\ssum x_i}\bigr) = \beta\cn
      \tsum_i \tsum_j (x_i,y_j) + \tsum_i \tsum_k (x_i,y'_k)
      \to
      \tsum_i \Bigl(\tsum_j (x_i,y_j) + \tsum_k (x_i,y'_k)\Bigr)
    \end{equation} 
    in $\bigP\bigl(X \times Y\bigr)$.
  \item For the symmetry $\beta$ in $\bigP X \otimes \bigP Y$,
    \[
      \Phi\bigl(\beta\bigr) = \beta \inspace \bigP\bigl(X \times Y\bigr).
    \]
  \item The remaining adjoined morphisms $\de^R$, $\ze^L$, and $\ze^R$ are sent to identities.
  \end{itemize} 
\end{defn}

\begin{thm}\label{thm:free-tensor-adj-equiv}
  Let $X$ and $Y$ be categories.
  There is an adjoint equivalence
  \[
    \Phi \cn \bigP X \otimes \bigP Y \lradjequiv
    \bigP \bigl( X \times Y \bigr) \bacn \Phidot
  \]
  such that the following statements hold.
  \begin{itemize}
  \item Each of $\Phi$ and $\Phidot$ is 2-natural in $X$ and $Y$.
  \item Each of $\Phi$ and $\Phidot$ is strict symmetric monoidal.
  \end{itemize}
\end{thm}
\begin{proof}
  Since $\Phi$ sends each adjoined morphism of $\bigP X \otimes \bigP Y$ to either a symmetry or identity morphism of $\bigP(X \times Y)$, verification that $\Phi$ preserves each of the imposed relations in \cref{defn:AtensorB} follows from coherence of the symmetry.
  Moreover, $\Phi$ is strict symmetric monoidal by definition.

  Define a strict symmetric monoidal adjoint inverse
  \[
    \Phidot\cn \bigP\bigl(X \times Y\bigr) \to \bigP X \otimes \bigP Y
  \]
  as the unique strict symmetric monoidal functor induced by the functor 
  \[
    X \times Y \hookrightarrow \bigP X \times \bigP Y \hookrightarrow
    \bigP X \otimes \bigP Y.
  \]
  The above is given explicitly on objects and morphisms $f \cn x \to x'$ in $X$ and $g\cn y \to y'$ in $Y$ by the assignments
  \[
    (x,y) \mapsto x.y, \andspace
    (f,g) \mapsto f.g.
  \]
  It is immediate from the definitions that $\Phi\Phidot$ is the identity on $\bigP\bigl(X \times Y\bigr)$.
  For the reverse composite, it suffices to define a monoidal transformation
  \[
    \al \cn \Phidot\Phi \To 1_{\bigP X \otimes \bigP Y}
  \]
  by its component at simple objects $s = \bigl( \tsum_i x_i \bigr).\bigl( \tsum_j y_j \bigr)$.
  Define $\al_s$ by the composite
  \[
    \tsum_i \tsum_j x_i.y_j
    \fto{\tsum_i \de^L_{x_i}}
    \tsum_i x_i.\bigl( \tsum_j y_j\bigr)
    \fto{\de^R_{\ssum y_j}}
    \bigl( \tsum_i x_i \bigr).\bigl( \tsum_j y_j \bigr),
  \]
  using \cref{convention:F2}.
  In the special case that $s = \bigl( \tsum_i x_i \bigr).0$, $\al_s$ is $\ze^L_{\ssum x_i}$, and in the special case that $s = 0. \bigl( \tsum_i x_i \bigr).\bigl( \tsum_j y_j \bigr)$, $\al_s$ is $\ze^R_{\ssum y_i}$.
  It is easy to check that $\al$ is a well-defined monoidal natural isomorphism using the naturality of $\de^L$, $\de^R$, $\ze_L$, and $\ze^R$ together with the coherence theorem for symmetric monoidal functors \cite{JS1993Braided}.
  The whiskering $\Phi * \al$ is the identity because each $\Phi(\de^L_{x_i}) = 1$.
  The whiskering $\al * \Phidot$ is the identity because in that case $s = x.y$ for $x \in X$ and $y \in Y$.
  Therefore, $(\Phi,\Phidot)$ is an adjoint equivalence.
  
  The 2-naturality of $\Phi$ with respect to functors and natural transformations
  \[
    \theta\cn J \to J' \cn X \to X'
    \andspace
    \sigma\cn K \to K' \cn Y \to Y'
  \]
  follows from a direct check, using the formulas in \cref{defn:Phi}.
  The 2-naturality of $\Phidot$ is likewise straightforward by checking generating objects and morphisms of $\bigP \bigl( X \times Y \bigr)$.
  This completes the proof.
\end{proof}

\begin{defn}\label{defn:formal-diagram}
  Let $A$ and $B$ be permutative categories.
  \begin{description}
  \item[Diagram] A \emph{diagram} $(\mathbb{D}, D)$ in a category $X$ consists of a small category $\mathbb{D}$ and a functor $D \cn \mathbb{D} \to X$. 
  We consider a morphism $f \cn x \to y$ in $X$ as a diagram by taking $\mathbb{D} = \{ \bullet \fto{!} \bullet \}$ and the functor $D$ to be defined by $D(!) = f$.
  \item[Formal diagrams] 
  A diagram $(\mathbb{D}, D)$ in $A \otimes B$
  is called \emph{formal} if there is a lift $\wt{D}$ below making the diagram commute,
  \[
    \begin{tikzpicture}[x=30mm,y=20mm]
      \draw[0cell] 
      (0,0) node (d) {\mathbb{D}}
      (1,1) node (f) {\bigP \bigl( \ob A \bigr) \otimes \bigP \bigl( \ob B \bigr)}
      (1,0) node (ab) {A \otimes B}
      ;
      \draw[1cell] 
      (d) edge[dashed] node {\wt{D}} (f)
      (f) edge node {\chi} (ab)
      (d) edge node {D} (ab)
      ;
    \end{tikzpicture}
  \]
  where $\ob A$ and $\ob B$ are regarded as discrete categories and $\chi$ is the unique strict symmetric monoidal functor induced by the inclusions $\ob A \hookrightarrow A$ and $\ob B \hookrightarrow B$.
 \item[Underlying permutations] Let $\ka$ be a formal morphism $\ka$ in $A \otimes B$.
 An \emph{underlying permutation} of $\ka$ is a morphism of the form
  $\Phi(\wt{\ka})$ in $\bigP\bigl(\ob A \times \ob B\bigr)$ for some lift $\wt{\ka}$ in $\bigP\bigl(\ob A\bigr) \otimes \bigP\bigl(\ob B\bigr)$.
  \end{description}
\end{defn}

\begin{rmk}
Note that the only morphisms in $\bigP\bigl(\ob A \times \ob B\bigr)$ are symmetries, so an underlying permutation should be viewed as a choice of source object and an element in the appropriate symmetric group.
\end{rmk}

\begin{rmk}[Example of underlying permutations]\label{rmk:foil-perm}
  For permutative categories $A$ and $B$ with $a,a' \in A$ and $b,b' \in B$, the formulas in \cref{defn:Phi} give
  \[
    \Phi\bigl( (a+a').(b+b') \bigr) = (a,b) + (a,b') + (a',b) + (a',b').
  \]
  Furthermore, using \cref{eq:Phi-deL}, the underlying permutation of the composite
  \[
    a.b + a.b' + a'.b + a'.b' \fto{\de^R \,\circ\, (\de^L + \de^L)} (a+a').(b+b')
  \]
  is the identity.
  On the other hand, \cref{eq:Phi-deL} gives the underlying permutation of the composite
  \[
    a.b + a'.b + a.b' + a'.b' \fto{\de^L \,\circ\, (\de^R + \de^R)} (a+a').(b+b')
  \]
  as that of $1 + \be + 1$.
  Thus, the underlying permutations of the two composites in the pentagon~\ref{defn:AtensorB}.(\ref{TEN2}) are equal.
\end{rmk}

\begin{rmk}[Non-uniqueness of underlying permutations]\label{rmk:non!-perm}
We note that a morphism $\ka$ in $A \otimes B$ may not have a uniquely defined underlying permutation. If an object $a \in A$ has the property that $\be_{a,a} \cn a + a \iso a + a$ is the identity, then for any object $b \in B$ the identity map on $(a+a).b$ has underlying permutation given by both the identity (by taking the identity lift) and the transposition $(1 \ 2)$ (by taking the lift $\be.1$).
\end{rmk}

\begin{thm}[Coherence in Tensor Products]\label{thm:tensor-coherence}
Let $\ka$ and $\ka'$ be two parallel morphisms in $A \otimes B$.
  Suppose the 2-arrow diagram consisting of $\ka$ and $\ka'$ is formal, and suppose $\ka$ and $\ka'$ have underlying permutations which coincide.
  Then $\ka = \ka'$.
\end{thm}
\begin{proof}
  Since the diagram $D$ consisting of $\ka$ and $\ka'$ is formal, we have lifts
  \[
    \bigl\{\!
    \begin{tikzpicture}[x=3em,baseline={(0,-3pt)}]
      \draw[0cell] (0,0) node (A) {\bullet} (1,0) node (B) {\bullet};
      \draw[1cell] (A) edge[transform canvas={shift={(0,-2pt)}}] (B);
      \draw[1cell] (A) edge[transform canvas={shift={(0,+2pt)}}] (B);
    \end{tikzpicture}
    \!\bigr\}
    \fto{\wt{D} = \{\wt{\ka},\wt{\ka}'\}} \bigP \bigl(\ob A\bigr) \otimes \bigP \bigl(\ob B\bigr)
  \]
that can be chosen such that $\Phi\wt{\ka} = \Phi\wt{\ka}'$ by hypothesis.
  The functor
  \[
    \Phi\cn \bigP \bigl(\ob A\bigr) \otimes \bigP \bigl(\ob B\bigr) \to
    \bigP \bigl( \ob A \times \ob B \bigr)
  \]
  from \cref{thm:free-tensor-adj-equiv} is a strict symmetric monoidal equivalence.
  Therefore for parallel $\wt{\ka}, \wt{\ka}'$, $\Phi\wt{\ka} = \Phi\wt{\ka}'$ if and only if $\wt{\ka} = \wt{\ka}'$ which then immediately implies that $\ka = \ka'$ by
  \[
  \ka = \chi \bigl( \wt{\ka} \bigr) = \chi \bigl( \wt{\ka}' \bigr) = \ka'.\qedhere
  \]
\end{proof}

We can extend \cref{thm:tensor-coherence} to iterated tensor products of permutative categories as follows.

\begin{defn}\label{defn:iterated-formal-diagram}
\ 
\begin{description}
\item[Iterated tensors] An \emph{iterated tensor of length 1} is a permutative category $A$. 
An \emph{iterated tensor of length $k$}, denoted $\ol{Q}$, is a triple 
\[
\ol{Q} = (Q, \ol{L_i}, \ol{R_j})
\]
where $\ol{L_i}$ is an iterated tensor of length $i$, $\ol{R_j}$ is an iterated tensor of length $j$ with $i + j = k$, and $Q=L_i\otimes R_j$. 
We call $Q$ the \emph{underlying object} of the iterated tensor.
\item[Iterated $\bigP \ob$ and $\chi$] For an iterated tensor of length 1 $A$, define
\begin{itemize}
    \item $\bigl(\bigP \ob\bigr)(A) = \bigP \bigl( \ob A \bigr)$,
    \item $T(A) = \ob A$, and
    \item $\chi \cn \bigl(\bigP \ob\bigr)(A) \to A$ as in \cref{defn:formal-diagram}.
\end{itemize}  
For an iterated tensor of length $k > 1$, $\ol{Q} = (Q, \ol{L_i}, \ol{R_j})$, define
\[
\bigl(\bigP \ob\bigr)\bigl( \ol{Q} \bigr) = \Bigl( \bigl(\bigP \ob\bigr)(L_i) \otimes  \bigl(\bigP \ob\bigr)(R_j), \bigl(\bigP \ob\bigr)\bigl( \ol{L_i} \bigr), \bigl(\bigP \ob\bigr)\bigl(\ol{R_j}\bigr)  \Bigr)
\] 
and
\[
T\bigl(\ol{Q}\bigr) = T(\ol{L_i}) \times T(\ol{R_j}).
\]
Define $\chi \cn \bigl(\bigP \ob\bigr)(Q)   \to Q$ to be 
\[ \bigl(\bigP \ob\bigr)(L_i) \otimes  \bigl(\bigP \ob\bigr)(R_j) \fto{\chi \otimes \chi}  L_i \otimes R_j.\]

\item[Formal diagrams in iterated tensors] Let $Q$ be the underlying object of an iterated tensor of length $k$. 
A diagram $(\mathbb{D}, D)$ in $Q$ is called \emph{formal} if there exists a lift $\wt{D} \cn \mathbb{D} \to \bigl(\bigP \ob\bigr)(Q) $ such that $D = \chi \circ \wt{D}$.
\item[Iterated $\Phi$] For an iterated tensor of length $1$, define $\Phi\cn \bigl(\bigP \ob\bigr)(A) \to \bigP T(A)$ to be the identity. 
For an iterated tensor of length $k >1$, $\ol{Q} = (Q,\ol{L_i}, \ol{R_j})$, define $\Phi\cn \bigl( \bigP\ob\bigr)\bigl(\ol{Q}\bigr) \to \bigP T\bigl(\ol{Q}\bigr)$ to be the composite
\[
\bigl(\bigP \ob\bigr)(L_i) \otimes \bigl(\bigP \ob\bigr)(R_j) \fto{\Phi \otimes \Phi} \bigP(T(\ol{L_i})) \otimes \bigP(T(\ol{R_j})) \fto{\Phi} \bigP\bigl( T(\ol{L_i}) \times T(\ol{R_j}) \bigr),
\]
where the second arrow above, labeled $\Phi$, is the strict symmetric monoidal equivalence from \cref{thm:free-tensor-adj-equiv}.
\item[Underlying permutations] Let $\ka$ be a formal morphism in the underlying object $Q$ of an iterated tensor of length $k$, $\ol{Q} = (Q,\ol{L_i}, \ol{R_j})$. An \emph{underlying permutation} of $\ka$ is a morphism of the form $\Phi(\wt{\ka})$ in $\bigP T\bigl(\ol{Q}\bigr)$ for some lift $\wt{\ka}$ in $\bigl(\bigP \ob\bigr)(Q) $.
\end{description}
\end{defn}

The same proof as in \cref{thm:tensor-coherence}, noting that the iterated versions of $\Phi$ are still strict symmetric monoidal equivalences, yields the following.

\begin{thm}[Coherence in Iterated Tensor Products]\label{thm:iterated-tensor-coherence}
Let $\ka$ and $\ka'$ be two parallel morphisms in the underlying object of an iterated tensor of length $k$, $\ol{Q} = (Q,\ol{L_i}, \ol{R_j})$. 
  Suppose the 2-arrow diagram consisting of $\ka$ and $\ka'$ is formal, and suppose $\ka$ and $\ka'$ have underlying permutations that coincide.
  Then $\ka = \ka'$.
\end{thm}

\section{The monoidal unit}\label{sec:unit}

In this section we define a monoidal unit $S$, along with left and right unit equivalences, for each permutative category $A$,
\[
  \dL \cn S \otimes A \to A
  \andspace
  \dR \cn A \otimes S \to A.
\]
These are \cref{defn:L,defn:R}, respectively.

\begin{convention}\label{conv:unit-generic}
  Throughout this section, we use the following notation for generic objects and morphisms, where $S$ is the unit permutative category defined below, and $A$ is a general permutative category.
  \begin{center}
    \renewcommand{\arraystretch}{1.25}
    \begin{tabular}{ccc}
      category & objects & morphisms\\
      \hline
      $S$ & $m,m',m_i$ & $\si, \si_i$
      \\
      $A$ & $a,a',a_i$ & $f, f_i$
    \end{tabular}
    \\ \ 
  \end{center}
\end{convention}

\begin{notn}\label{defn:S}
  Let $S$ denote the permutative category with objects given by the natural numbers $n \ge 0$ and hom sets
  \[
    S(m,n) =
    \begin{cases}
      \varnothing & \ifspace m \neq n\\
      \Sigma_m & \ifspace m = n,
    \end{cases}
  \]
  where $\Sigma_m$ denotes the symmetric group on $m$ letters.  The monoidal product is given by addition on objects and block sum on morphisms.
\end{notn}

\begin{rmk}\label{rmk:what-is-S}
The permutative category $S$ is a skeleton of the symmetric monoidal category of finite sets and bijections, with disjoint union as the monoidal structure. $S$ is also $\bigP 1$, the free permutative category generated by a single object. Using the adjunction for which $\bigP$ is the left adjoint, we see that $S$ corepresents the underlying category functor $U \cn \permcats \to \cat$:
\[
\permcats(S, A) \cong \cat(1, UA) \cong UA,
\]
where both isomorphisms are 2-natural in $A$.
\end{rmk}

\begin{notn}\label{notn:ma}
  Let $A$ be a permutative category and let $f \cn a \to b$ be a morphism of $A$.
  \begin{description}
    \item[Iterated sum] For a natural number $m > 0$, we write
      \[
        ma = \underbrace{a + \cdots + a}_{m \text{\ times}}
        \andspace
        mf = \underbrace{f + \cdots + f}_{m \text{\ times}}
      \]
      for the $m$-fold sums in $A$.
    \item[Empty sum] For $m = 0$ we define $0a$ to be the monoidal unit of $A$ and $0f$ is the identity, so
      \[
        0a = 0 \andspace 0f = 1_0.
      \]
    \item[Permutation action] For a permutation $\si \in \Si_m$, we let $\si$ denote the symmetry isomorphism $ma \to ma$ that permutes summands according to $\si$.
      Such a morphism is unique by the coherence theorem for symmetric monoidal categories \cite{ML98Categories}.
      We write $f^\si$ for either of the two composites below, which are equal by naturality of $\be$.
      \[
        \begin{tikzpicture}[x=20mm,y=13mm]
          \draw[0cell] 
          (0,0) node (ma) {ma}
          (1,0) node (ma2) {ma}
          (0,-1) node (mb) {mb}
          (1,-1) node (mb2) {mb}
          ;
          \draw[1cell] 
          (ma) edge node {\si} (ma2)
          (mb) edge node {\si} (mb2)
          (ma) edge['] node {mf} (mb)
          (ma2) edge node {mf} (mb2)
          ;
        \end{tikzpicture}
      \]
    \item[Perfect shuffle]\label{it:perfshuf} For objects $x_i, x'_i \in A$, there is a symmetry isomorphism
      \[
        p \cn \sum_{i=1}^m x_i + \sum_{i=1}^m x'_i \to \sum_{i=1}^m (x_i + x'_i)
      \]
      that we will call the \emph{perfect shuffle}.
      We will use the notation
      \[
        p_{m;a,a'} \cn ma + ma' \to m(a + a')
      \]
      in the case that each $x_i = a$ and each $x'_i = a'$.
    \end{description}
    The following properties of $p$ will be used below.
    \begin{enumerate}
    \item\label{it:perfshuf-ids} Each of $p_{0;a,a'}$, $p_{1;a,a'}$, $p_{m;0,a'}$, and $p_{m;a,0}$ is an identity morphism.
    \item Naturality of the symmetry isomorphism in $A$ implies that $p$ is natural in all of its variables.
    \item The diagram below commutes.
    \begin{equation}\label{eq:p-symm}
        \begin{tikzpicture}[x=60mm,y=15mm,vcenter]
          \draw[0cell] 
          (0,0) node (a) {
            \tsum x_i + \tsum y_j + \tsum x'_i + \tsum y'_j
          }
          (1,0) node (b) {
            \tsum x_i + \tsum x'_i + \tsum y_j + \tsum y'_j
          }
          (.5,-1) node (c) {
            \tsum (x_i + x'_i) + \tsum (y_j + y'_j)
          }
          ;
          \draw[1cell] 
          (a) edge node {1 + \be + 1} (b)
          (a) edge['] node[pos=.3] {p} (c)
          (b) edge node[pos=.3] {p + p} (c)
          ;
          \draw[2cell] 
          
          ;
        \end{tikzpicture}
      \end{equation}
    \end{enumerate}
    \ 
\end{notn}

\begin{defn}\label{defn:L}
  For a permutative category $A$, define a strict symmetric monoidal functor
  \[
    \dL\cn S \otimes A \to A
  \]
  as follows.
  \begin{description}
  \item[Step 1] Consider the strict monoidal functor
    \[
      \dL\cn T \bigl( S \times A \bigr) \to A
    \]
    induced by $(m,a) \mapsto ma$ and $(\si,f) \mapsto f^\si$.

  \item[Step 2]
    Extend $\dL$ to the adjoined morphisms of $S \otimes A$ with the assignments
    \begin{equation}\label{eq:Lkappa}
      \begin{aligned}
        \dL(\de^L_{m}(a,a')) & = p_{m;a,a'}, \\
        \dL(\de^R_{a}(m,m')) & = 1_{(m+m')a}, \\
        \dL(\be_{m.a,m'.a'}) & = \be_{ma,m'a'}, \\
        \dL(\ze^L_{m}) & = 1_0, \andspace \\
        \dL(\ze^R_{a}) & = 1_0.
      \end{aligned}
    \end{equation} 
  \item[Step 3]
    To verify that $\dL$ is well-defined with respect to the relations imposed in Step 3 of \cref{defn:AtensorB}, one uses naturality of $p$, the naturality and symmetric moniodal axioms for $\be$, and \cref{eq:p-symm}.
    %
    %
    %
    %
    %
  \end{description}
\end{defn}

\begin{prop}\label{prop:Lpsnat}
  The strict symmetric monoidal functors $\dL$ are the components of a pseudonatural transformation
  \[
    \dL\cn S \otimes - \To 1_{\permcat}.
  \]
\end{prop}
\begin{proof}
  The pseudonaturality constraint for $\dL$ at a symmetric monoidal functor $F \cn A \to A'$ is a monoidal natural isomorphism
  \[
    \dL_{F}\cn \dL \circ (1 \otimes F) \To F \circ \dL
  \]
  defined as follows.
  Recall from \cref{convention:F2} that we use the notation $F_2$ to denote any composite of the monoidal constraints for $F$.

  The component of $\dL_F$ at an object $\tsum m_i.a_i \in S \otimes A$ is defined to be
  \[
    F_2\cn \tsum m_i F(a_i) \to F \bigl( \tsum m_i a_i  \bigr).
  \]
  The component of $\dL_F$ at the monoidal unit $0 \in S \otimes A$ is defined to be
  \[
    F_0\cn 0 \to F(0).
  \]

  For a morphism $\si.f$ in $S \otimes A$, the morphism $\dL(\si.f) = f^\si$ is a composite of $f$ with symmetry isomorphisms.
  For each adjoined morphism $\ka$ in $S \otimes A$, the morphism $\dL(\ka)$ is either an identity or a composite of symmetry isomorphisms.
  Thus, naturality of $\dL_F$ with respect to morphisms in $S \otimes A$ follows from the symmetric monoidal functor axioms for $(F,F_2,F_0)$.
  Likewise, the symmetric monoidal functor axioms for $(F,F_2,F_0)$ imply the monoidal naturality axioms for $\dL_F$.

  To verify the pseudonaturality axioms for $\dL$, one uses
  \begin{itemize}
  \item the definition of the monoidal constraint for a composite of symmetric monoidal functors, $GF$, and
  \item the monoidal naturality axioms for a monoidal transformation $\al$ from  $F$ to $F'$.
  \end{itemize}
  Thus, $\dL$ is a pseudonatural transformation.
\end{proof}

\begin{rmk}\label{rmk:psnatL}
  The pseudonaturality constraints $\dL_{F}$ in \cref{prop:Lpsnat} have components determined by the monoidal and unit constraints of $F$.
  Therefore, if $F$ is assumed to be only lax monoidal, then $\dL_F$ will be a not-necessarily-invertible transformation.
\end{rmk}

\begin{prop}\label{prop:Lequiv}
  The pseudonatural transformation
  \[
    \dL\cn S \otimes - \to 1_{\permcat}
  \]
  is a pseudonatural equivalence.
\end{prop}
\begin{proof}
  Let $A$ be a permutative category. 
  A pseudonatural transformation is an equivalence if and only if each component is an equivalence \cite{Kel89Elementary} so we construct a weak inverse for each $\dL_A$. 
  The natural number $m=1$ gives a symmetric monoidal functor
  \begin{equation}\label{eq:Ldot}
    \dLdot = 1.(-)\cn A \to S \otimes A
  \end{equation} 
  as in \cref{defn:AtensorB}~\cref{TEN1}.
  The composite $\dL \circ \dLdot$ is the identity symmetric monoidal functor on $A$ because, recalling \cref{eq:Lkappa} and  \cref{notn:ma}~\cref{it:perfshuf-ids}, we have
  \begin{equation}\label{eq:LLdot-id}
    \dL\bigl(\de^L_1(a,a')\bigr) = p_{1;a,a'} = 1_{a+a'}
    \andspace
    \dL \bigl( \ze^L_1 \bigr) = 1_0.
  \end{equation}
  We will show that the other composite,
  \[
    S \otimes A \fto{\dL} A \fto{1.(-)} S \otimes A,
  \]
  is isomorphic, as a symmetric monoidal functor, to the identity on $S \otimes A$.
  It suffices to just check the isomorphism at the level of functors between categories, but our proof uses the monoidal structure to simplify the calculations.
  It will follow automatically that these isomorphisms will be the components of a modification
  \[
    \eta\cn 1_{S \otimes -} \To \dLdot \dL,
  \]
  and our calculations will show that $\eta * \dLdot = 1_{\dLdot}$ and $\dL * \eta = 1_{\dL}$.

  Recall that $(-).a$ is a monoidal functor with $\de^R_a$ as its monoidal constraint.
  Therefore, by coherence for monoidal functors \cite{JS1993Braided}, there is a unique composite of sums of morphisms $\de^R_a(i,j)$ 
  \[
    m(1.a) = \underbrace{1.a + \cdots + 1.a}_{m \text{\ terms}} \to m.a.
  \]
  As in \cref{convention:F2}, we will write $\de^R_a$ for any such composite.
  Similarly, we write $\de^L_1$ for any composite of sums of $\de^L_1(ia,ja)$ giving a morphism $m(1.a) \to 1.(ma)$.

  Now we define component morphisms of
  \begin{equation}\label{eq:etaLA}
    \eta = \eta_A \cn 1_{S \otimes A} \To 1.(-) \circ \dL
  \end{equation}
  as follows. 
  For the unit $0$, let
  \[
    \eta_0 = \ze^L_{1} \cn 0 \to 1.0.
  \]
  Next, let $\ze^{-R}$ and $\de^{-R}_a$ denote the inverses of $\ze^R$ and $\de^R_a$.
  Define
  \[
    \begin{tikzpicture}[x=30mm,y=13mm,vcenter]
      \draw[0cell] 
      (0,0) node (a) {0.a}
      (.5,-1) node (b) {0}
      (1,0) node (c) {1.0}
      ;
      \draw[1cell] 
      (a) edge['] node {\ze^{-R}_{a}} (b)
      (b) edge['] node {\ze^L_{1}} (c)
      (a) edge[dashed] node {\eta_{0.a}} (c)
      ;
    \end{tikzpicture}
    \andspace
    \begin{tikzpicture}[x=30mm,y=13mm,vcenter]
      \draw[0cell] 
      (0,0) node (a) {m.a}
      (.5,-1) node (b) {m(1.a)}
      (1,0) node (c) {1.(ma)}
      ;
      \draw[1cell] 
      (a) edge['] node {\de^{-R}_a} (b)
      (b) edge['] node {\de^L_1} (c)
      (a) edge[dashed] node {\eta_{m.a}} (c)
      ;
    \end{tikzpicture}
  \]
  for $m > 0$.
  The component at a sum of objects is then given by the following composite.
  \[
    \begin{tikzpicture}[x=40mm,y=13mm]
      \draw[0cell] 
      (0,0) node (a) {\tsum m_i.a_i}
      (.5,-1) node (b) {\tsum 1.(m_i a_i)}
      (1,0) node (c) {1.(\tsum m_i a_i)}
      ;
      \draw[1cell] 
      (a) edge['] node {\tsum \eta_{m_i.a_i}} (b)
      (b) edge['] node {\de^L_1} (c)
      (a) edge[dashed] node {\eta_{\ssum m_i.a_i}} (c)
      ;
    \end{tikzpicture}
  \] 

  Each component of $\eta = \eta_A$ is an isomorphism by construction.
  Naturality of $\eta_A$ is verified in cases, first for morphisms of the form $\si.f$ and then for each of the adjoined morphisms in $S \otimes A$.
  We give one example, and note that all of the naturality squares commute by the Coherence Theorem~\ref{thm:tensor-coherence} for adjoined morphisms in a tensor product.

  The naturality square for $\eta_A$ with respect to morphisms $\de^L_m$ is the square below, where the relevant components of $\eta$ are the vertical composites at left and right.
  \[
    \begin{tikzpicture}[x=80mm,y=15mm]
      \draw[0cell] 
      (0,0) node (a) {m.a + m.a'}
      (0,-1) node (b) {m(1.a) + m(1.a')}
      (0,-2) node (c) {1.(ma) + 1.(ma')}
      (0,-3) node (d) {1.(ma + ma')}
      (a)+(1,0) node (a') {m.(a+a')}
      (b)+(1,-.5) node (b') {m\bigl(1.(a+a')\bigr)}
      (d)+(1,0) node (d') {1.\bigl( m(a+a') \bigr)}
      ;
      \draw[1cell] 
      (a) edge['] node {\de^\inv_a + \de^\inv_{a'}} (b)
      (b) edge['] node {\de_1 + \de_1} (c)
      (c) edge['] node {\de_1} (d)
      (a') edge node {\de_{a+a'}^\inv} (b')
      (b') edge node {\de_1} (d')
      (a) edge node {\de^L_m} (a')
      (d) edge node {\dL(\de^L_m) = 1.p} (d')
      ;
    \end{tikzpicture}
  \]
  Each of the other cases follows similarly.
  One can use the Coherence Theorem~\ref{thm:tensor-coherence} or, more specifically, use naturality from condition \cref{TENnat} for the relevant data, the symmetric monoidal functor axioms for $1.(-)$ and $(-).a$ from \cref{TEN1}, and other relations imposed in the definition of $S \otimes A$.

  The monoidal transformation axioms for each $\eta_A$ follow by construction of $\eta_A$ and the monoidal functor axioms for $1.(-)$.
  This completes the proof that each $\eta_A$ in \cref{eq:etaLA} is a monoidal natural isomorphism $1_{S \otimes A} \to 1.(-) \circ \dL$, and thus verifies that $\dL$ is a pseudonatural equivalence.
\end{proof}

\begin{lem}\label{lem:Ldot-2nat}
  The components
  \[
    \dLdot_A = 1.(-)\cn A \to S \otimes A  
  \]
  in \cref{eq:Ldot} are 2-natural in $A$.
\end{lem}
\begin{proof}
  Suppose given a symmetric monoidal functor $F \cn A \to B$ and consider the two composites around the following square.
  \begin{equation}\label{eq:Ldot-2nat-square}
    \begin{tikzpicture}[x=30mm,y=15mm,vcenter]
      \draw[0cell] 
      (0,0) node (a) {A}
      (a)++(1,0) node (b) {B}
      (a)++(0,-1) node (sa) {S \otimes A}
      (b)++(0,-1) node (sb) {S \otimes B}
      ;
      \draw[1cell] 
      (a) edge node {F} (b)
      (sa) edge node {1_{S} \otimes F} (sb)
      (a) edge['] node {\dLdot_A} (sa)
      (b) edge node {\dLdot_B} (sb)
      ;
    \end{tikzpicture}
  \end{equation} 
  On objects $a \in A$, the composite $\bigl(1_{S} \otimes F \bigr) \circ \dLdot_A$ is given by
  \[
    \Bigl(\bigl(1_{S} \otimes F \bigr) \circ \dLdot_A\Bigr) (a)
    = \bigl(1_{S} \otimes F \bigr) \bigl(1.a \bigr) = 1.(Fa) = \bigl( \dLdot_B \circ F \bigr)(a).
  \]
  A similar calculation on morphisms shows that the underlying functors of the two composites around \cref{eq:Ldot-2nat-square} are equal.

  It remains to check that the two composites around \cref{eq:Ldot-2nat-square} have the same unit and monoidal constraints.
  Both arguments proceed in the same fashion, so we only give the proof for the monoidal constraint.
  Since $1_S \otimes F$ is strict monoidal, the monoidal constraint of the composite $\bigl( 1_S \otimes F \bigr) \circ \dLdot_A$ is
  \[
    \bigl( \bigl(1_{S} \otimes F \bigr) \circ \dLdot \bigr)_2
    = \bigl(1_{S} \otimes F \bigr) \bigl( \bigl( \dLdot_A \bigr)_2 \bigr).
  \]
  The component of $\bigl( \dLdot_A \bigr)_2$ at a pair of objects $a$ and $a'$ is given by $\de^L \cn 1.a + 1.a' \cong 1.(a + a')$.
  Applying $1 \otimes F$ to $\de^L$ produces the composite
  \begin{equation}\label{eq:leftbot-mon}
    1.Fa + 1.Fa' \fto{\de^L} 1.\bigl( Fa + Fa' \bigr) \fto{1.F_2} 1.F\bigl( a + a' \bigr).
  \end{equation} 
  On the other hand, the monoidal constraint of the composite $\dLdot_B \circ F$ is
  \[
    \bigl( \dLdot_B \circ F \bigr)_2 = \bigl( \dLdot_B \bigr)_2 \circ \dLdot_B \bigl( F_2 \bigr).
  \]
  The component of $\bigl( \dLdot_B \circ F \bigr)_2$ at a pair of objects $a$ and $a'$ is therefore equal to \cref{eq:leftbot-mon}.
  Thus, the monoidal constraints for the two composites around \cref{eq:Ldot-2nat-square} are equal.
  A similar computation shows that their unit constraints are also equal, and hence $\dLdot$ is natural with respect to symmetric monoidal functors $F \cn A \to B$.

  For 2-naturality of $\dLdot$, suppose given a monoidal transformation
  \[
    \phi \cn F \to F' 
  \]
  between symmetric monoidal functors $F,F'\cn A \to B$.
  Recalling \cref{defn:phi-tensor-psi} and checking components shows that
  \[
    \dLdot_B * \phi = \bigl( 1_S \otimes \phi \bigr) * \dLdot_A,
  \]
  as desired.
\end{proof}

\begin{rmk}
  An alternative proof of \cref{lem:Ldot-2nat} can be given as follows.
  The proof of \cref{prop:Lequiv} shows that each $\dLdot_A$ is an inverse equivalence to $\dL_A$.
  It follows from \cite{Kel89Elementary} that the components $\dLdot_A$ assemble to a pseudonatural transformation, with pseudonaturality constraints constructed using whiskerings of $\dL_F$ and $\eta$ by $\dLdot$.
  One can check directly that these whiskerings are identities, giving another proof that $\dLdot$ is 2-natural.
\end{rmk}

\begin{lem}\label{lem:L-factorization}
For any symmetric monoidal functor $F \cn A \to B$, the equality
\[
F = \dL \circ \bigl(1_{S} \otimes F \bigr) \circ \dLdot
\]
holds.
\end{lem}
\begin{proof}
  By 2-naturality of $\dLdot$ we have $\bigl( 1_S \otimes F \bigr) \circ \dLdot = \dLdot \circ F$.
  Then, the result follows because, recalling \cref{eq:LLdot-id}, the composite $\dL \circ \dLdot \cn B \to B$ is the identity on $B$.
\end{proof}

\begin{lem}\label{lem:L-1Ldot}
For any permutative category $A$, the composite
\[
S \otimes A \fto{1 \otimes \dLdot} S \otimes (S \otimes A) \fto{\dL} S \otimes A
\]
is isomorphic, as a symmetric monoidal functor, to the identity. These isomorphisms constitute the components of an invertible modification.
\end{lem}
\begin{proof}
The composite $\dL \circ \bigl( 1 \otimes \dLdot \bigr)$ is strict monoidal and sends a simple tensor $n.a$ in $S \otimes A$ to the sum
\[
\sum_{i = 1}^n 1.a.
\]
Define a monoidal natural isomorphism $d \cn \dL \circ \bigl( 1 \otimes \dLdot \bigr) \cong 1_{S \otimes A}$ to have its component $d_{n.a}$ at a simple tensor $n.a$ be
\begin{equation}\label{eq:dna}
  \begin{aligned}
    d_{n.a} & = \de^R \cn \sum_{i = 1}^n 1.a \cong n.a, \quad n >0, \\
    d_{0.a} &  = \ze^R_a,
  \end{aligned}
\end{equation}
and to have its component at a sum of simple tensors be the sum of these.
This morphism is clearly natural in $a$, and it is natural in $n$ using \cref{TEN1} for the right multiplication functor $-.a$.
It is a simple exercise to compute, case-by-case, that it is natural in each of the adjoined morphisms $\de$, $\be$, $\ze$.
The components are isomorphisms by definition, and monoidal by construction, completing the proof of the existence of a monoidal natural isomorphism.

To show that these isomorphisms assemble into a modification, we begin by noting that the composite $\dL \circ \bigl( 1 \otimes \dLdot \bigr)$ is in fact 2-natural with respect to symmetric monoidal functors $F \cn A \to B$.
This holds as a consequence of the following two observations.
First, $\dLdot$ is 2-natural, by \cref{lem:Ldot-2nat}.
Second, the pseudonaturality constraint $\dL_{1 \otimes F}$---determined in the proof of \cref{prop:Lpsnat} by the monoidal constraint $\bigl( 1 \otimes F \bigr)_2$---is the identity because $1 \otimes F$ is strict monoidal.

Hence, we only need to verify the equality
\[
\bigl( 1 \otimes F \bigr) * d = d * \bigl( 1 \otimes F \bigr)
\]
for each symmetric monoidal functor $F$.
This equality holds by checking components and using \cref{eq:dna} along with \cref{defn:FtensorG}.
\end{proof}

\begin{cor}\label{cor:Sotimes-left-biadj}
The 2-functor $S \otimes - \cn \permcat \to \permcats$ is left biadjoint to the inclusion $\permcats \to \permcat$.
\end{cor}
\begin{proof}
Following \cref{prop:biadj-loceq}, we produce an equivalence of categories 
\[
\permcats \bigl( S \otimes A, B) \fto{\simeq} \permcat \bigl(A, B \bigr),
\]
that is 2-natural in $A$ and $B$.
The functor $\permcats \bigl( S \otimes A, B) \to \permcat \bigl(A, B \bigr)$ is given by precomposition with $\dLdot$.
This functor is visibly 2-natural in $B$, and 2-natural in $A$ by the 2-naturality of $\dLdot$ (\cref{lem:Ldot-2nat}).
The functor $\permcat \bigl(A, B) \to \permcats \bigl(S \otimes A, B \bigr)$ is given on objects by $F \mapsto \dL \circ \bigl( 1_S \otimes F \bigr)$ and similarly on morphisms.
Note that both $\dL$ and any symmetric monoidal functor of the form $F \otimes G$ are strict symmetric monoidal.
This functor is 2-natural in $A$ using the 2-functoriality of $S \otimes -$, and pseudonatural in $B$ using the pseudonaturality of $\dL$.
By \cref{lem:L-factorization}, the composite
\[
\permcat \bigl(A, B) \to \permcats \bigl(S \otimes A, B \bigr) \to \permcat \bigl(A, B)
\]
is the identity.
By \cref{lem:L-1Ldot}, the composite
\[
\permcats \bigl( S \otimes A, B) \to \permcat \bigl(A, B \bigr) \to \permcats \bigl( S \otimes A, B)
\]
is isomorphic to the identity via
\[
\dL \circ \bigl(1 \otimes F \bigr) \circ \bigl(1 \otimes \dLdot \bigr) = F \circ \dL \circ \bigl(1 \otimes \dLdot \bigr) \fto{1_F * d} F.
\]
These verifications satisfy the hypotheses of  \cref{prop:biadj-loceq} and therefore produce the desired biadjunction.
\end{proof}

\begin{rmk}\label{rmk:Sotimes-fact}
By \cite{BKP1989Two}, the inclusion $\permcats \to \permcat$ actually has a left 2-adjoint.
Inspecting the proofs above, it is easy to see that it is given by quotienting $S \otimes A$ by forcing the components of $d$ to be identities.
This quotient is the expected form of the pseudomorphism classifier (see \cite{JS1993Braided} for example).
Thus it is possible to have a tighter form of adjunction at the expense of having to perform additional quotients.
From a homotopical perspective, \cref{lem:L-factorization} already gives a well-behaved factorization of any symmetric monoidal functor into an equivalence followed by a strict symmetric monoidal functor, showing that $S \otimes -$ is a perfectly good candidate for a kind of cofibrant replacement.
From a 2-monadic perspective, the proof of \cref{prop:tensor-cofib-adj-equiv} shows that any permutative category of the form $A \otimes B$ is semi-flexible \cite{BKP1989Two}, so in particular $S \otimes A$ is.
\end{rmk}

\begin{defn}\label{defn:R}
  For a permutative category $A$, define a strict symmetric monoidal functor
  \[
    \dR\cn A \otimes S \to A
  \]
  as follows.
  The assignments $(a,m) \mapsto ma$ and $(f,\si) \mapsto f^\si$ induce a strict monoidal functor.
  \[
    \dR\cn T \bigl( A \times S  \bigr) \to A.
  \]
  Then we extend $\dR$ to the adjoined morphisms of $A \otimes S$ with the assignments
  \begin{align*}
    \dR(\de^L_{a}(m,m')) & = 1_{(m+m')a}, \\
    \dR(\de^R_{m}(a,a')) & = p_{m;a,a'}, \\
    \dR(\be_{a.m,a'.m'}) & = \be_{ma,m'a'}, \\
    \dR(\ze^L_{a}) & = 1_0, \andspace \\
    \dR(\ze^R_{m}) & = 1_0.
  \end{align*} 
  Verification that $\dR$ is well-defined with respect to the imposed relations in $A \otimes S$ is similar to that of $\dL$.

  Given a symmetric monoidal functor $F\cn A \to A'$, let
  \[
    \dR_F\cn \dR \circ (F \otimes 1) \To F \circ \dR
  \]
  be defined by components $F_2$ and $F_0$ as in \cref{prop:Lpsnat}.
  Verification that these define a pseudonaturality constraint for $\dR$ is the same as that for $\dL$.
  \Cref{prop:Requiv} below shows that $\dR$ is a pseudonatural equivalence.
\end{defn}

\section{The symmetry}\label{sec:B}

In this section we define a strict symmetric monoidal functor
\[
  \dB \cn A \otimes C \to C \otimes A
\]
for permutative categories $A$ and $C$.  We show that $\dB$ is an idempotent isomorphism in \cref{prop:Bequiv} and use this to show that $\dR$ is a pseudonatural equivalence in \cref{lem:R=LB,prop:Requiv}.

\begin{defn}\label{defn:B}
  For permutative categories $A$ and $C$, define a strict symmetric monoidal functor
  \[
    \dB\cn A \otimes C \to C \otimes A
  \]
  as follows.
  The symmetry of the cartesian product induces a strict monoidal functor
  \[
    \dB\cn T(A \times C) \to C \otimes A.
  \]
  Then we extend $\dB$ to the adjoined morphisms of $A \otimes C$ with the following assignments for $a,a' \in A$ and $c,c' \in C$:
  \begin{align*}
    \dB(\de^L_{a}(c,c')) & = \de^R_{a}(c,c'), \\
    \dB(\de^R_c(a,a')) & = \de^L_{c}(a,a'), \\
    \dB(\be_{a.c,a'.c'}) & = \be_{c.a,c'.a'}, \\
    \dB(\ze^L_{a}) & = \ze^R_a, \andspace \\
    \dB(\ze^R_c) & = \ze^L_c. \\
  \end{align*}
  Verification that $\dB$ is well-defined with respect to the imposed relations in $A \otimes C$ is similar to that of $\dR$ and $\dL$.  Comparing the definition of $\dB$ with that of $F \otimes G$ in \cref{defn:FtensorG} and $\phi \otimes \psi$ in \cref{defn:phi-tensor-psi} shows that $\dB$ is 2-natural in $A$ and $C$.
\end{defn}

The assignments listed above give the following result.
\begin{prop}\label{prop:Bequiv}
  The composite
  \[
    A \otimes C \fto{\dB} C \otimes A \fto{\dB} A \otimes C
  \]
  is the identity symmetric monoidal functor.
  Therefore, $\dB$ is a 2-natural isomorphism of permutative categories.
\end{prop}

The formulas in \cref{defn:L,defn:R,defn:B} imply the following.
\begin{lem}\label{lem:R=LB}
  There is an equality of pseudonatural transformations
  \[
    \dR = \dL\dB.
  \]
\end{lem}
Since $\dL$ is a pseudonatural equivalence with adjoint $\dLdot = 1.(-)$, and $\dB$ is a 2-natural isomorphism, \cref{lem:R=LB} implies the following. \begin{cor}\label{prop:Requiv}
  The strict symmetric monoidal functors $\dR$ are the components of a pseudonatural equivalence
  \[
    \dR \cn - \otimes S \to 1_{\permcat}
  \]
  with adjoint $\dRdot = (-).1$.
\end{cor}

\section{The associativity}\label{sec:A}

In this section we define a strict symmetric monoidal functor
\[
  \dA\cn \bigl( B \otimes C \bigr) \otimes D \to
  B \otimes \bigl( C \otimes D \bigr)
\]
for permutative categories $B$, $C$, and $D$.
We give the assignments on objects and morphisms in \cref{defn:A} and verify that $\dA$ is well-defined in \cref{prop:Awelldef} below.

\begin{convention}\label{conv:assoc-generic}
  Throughout this section, we use the following notation for generic objects and morphisms.
  \begin{center}
    \renewcommand{\arraystretch}{1.25}
    \begin{tabular}{ccc}
      category & objects & morphisms\\
      \hline
      $B$ & $b,b',b_i$ & $f, f_i$
      \\
      $C$ & $c,c',c_i$ & $g, g_i$
      \\
      $D$ & $d,d',d_i$ & $h, h_i$\\
      $B \otimes C$ & $X, X', Y, Y'$ & $\ka$, $\ka_i$ \\
      $B \otimes (C \otimes D)$ & - & $\la$
    \end{tabular}
    \\ \ 
  \end{center}
\end{convention}

\begin{defn}\label{defn:A}
  We define the assignments for
  \[
    \dA \cn (B \otimes C) \otimes D \to B \otimes (C \otimes D)
  \]
  as follows.
  \begin{description}
  \item[Objects]
    The assignment on objects of $(B \otimes C) \otimes D$ is given by
    \begin{align*}
      \dA (0) & = 0, \\
      \dA (0.d) & = 0, \\
      \dA \bigl(\, \bigl( \tsum b_i.c_i  \bigr).d  \,\bigr)
              & = \tsum b_i.(c_i.d),
    \end{align*}
    and extended to sums in $(B \otimes C) \otimes D$ so that $\dA$ is strictly monoidal.

  \item[Morphisms $\ka.h$]
    Suppose $\ka$ is a morphism in $B \otimes C$ and $h$ is a morphism in $D$.  If $\ka$ is itself a sum $\tsum \ka_i$, then we define
    \[
      \dA \bigl( \bigl( \tsum \ka_i  \bigr).h  \bigr) = \sum \dA \bigl(\ka_i.h \bigr).
    \]
    We now treat each possible type of simple or adjoined morphism $\ka$ in $B \otimes C$.
    \begin{itemize}
    \item If $\ka = f.g$, then we define
      \[
        \dA \bigl( ( f.g ).h  \bigr) =  f.(g.h).
      \]
    \item If $\ka = 1_0$, then we define
      \[
        \dA (1_0 . h) = 1_0.
      \]
    \item If $\ka$ is one of the adjoined morphisms in $B \otimes C$,
      we define $\dA(\ka.h)$ via the indicated composite below.
      In each of the first three cases below, there are alternative composites using $1.(1.h)$ in different positions.  
      These all result in the same definition for $\dA(\ka.h)$ by naturality.
    \end{itemize}
    \[
      \begin{tikzpicture}[x=25mm,y=15mm]
        \draw[0cell] 
        (0,0) node (a) {b.(c.d) + b.(c'.d)}
        (.75,-1) node (b) {b.\bigl( c.d + c'.d \bigr)}
        (2.25,-1) node (c) {b.\bigl( (c+c').d  \bigr)}
        (3,0) node (d) {b.\bigl( (c+c').d'  \bigr)}
        (a)+(0,.6) node (a') {\dA\bigl(\, (b.c + b.c').d \,\bigr)}
        (d)+(0,.6) node (d') {\dA\bigl(\, (b.(c + c')).d' \,\bigr)}
        ;
        \draw[1cell] 
        (a) edge['] node {\de^L_{b}} (b)
        (b) edge['] node {1.\de^R_{d}} (c)
        (c) edge['] node {1.(1.h)} (d)
        (a') edge[dashed] node {\dA\bigl(\,\de^L_{b}.h\,\bigr)} (d')
        (a') edge[equal] node {} (a)
        (d') edge[equal] node {} (d)
        ;
      \end{tikzpicture}
    \]
  
    \[
      \begin{tikzpicture}[x=25mm,y=15mm]
        \draw[0cell] 
        (0,0) node (a) {b.(c.d) + b'.(c.d)}
        (a)+(1,-1) node (b) {(b+b').(c.d)}
        (a)+(2,0) node (d) {(b+b').(c.d')}
        (a)+(0,.6) node (a') {\dA\bigl(\, (b.c + b'.c).d \,\bigr)}
        (d)+(0,.6) node (d') {\dA\bigl(\, ((b+b').c).d' \,\bigr)}
        ;
        \draw[1cell] 
        (a) edge['] node {\de^R_{c.d}} (b)
        (b) edge['] node {1.(1.h)} (d)
        (a') edge[dashed] node {\dA\bigl(\,\de^R_{c}.h\,\bigr)} (d')
        (a') edge[equal] node {} (a)
        (d') edge[equal] node {} (d)
        ;
      \end{tikzpicture}
    \]

    \[
      \begin{tikzpicture}[x=25mm,y=15mm]
        \draw[0cell] 
        (0,0) node (a) {b.(c.d) + b'.(c'.d)}
        (a)+(1,-1) node (b) {b'.(c'.d) + b.(c.d)}
        (a)+(2,0) node (d) {b'.(c'.d') + b.(c.d')}
        (a)+(0,.6) node (a') {\dA\bigl(\, (b.c + b'.c').d \,\bigr)}
        (d)+(0,.6) node (d') {\dA\bigl(\, (b'.c' + b.c).d' \,\bigr)}
        ;
        \draw[1cell] 
        (a) edge['] node {\be} (b)
        (b) edge['] node {1.(1.h) + 1.(1.h)} (d)
        (a') edge[dashed] node {\dA\bigl(\,\be.h\,\bigr)} (d')
        (a') edge[equal] node {} (a)
        (d') edge[equal] node {} (d)
        ;
      \end{tikzpicture}
    \]

    \[
      \begin{tikzpicture}[x=25mm,y=15mm]
        \draw[0cell] 
        (0,0) node (a) {0}
        (a)+(.75,-1) node (b) {b.0}
        (b)+(1.5,0) node (c) {b.(0.d)}
        (a)+(3,0) node (d) {b.(0.d')}
        (a)+(0,.6) node (a') {\dA\bigl(\, 0.d \,\bigr)}
        (d)+(0,.6) node (d') {\dA\bigl(\, (b.0).d' \,\bigr)}
        ;
        \draw[1cell] 
        (a) edge['] node {\ze^L_{b}} (b)
        (b) edge['] node {1.\ze^R_{d}} (c)
        (c) edge['] node {1.(1.h)} (d)
        (a') edge[dashed] node {\dA\bigl(\,\ze^L_{b}.h\,\bigr)} (d')
        (a') edge[equal] node {} (a)
        (d') edge[equal] node {} (d)
        ;
      \end{tikzpicture}
    \]

    \[
      \begin{tikzpicture}[x=25mm,y=15mm]
        \draw[0cell] 
        (0,0) node (a) {0}
        (a)+(1,-1) node (b) {0.(c.d)}
        (a)+(2,0) node (d) {0.(c.d')}
        (a)+(0,.6) node (a') {\dA\bigl(\, 0.d \,\bigr)}
        (d)+(0,.6) node (d') {\dA\bigl(\, (0.c).d' \,\bigr)}
        ;
        \draw[1cell] 
        (a) edge['] node {\ze^R_{c.d}} (b)
        (b) edge['] node {1.(1.h)} (d)
        (a') edge[dashed] node {\dA\bigl(\,\ze^R_{c}.h\,\bigr)} (d')
        (a') edge[equal] node {} (a)
        (d') edge[equal] node {} (d)
        ;
      \end{tikzpicture}
    \] 

  \item[Adjoined morphisms] For morphisms $\la$ adjoined in the passage from $- \times D$ to $- \otimes D$, we define $\dA(\la)$ via the indicated composites below.  Then $\dA$ is extended to sums of morphisms so that it is symmetric monoidal.

    \[
      \begin{tikzpicture}[x=30mm,y=15mm]
        \draw[0cell] 
        (0,0) node (a) {\tsum b_i.(c_i.d) + \tsum b_i.(c_i.d')}
        (a)+(.25,-1) node (p) {\tsum \bigl( b_i.(c_i.d) + b_i.(c_i.d') \bigr)}
        (p)+(1.5,0) node (b) {\tsum b_i.(c_i.d + c_i.d')}
        (a)+(2,0) node (d) {\tsum b_i.(c_i.(d+d'))}
        (a)+(0,.6) node (a') {
          \dA\bigl(\, \bigl(\tsum b_i.c_i\bigr).d + \bigl(\tsum b_i.c_i\bigr).d' \,\bigr)
        }
        (d)+(0,.6) node (d') {
          \dA\bigl(\, \bigl(\tsum b_i.c_i\bigr).\bigl(d+d'\bigr) \,\bigr)
        }
        ;
        \draw[1cell] 
        (a) edge['] node[pos=.4] {p} (p)
        (p) edge['] node {\tsum \de^L_{b_i}} (b)
        (b) edge['] node[pos=.8] {\tsum 1.\de^L_{c_i}} (d)
        (a') edge[dashed] node {\dA\bigl( \de^L_{\ssum b_i.c_i} \bigr)} (d')
        (a') edge[equal] node {} (a)
        (d') edge[equal] node {} (d)
        ;
      \end{tikzpicture}
    \]

    For the object $0 \in B \otimes C$, define
    \[
      \dA \bigl( \de^L_0  \bigr) = 1_0 \cn 0+0 \to 0.
    \]
    
    \[
      \begin{tikzpicture}[x=30mm,y=15mm]
        \draw[0cell] 
        (0,0) node (a) {\tsum b_i.(c_i.d) + \tsum b'_j.(c'_j.d)}
        (a)+(2,0) node (d) {\tsum b_i.(c_i.d) + \tsum b'_j.(c'_j.d)}
        (a)+(0,.67) node (a') {
          \dA\bigl(\, \bigl(\tsum b_i.c_i\bigr).d + \bigl(\tsum b'_j.c'_j\bigr).d \,\bigr)
        }
        (d)+(0,.67) node (d') {
          \dA\bigl(\, \bigl(\tsum b_i.c_i + \tsum b'_j.c'_j\bigr).d \,\bigr)
        }
        ;
        \draw[1cell] 
        (a) edge['] node {1} (d)
        (a') edge[dashed] node {\dA\bigl(\, \de^R_{d} \,\bigr)} (d')
        (a') edge[equal] node {} (a)
        (d') edge[equal] node {} (d)
        ;
      \end{tikzpicture}
    \]

    \[
      \begin{tikzpicture}[x=30mm,y=15mm]
        \draw[0cell] 
        (0,0) node (a) {\tsum b_i.(c_i.d) + \tsum b'_j.(c'_j.d')}
        (a)+(2,0) node (d) {\tsum b'_j.(c'_j.d') + \tsum b_i.(c_i.d)}
        (a)+(0,.67) node (a') {
          \dA\bigl(\, \bigl(\tsum b_i.c_i\bigr).d + \bigl(\tsum b'_j.c'_j\bigr).d' \,\bigr)
        }
        (d)+(0,.67) node (d') {
          \dA\bigl(\, \bigl(\tsum b'_j.c'_j\bigr).d' + \bigl(\tsum b_i.c_i\bigr).d \,\bigr)
        }
        ;
        \draw[1cell] 
        (a) edge['] node {\beta} (d)
        (a') edge[dashed] node {\dA\bigl(\, \be \,\bigr)} (d')
        (a') edge[equal] node {} (a)
        (d') edge[equal] node {} (d)
        ;
      \end{tikzpicture}
    \]

    \[
      \begin{tikzpicture}[x=25mm,y=15mm]
        \draw[0cell] 
        (0,0) node (a) {0}
        (a)+(1,-1) node (b) {\tsum b_i.0}
        (a)+(2,0) node (d) {\tsum b_i.(c_i.0)}
        (a)+(0,.6) node (a') {\dA\bigl(\, 0 \,\bigr)}
        (d)+(0,.6) node (d') {\dA\bigl(\, \bigl(\tsum b_i.c_i\bigr).0 \,\bigr)}
        ;
        \draw[1cell] 
        (a) edge['] node {\tsum \ze^L_{b_i}} (b)
        (b) edge['] node {\tsum 1.\ze^L_{c_i}} (d)
        (a') edge[dashed] node {\dA\bigl(\, \ze^L_{\ssum b_i.c_i} \,\bigr)} (d')
        (a') edge[equal] node {} (a)
        (d') edge[equal] node {} (d)
        ;
      \end{tikzpicture}
    \]

    For the object $0 \in B \otimes C$, define
    \[
      \dA \bigl( \ze^L_0  \bigr) = 1_0.
    \]
    
    \[
      \begin{tikzpicture}[x=25mm,y=15mm]
        \draw[0cell] 
        (0,0) node (a) {0}
        (a)+(2,0) node (d) {0}
        (a)+(0,.6) node (a') {\dA\bigl(\, 0 \,\bigr)}
        (d)+(0,.6) node (d') {\dA\bigl(\, 0.d \,\bigr)}
        ;
        \draw[1cell] 
        (a) edge['] node {1_0} (d)
        (a') edge[dashed] node {\dA\bigl(\, \ze^R_{d} \,\bigr)} (d')
        (a') edge[equal] node {} (a)
        (d') edge[equal] node {} (d)
        ;
      \end{tikzpicture}
    \]
  \end{description}
  We show that $\dA$ is well-defined in \cref{lem:Awelldef-2.1,prop:Awelldef} below.  Then we verify 2-naturality of $\dA$ in \cref{prop:A-2nat}.
\end{defn}

The following result confirms \cref{TENnat} for morphisms adjoined in the passage from $- \times D$ to $- \otimes D$.  We state it separately because it requires significantly more sub-cases than any of the other steps in checking that $\dA$ is well-defined.
\begin{lem}\label{lem:Awelldef-2.1}
  The associativity
  \[
    \dA \cn (B \otimes C) \otimes D \to B \otimes (C \otimes D)
  \]
  preserves naturality of each morphism $\la$ in $(B \otimes C) \otimes D$ that is adjoined in the passage from $- \times D$ to $- \otimes D$.
\end{lem}
\begin{proof}
  By the definition of $\dA$, naturality with respect to any morphism containing a summation will follow automatically from the preservation of the naturality squares with respect to morphisms of the form $(f.g).h$ and $\ka.h$ for $\ka$ an adjoined morphism.
  Naturality of each $\la$ with respect to morphisms $(f.g).h$ is preserved because each $\dA(\la)$ is a composite of natural morphisms in $B \otimes (C \otimes D)$.
  Next we must check, for each $\la$, that $\dA$ preserves naturality with respect to morphisms 
  \[
    \ka.h \cn X.d \to X'.d' \inspace (B \otimes C) \otimes D,
  \]
  where $\ka$ denotes any of the five adjoined morphisms in $B \otimes C$.
  By definition of $\dA$, it suffices to check the case $h = 1_d$ for some object $d$ of $D$.
  Each of these follows from the Coherence Theorem~\ref{thm:iterated-tensor-coherence} for iterated tensor products.

  For example, $\dA$ sends the naturality diagram for $\la = \de^L$ with respect to $\ka.1 = \ze^L.1$ to the boundary of the following diagram in $B \otimes (C \otimes D)$, where the left and right vertical composites are
  \[
    \dA\bigl( \ze^L.1 + \ze^L.1 \bigr)
    \andspace
    \dA\bigl( \ze^L.1 \bigr),
  \]
  respectively.
  The bottom composite is $\dA\bigl( \de^L_b \bigr)$, and commutativity follows from coherence of the adjoined morphisms.
  \begin{equation}\label{Dnat-example}
    \begin{tikzpicture}[x=70mm,y=15mm,vcenter]
      \draw[0cell] 
      (0,0) node (a) {0+0}
      (1,0) node (b) {0}
      (a)+(0,-1) node (a') {b.0 + b.0}
      (b)+(0,-1) node (b') {b.(0)}
      (a')+(0,-1) node (a'') {b.(0.d) + b.(0.d')}
      (b')+(0,-1) node (b'') {b.(0.(d+d'))}
      (b'')+(-.5,0) node (c'') {b.(0.d + 0.d')}
      ;
      \draw[1cell] 
      (a) edge node {1_0} (b)
      (a'') edge node {\de^L} (c'')
      (c'') edge node {1.\de^L_0} (b'')
      (a) edge['] node {\ze^L + \ze^L} (a')
      (a') edge['] node {1.\ze^R + 1.\ze^R} (a'')
      (b) edge node {\ze^L} (b')
      (b') edge node {1.\ze^R} (b'')
      ;
    \end{tikzpicture}
  \end{equation}

  Confirming each of the other 24 naturality relations is similar.
  In each case one applies $\dA$ to a naturality diagram of adjoined morphisms in $(B \otimes C) \otimes D$ and obtains a formal diagram of adjoined morphisms in $B \otimes (C \otimes D)$.  One can give direct arguments for commutativity in each case, or use the Iterated Coherence Theorem~\ref{thm:iterated-tensor-coherence}.
\end{proof} 

\begin{prop}\label{prop:Awelldef}
  The assignments given in \cref{defn:A} induce a well-defined strict symmetric monoidal functor
  \[
    \dA\cn (B \otimes C) \otimes D \to B \otimes (C \otimes D).
  \]
\end{prop}
\begin{proof}
  We check that $\dA$ is well-defined in two parts.  The first part checks that $\dA$ preserves the relations imposed in the passage from $B \times C$ to $B \otimes C$.  The second part checks similarly for the passage from $- \times D$ to $- \otimes D$. Each part below can be verified with a relatively straightforward diagram chase, and the primary obstruction is one of organization. To aid the reader, we label each calculation by the axiom in \cref{defn:AtensorB} that it briefly explains.

  \medskip
  \noindent\textbf{Part 1.} (From $B \times C$ to $B \otimes C$)
  \newcommand{\BCstep}[1]{\vspace{.25pc}\noindent \textbf{\cref{#1} for $B \otimes C$:}}

  \BCstep{TENnat} Each morphism $\ka$ in $B \otimes C$ is natural with respect to morphisms $f.g$ in $B \otimes C$.  To check that $\dA(\ka.1)$ preserves this naturality, note that the definition of $\dA(\ka.1)$ above is a composite of natural morphisms in $B \otimes (C \otimes D)$.
    Therefore, preservation of the naturality square for $\ka.1$ in $(B \otimes C) \otimes D$ follows from the naturality of $\dA(\ka.1)$ in $B \otimes (C \otimes D)$.  

  \BCstep{TENsmc} The symmetric monoidal axioms for $\be.1$ are preserved because $\dA(\be.1) = \be$.

  \BCstep{TEN1} The symmetric monoidal functor axioms for
  \[
    \bigl( ((b.-).d) , \de^L.1, \ze^L.1 \bigr)
    \andspace
    \bigl( ((-.c).d) , \de^R.1, \ze^R.1 \bigr)
  \]
  are preserved by using the same axioms for the corresponding data in $B \otimes (C \otimes D)$.

  \BCstep{TEN2} The interchange condition for $\be.1$ is preserved by using the naturality of $\de^R$ and the same condition for $\be$ in $B \otimes (C \otimes D)$. 

  \BCstep{TEN3} Preservation of the condition $\ze^L_0.1 = \ze^R_0.1$ follows from naturality of $\ze^R$ and the corresponding condition in $B \otimes (C \otimes D)$.

  \BCstep{TEN6} Preservation of the left, respectively right, unit condition follows from naturality of $\de^R$, respectively $\ze^R$, and the corresponding condition in $B \otimes (C \otimes D)$.

  \medskip
  \noindent\textbf{Part 2.} (From $- \times D$ to $- \otimes D$)
  \newcommand{\Dstep}[1]{\vspace{.25pc}\noindent \textbf{\cref{#1} for $- \otimes D$:}}
  
  \Dstep{TENnat} This is \cref{lem:Awelldef-2.1}.
  
  \Dstep{TENsmc} The symmetric monoidal axioms for $\be$ are preserved because $\dA(\be) = \be$.

  \Dstep{TEN1} For an object $X=\tsum b_i.c_i$ in $B \otimes C$, the symmetric monoidal functor axioms for $\bigl( (X.-) , \de^L, \ze^L \bigr)$ follow from those for composite symmetric monoidal functors $b_i.(c_i.-)$.
  For $X = 0$, all of the morphisms involved are mapped to $1_0$.  
  The symmetric monoidal functor axioms for $\bigl( (-.d) , \de^R, \ze^R \bigr)$ are preserved because $\de^R$ and $\ze^R$ both get mapped to identities.
  
  \Dstep{TEN2} The interchange condition holds by the property \cref{eq:p-symm} of the perfect shuffle.

  \Dstep{TEN3,TEN6} These conditions are trivially preserved since $\dA(\ze^R) = \dA(\ze^L) = 1_0$.

  This completes the verification that $\dA$ preserves each of the relations imposed in the definition of $(B \otimes C) \otimes D$.  It follows from the definition that $\dA$ is strict symmetric monoidal, and this completes the proof.
\end{proof}

\begin{lem}\label{prop:A-2nat}
  The strict symmetric monoidal functors $\dA$ are the components of a 2-natural transformation
  \[
    \dA\cn (- \otimes - ) \otimes -
    \To
    - \otimes (- \otimes -)
  \]
  between functors $\permcat^3 \to \permcats$.
\end{lem}
\begin{proof}
  To verify naturality of $\dA$ with respect to symmetric monoidal functors, let
  \[
    F \cn B \to \ol{B},\quad
    G \cn C \to \ol{C}, \andspace
    H \cn D \to \ol{D}
  \]
  be symmetric monoidal functors.
  Then one computes, using the assignments given in \cref{defn:FtensorG,defn:A}, that the following diagram commutes.
  \begin{equation}\label{eq:A-2nat}
    \begin{tikzpicture}[x=50mm,y=20mm]
      \draw[0cell] 
      (0,0) node (a) {(B \otimes C) \otimes D}
      (1,0) node (b) {B \otimes (C \otimes D)}
      (0,-1) node (a') {(\ol{B} \otimes \ol{C}) \otimes \ol{D}}
      (1,-1) node (b') {\ol{B} \otimes (\ol{C} \otimes \ol{D})}
      ;
      \draw[1cell] 
      (a) edge node {\dA} (b)
      (a') edge node {\dA} (b')
      (a) edge['] node {(F \otimes G) \otimes H} (a')
      (b) edge node {F \otimes (G \otimes H)} (b')
      ;
    \end{tikzpicture}
  \end{equation} 
  For example, on the adjoined morphism $\de^L_b.1$ in $(B \otimes C) \otimes D$, the left bottom composite above is
  \begin{align*}
    \de^L_b.1 & \mapsto (1.G_2).1 \circ \de^L_{Fb}.1\\
    & \mapsto 1.(G_2.1) \circ \bigl( 1.\de^R_{Hd} \circ \de^L_{Fb} \bigr).
  \end{align*}
  On the other hand, the top right composite above is
  \begin{align*}
    \de^L_b.1 & \mapsto 1.\de^R_d \circ \de^L_b\\
            & \mapsto \bigl( 1.(G_2.1) \circ 1.\de^R_{Hd} \bigr)
              \circ \bigl( 1.(G \otimes H)_2 \circ \de^L_{Fb} \bigr).
  \end{align*}
  These assignments are equal because $G \otimes H$ is strict monoidal, and hence $(G \otimes H)_2$ is an identity.  For the rest of the objects and morphisms in $(B \otimes C) \otimes D$, verifications that the two composites around \cref{eq:A-2nat} agree are similar.  One uses strictness of $F \otimes G$ for the left bottom composite, and strictness of $G \otimes H$ for the top right composite.

  For monoidal transformations
  \[
    F \To \ol{F}, \quad G \To \ol{G}, \andspace H \To \ol{H},
  \]
  one verifies that the corresponding whiskerings by $\dA$ in \cref{eq:A-2nat} are equal componentwise.
  This completes the proof that $\dA$ is 2-natural.
\end{proof}

\section{The adjoint inverse, \texorpdfstring{$\dAdot$}{Adot}}\label{sec:Adot}

\begin{defn}\label{defn:Adot}
  We define a strict symmetric monoidal functor
  \[
    \dAdot \cn B \otimes (C \otimes D) \to (B \otimes C) \otimes D
  \]
  as the following composite.
  \begin{equation}\label{eq:Adot}
    \begin{tikzpicture}[x=30mm,y=20mm,vcenter]
      \draw[0cell] 
      (0,0) node (a) {B \otimes (C \otimes D)}
      (a)+(-60:1) node (b) {(C \otimes D) \otimes B}
      (b)+(60:1) node (c) {(D \otimes C) \otimes B}
      (c)+(1,0) node (d) {D \otimes (C \otimes B)}
      (d)+(-60:1) node (e) {D \otimes (B \otimes C)}
      (e)+(60:1) node (f) {(B \otimes C) \otimes D}
      ;
      \draw[1cell] 
      (a) edge['] node[pos=.4] {\dB} (b)
      (b) edge['] node[pos=.6] {\dB \otimes 1} (c)
      (c) edge['] node {\dA} (d)
      (d) edge['] node[pos=.4] {1 \otimes \dB} (e)
      (e) edge['] node[pos=.6] {\dB} (f)
      ;
      \draw[1cell]
      (a) [rounded corners=5pt,dashed] |- ($(c)+(0,.45)$)
      -- node {\dAdot} ($(d)+(0,.45)$) -| (f)
      ;
    \end{tikzpicture}
  \end{equation}
  Since both $\dA$ and $\dB$ are 2-natural, so is the composite $\dAdot$.
\end{defn}

It will be useful to have the following explicit description of the $\dAdot$ as an assignment on objects and morphisms.
As with $\dA$, we describe $\dAdot$ on morphisms $f.\ka$ and $\la$, and extend to sums so that $\dAdot$ is strictly monoidal.
\begin{description}
\item[Objects]
  The assignment on objects of $B \otimes (C \otimes D)$ is given by
  \begin{align*}
    \dAdot (0) & = 0, \\
    \dAdot (b.0) & = 0, \\
    \dAdot \bigl(\, b.\bigl( \tsum c_i.d_i  \bigr)  \,\bigr)
               & = \tsum (b.c_i).d_i,
  \end{align*}
  and extended to sums in $B \otimes (C \otimes D)$ to be strictly monoidal.
\item[Morphisms $f.\ka$]
  Suppose $\ka$ is a morphism in $C \otimes D$ and $f$ is a morphism in $B$.
  If $\ka$ is itself a sum $\tsum \ka_i$, then we define
  \[
    \dA \bigl( f.\bigl( \tsum \ka_i  \bigr)  \bigr) = \sum \dA \bigl(f.\ka_i \bigr).
  \]
  We now treat each possible type of simple or adjoined morphism $\ka$ in $C \otimes D$.
  \begin{itemize}
  \item If $\ka = g.h$, then we define
    \[
      \dAdot \bigl( f.(g.h) \bigr) = (f.g).h.
    \]
  \item If $\ka = 1_0$, then we define
    \[
      \dAdot (f . 1_0) = 1_0.
    \]
  \item If $\ka$ is one of the adjoined morphisms in $C \otimes D$,
    then we define $\dAdot(f.\ka)$ via the indicated composite below.
    As with $\dA$, there are alternative equivalent definitions using $(f.1).1$ in different positions by naturality.
  \end{itemize}

  \[
    \begin{tikzpicture}[x=25mm,y=15mm]
      \draw[0cell] 
      (0,0) node (a) {(b.c).d + (b.c).d'}
      (a)+(1,-1) node (b) {(b.c).(d+d')}
      (a)+(2,0) node (d) {(b'.c).(d+d')}
      (a)+(0,.6) node (a') {\dAdot\bigl(\, b.(c.d + c.d') \,\bigr)}
      (d)+(0,.6) node (d') {\dAdot\bigl(\, b'.(c.(d+d')) \,\bigr)}
      ;
      \draw[1cell] 
      (a) edge['] node {\de^L_{b.c}} (b)
      (b) edge['] node {(f.1).1} (d)
      (a') edge[dashed] node {\dAdot\bigl(\,f.\de^L_{c}\,\bigr)} (d')
      (a') edge[equal] node {} (a)
      (d') edge[equal] node {} (d)
      ;
    \end{tikzpicture}
  \]

  \[
    \begin{tikzpicture}[x=25mm,y=15mm]
      \draw[0cell] 
      (0,0) node (a) {(b.c).d + (b.c').d}
      (.75,-1) node (b) {\bigl( b.c + b.c' \bigr).d}
      (2.25,-1) node (c) {\bigl( b.(c+c')  \bigr).d}
      (3,0) node (d) {\bigl( b'.(c+c')  \bigr).d}
      (a)+(0,.6) node (a') {\dAdot\bigl(\, b.(c.d + c'.d) \,\bigr)}
      (d)+(0,.6) node (d') {\dAdot\bigl(\, b'.((c + c').d) \,\bigr)}
      ;
      \draw[1cell] 
      (a) edge['] node {\de^R_{d}} (b)
      (b) edge['] node {\de^L_{b}.1} (c)
      (c) edge['] node {(f.1).1} (d)
      (a') edge[dashed] node {\dAdot\bigl(\,f.\de^R_{d}\,\bigr)} (d')
      (a') edge[equal] node {} (a)
      (d') edge[equal] node {} (d)
      ;
    \end{tikzpicture}
  \]
  
  \[
    \begin{tikzpicture}[x=25mm,y=15mm]
      \draw[0cell] 
      (0,0) node (a) {(b.c).d + (b.c').d'}
      (a)+(1,-1) node (b) {(b.c').d' + (b.c).d}
      (a)+(2,0) node (d) {(b'.c').d' + (b'.c).d}
      (a)+(0,.6) node (a') {\dAdot\bigl(\, b.(c.d + c'.d') \,\bigr)}
      (d)+(0,.6) node (d') {\dAdot\bigl(\, b'.(c'.d' + c.d) \,\bigr)}
      ;
      \draw[1cell] 
      (a) edge['] node {\be} (b)
      (b) edge['] node {(f.1).1 + (f.1).1} (d)
      (a') edge[dashed] node {\dAdot\bigl(\,f.\be\,\bigr)} (d')
      (a') edge[equal] node {} (a)
      (d') edge[equal] node {} (d)
      ;
    \end{tikzpicture}
  \]

  \[
    \begin{tikzpicture}[x=25mm,y=15mm]
      \draw[0cell] 
      (0,0) node (a) {0}
      (a)+(1,-1) node (b) {(b.c).0}
      (a)+(2,0) node (d) {(b'.c).0}
      (a)+(0,.6) node (a') {\dAdot\bigl(\, b.0 \,\bigr)}
      (d)+(0,.6) node (d') {\dAdot\bigl(\, b'.(c.0) \,\bigr)}
      ;
      \draw[1cell] 
      (a) edge['] node {\ze^L_{b.c}} (b)
      (b) edge['] node {(f.1).1} (d)
      (a') edge[dashed] node {\dAdot\bigl(\,f.\ze^L_{c}\,\bigr)} (d')
      (a') edge[equal] node {} (a)
      (d') edge[equal] node {} (d)
      ;
    \end{tikzpicture}
  \]

  \[
    \begin{tikzpicture}[x=25mm,y=15mm]
      \draw[0cell] 
      (0,0) node (a) {0}
      (a)+(.75,-1) node (b) {0.d}
      (b)+(1.5,0) node (c) {(b.0).d}
      (a)+(3,0) node (d) {(b'.0).d}
      (a)+(0,.6) node (a') {\dAdot\bigl(\, b.0 \,\bigr)}
      (d)+(0,.6) node (d') {\dAdot\bigl(\, b'.(0.d) \,\bigr)}
      ;
      \draw[1cell] 
      (a) edge['] node {\ze^R_{d}} (b)
      (b) edge['] node {\ze^L_{b}.1} (c)
      (c) edge['] node {(f.1).1} (d)
      (a') edge[dashed] node {\dAdot\bigl(\,f.\ze^R_{d}\,\bigr)} (d')
      (a') edge[equal] node {} (a)
      (d') edge[equal] node {} (d)
      ;
    \end{tikzpicture}
  \]
\item[Adjoined morphisms] For morphisms $\la$ adjoined in the passage from $B \times -$ to $B \otimes -$, $\dAdot(\la)$ is the indicated composite below.

  \[
    \begin{tikzpicture}[x=30mm,y=15mm]
      \draw[0cell] 
      (0,0) node (a) {\tsum (b.c_i).d_i + \tsum (b.c'_j).d'_j}
      (a)+(2,0) node (d) {\tsum (b.c_i).d_i + \tsum (b.c'_j).d'_j}
      (a)+(0,.67) node (a') {
        \dAdot\bigl(\, b.\bigl(\tsum c_i.d_i\bigr) + b.\bigl(\tsum c'_j.d'_j\bigr) \,\bigr)
      }
      (d)+(0,.67) node (d') {
        \dAdot\bigl(\, b.\bigl(\tsum c_i.d_i + \tsum c'_j.d'_j\bigr) \,\bigr)
      }
      ;
      \draw[1cell] 
      (a) edge['] node {1} (d)
      (a') edge[dashed] node {\dAdot\bigl(\, \de^L_{b} \,\bigr)} (d')
      (a') edge[equal] node {} (a)
      (d') edge[equal] node {} (d)
      ;
    \end{tikzpicture}
  \]

  \[
    \begin{tikzpicture}[x=30mm,y=15mm]
      \draw[0cell] 
      (0,0) node (a) {\tsum (b.c_i).d_i + \tsum (b'.c_i).d_i}
      (a)+(.25,-1) node (p) {\tsum \bigl( (b.c_i).d_i + (b'.c_i).d_i \bigr)}
      (p)+(1.5,0) node (b) {\tsum (b.c_i + b'.c_i).d_i}
      (a)+(2,0) node (d) {\tsum ((b+b').c_i).d_i}
      (a)+(0,.6) node (a') {
        \dAdot\bigl(\, b.\bigl(\tsum c_i.d_i\bigr) + b'.\bigl(\tsum c_i.d_i\bigr) \,\bigr)
      }
      (d)+(0,.6) node (d') {
        \dAdot\bigl(\, \bigl(b+b'\bigr).\bigl(\tsum c_i.d_i\bigr) \,\bigr)
      }
      ;
      \draw[1cell] 
      (a) edge['] node[pos=.4] {p} (p)
      (p) edge['] node {\tsum \de^R_{d_i}} (b)
      (b) edge['] node[pos=.8] {\tsum \de^R_{c_i}.1} (d)
      (a') edge[dashed] node {\dAdot\bigl( \de^R_{\ssum c_i.d_i} \bigr)} (d')
      (a') edge[equal] node {} (a)
      (d') edge[equal] node {} (d)
      ;
    \end{tikzpicture}
  \]

  For the object $0 \in C \otimes D$,
  \[
    \dAdot \bigl( \de^R_0  \bigr) = 1_0 \cn 0+0 \to 0.
  \]

  \[
    \begin{tikzpicture}[x=30mm,y=15mm]
      \draw[0cell] 
      (0,0) node (a) {\tsum (b.c_i).d_i + \tsum (b'.c'_j).d'_j}
      (a)+(2,0) node (d) {\tsum (b'.c'_j).d'_j + \tsum (b.c_i).d_i}
      (a)+(0,.67) node (a') {
        \dAdot\bigl(\, b.\bigl(\tsum c_i.d_i\bigr) + b'.\bigl(\tsum c'_j.d'_j\bigr) \,\bigr)
      }
      (d)+(0,.67) node (d') {
        \dAdot\bigl(\, b'.\bigl(\tsum c'_j.d'_j\bigr) + b.\bigl(\tsum c_i.d_i\bigr) \,\bigr)
      }
      ;
      \draw[1cell] 
      (a) edge['] node {\beta} (d)
      (a') edge[dashed] node {\dAdot\bigl(\, \be \,\bigr)} (d')
      (a') edge[equal] node {} (a)
      (d') edge[equal] node {} (d)
      ;
    \end{tikzpicture}
  \]
  
  \[
    \begin{tikzpicture}[x=25mm,y=15mm]
      \draw[0cell] 
      (0,0) node (a) {0}
      (a)+(2,0) node (d) {0}
      (a)+(0,.6) node (a') {\dAdot\bigl(\, 0 \,\bigr)}
      (d)+(0,.6) node (d') {\dAdot\bigl(\, b.0 \,\bigr)}
      ;
      \draw[1cell] 
      (a) edge['] node {1_0} (d)
      (a') edge[dashed] node {\dAdot\bigl(\, \ze^L_{b} \,\bigr)} (d')
      (a') edge[equal] node {} (a)
      (d') edge[equal] node {} (d)
      ;
    \end{tikzpicture}
  \]

  \[
    \begin{tikzpicture}[x=25mm,y=15mm]
      \draw[0cell] 
      (0,0) node (a) {0}
      (a)+(1,-1) node (b) {\tsum 0.d_i}
      (a)+(2,0) node (d) {\tsum (0.c_i).d_i}
      (a)+(0,.6) node (a') {\dAdot\bigl(\, 0 \,\bigr)}
      (d)+(0,.6) node (d') {\dAdot\bigl(\, 0.\bigl(\tsum c_i.d_i\bigr) \,\bigr)}
      ;
      \draw[1cell] 
      (a) edge['] node {\tsum \ze^R_{d_i}} (b)
      (b) edge['] node {\tsum \ze^R_{c_i}.1} (d)
      (a') edge[dashed] node {\dAdot\bigl(\, \ze^R_{\ssum c_i.d_i} \,\bigr)} (d')
      (a') edge[equal] node {} (a)
      (d') edge[equal] node {} (d)
      ;
    \end{tikzpicture}
  \]

  For the object $0 \in C \otimes D$, 
  \[
    \dAdot \bigl( \ze^R_0  \bigr) = 1_0.
  \]
\end{description}
This completes the description of $\dAdot$.

\begin{prop}\label{prop:AAdot}
  The strict symmetric monoidal functors $\dA$ and $\dAdot$ are part of an adjoint equivalence in $\permcats$
  \[
    \dA \cn (B \otimes C) \otimes D \lradjequiv B \otimes (C \otimes D) \bacn \dAdot.
  \]
\end{prop}
\begin{proof}
  For an object $b.\bigl(\tsum c_i.d_i\bigr)$ in $B \otimes (C \otimes D)$, we have
  \[
    \dA \dAdot \bigl(\, b.\bigl(\tsum c_i.d_i\bigr) \,\bigr)
    = \dA \bigl(\, \tsum\, (b.c_i).d_i \,\bigr)
    = \tsum b.(c_i.d_i).
  \]
  Also,
  \[
   \dA \dAdot \bigl( 0 \bigr)=0 \quad \text{and} \quad \dA \dAdot \bigl( b.0 \bigr) = 0.
  \]
  Define a monoidal natural isomorphism
  \[
    \epza\cn \dA \dAdot \To \Id
  \]
  with components
  \begin{align}
    \epza_{b.(\ssum c_i.d_i)}
    & = \de^L \cn \tsum b.(c_i.d_i) \to
      b.\bigl(\tsum c_i.d_i \bigr) \andspace\label{eq:epza} \\
    \epza_{b.0} & = \ze^L \cn 0 \to b.0\nonumber
  \end{align}
  where, as in \cref{convention:F2}, we let $\de^L$ denote any composite of the monoidal constraints $\de^L_b$ for $b.(-)$.
  The symmetric monoidal functor axioms for $b.(-)$ are used to verify naturality of $\epza$.

  Similarly, define a monoidal natural isomorphism
  \[
    \epz' \cn \dAdot \dA \To \Id
  \]
  with components
  \begin{align}
    \epz'_{(\ssum b_i.c_i).d}
    & = \de^R \cn \tsum (b_i.c_i).d \to
      \bigl( \tsum b_i.c_i  \bigr).d \andspace\label{eq:etaa} \\
    \epz'_{0.d} & = \ze^R \cn 0 \to 0.d\nonumber
  \end{align}
  where, as above, $\de^R$ denotes any composite of the monoidal constraints $\de^R_d$ for $(-).d$.
  The symmetric monoidal functor axioms for $(-).d$ are used to verify naturality of $\epz'$.

  The modification axiom for $\epza$ holds because, for symmetric monoidal functors
  \[
    P \cn B \to \ol{B}, \quad
    Q \cn C \to \ol{C}, \andspace
    R \cn D \to \ol{D},
  \]
  the product $Q \otimes R$ is strict symmetric monoidal and hence
  \[
    \bigl(P \otimes (Q \otimes R)\bigr) * \epz = \epz * \bigl( P \otimes (Q \otimes R)\bigr).
  \]
  A similar check, using strictness of $P \otimes Q$, shows that $\epz'$ is also a modification.

  The unit/counit triangle axioms for $\etaa = \bigl(\epz'\bigr)^\inv$ and $\epza$ follow because, checking the relevant formulas for $\dA$ and $\dAdot$, each of the following whiskerings is an identity:
  \[
    \epza * \dA, \quad
    \dAdot * \epza, \quad
    \epz' * \dAdot, \andspace
    \dA * \epz'.
  \]
  Of the whiskerings above, the first and third are identities because each of $\epza$ and $\epz'$ is the identity on a monomial $b.(c.d)$, respectively $(b.c).d$.  The second and fourth are identities because $\dAdot(\de^L) = 1$ and $\dA(\de^R) = 1$.
\end{proof}

\section{2-Dimensional data and axioms}
\label{sec:2Ddata}

In this section we discuss the 2-dimensional data and corresponding axioms to show that $ \bigl( \permcat,\otimes \bigr) $ is a symmetric monoidal 2-category.
First, we describe one additional nontrivial transformation, $\la$.
Then, we identify certain commutative diagrams that simplify the symmetric monoidal structure of $\bigl( \permcat, \otimes \bigr)$.
The main result is \cref{thm:permcat-smb} below.

Throughout, we let $X$, $Y$, $Z$, and $W$ denote permutative categories.
In this section we often write the tensor product as juxtaposition, so
\[
   XY = X \otimes Y.
\]
Recall the monoidal unit $S$ from \cref{defn:S}.
\begin{defn}\label{defn:lambda}
  Let 
  \[
    \begin{tikzpicture}[x=45mm,y=20mm,scale=.7]
      \draw[0cell=.8] 
      (0,0) node (0) {\bigl(SX\bigr)Y}
      (1,0) node (1) {XY}
      (.5,-1) node (2) {S\bigl(XY\bigr)};
      \draw[1cell=.8] 
      (0) edge node {\dL\,1} (1)
      (0) edge[swap] node {\dA} (2)
      (2) edge[swap] node {\dL} (1);
      \draw[2cell]
      (.5,-.5) node[rotate=270, 2label={above,\,\la}] {\Rightarrow}
      ;
    \end{tikzpicture}
  \]
  denote the monoidal transformation with the following components.  For $y \in Y$ and $\tsum m_i.x_i \in S \otimes X$, 
  \begin{align*}
    \la_0 & = 1_0,\\
    \la_{0.y} & = \ze^{-R}_y \cn 0.y \to 0, \andspace\\
    \la_{(\ssum m_i.x_i).y} & = \de^{-R}_y \cn \bigl( \tsum m_i x_i \bigr).y \to \tsum m_i(x_i.y),
  \end{align*}
  where $\ze^{-R}$ and $\de^{-R}$ denote the inverses of $\ze^R$ and $\de^{R}$, or any composite of such, in the case of $\de^{-R}$ (see \cref{convention:F2}).
  Then $\la$ is extended to sums in $(S X) Y$ so that it is monoidal.

  By 2-functoriality of $\otimes$ and naturality of $\de^{-R}$, it suffices to verify naturality of $\la$ with respect to the following morphisms in $(SX)Y$:
  \begin{itemize}
  \item simple morphisms $(\si.f).h$, 
  \item simple morphisms $\ka.1$, where $\ka$ is an adjoined morphism of $SX$, and
  \item morphisms $\ka'$ adjoined in the passage from $- \times Y$ to $- \otimes Y$.
  \end{itemize}
  In the first case, naturality follows from naturality of the adjoined morphisms $\de^{-R}$ defining $\la$.
  In the other two cases, each naturality square is a formal diagram of adjoined morphisms in $XY$ whose two boundary composites have the same underlying permutation.
  Hence, naturality of $\la$ for the second two cases follows from the Coherence Theorem~\ref{thm:tensor-coherence}.
\end{defn}

\begin{lem}\label{lem:pi-mu-rho-triv}
  The following diagrams commute in $\permcat$.
  \[
    \begin{tikzpicture}[x=20mm,y=20mm,scale=.6,vcenter]
      \draw[0cell=.8] 
      (0,0) node (0) {\bigl(\bigl(XY\bigr)Z\bigr)W}
      (1,1) node (1) {\bigl(X\bigl(YZ\bigr)\bigr)W}
      (3,1) node (2) {X\bigl(\bigl(YZ\bigr)W\bigr)}
      (4,0) node (3) {X\bigl(Y\bigl(ZW\bigr)\bigr)}
      (2,-1) node (4) {\bigl(XY\bigr)\bigl(ZW\bigr)};
      \draw[1cell=.8] 
      (0) edge node {\dA \,1} (1)
      (1) edge node {\dA} (2)
      (2) edge node {1\,\dA} (3)
      (0) edge[swap] node {\dA} (4)
      (4) edge[swap] node {\dA} (3);
    \end{tikzpicture}
    \quad
    \begin{tikzpicture}[x=30mm,y=20mm,scale=.7,vcenter]
      \draw[0cell=.8] 
      (0,0) node (0) {\bigl(XS\bigr)Y}
      (1,0) node (1) {X\bigl(SY\bigr)}
      (-.5,-1) node (2) {XY}
      (1.5,-1) node (3) {XY};
      \draw[1cell=.8] 
      (0) edge node {\dA} (1)
      (2) edge node {\dRdot\,1} (0)
      (2) edge[swap] node {1} (3)
      (1) edge node {1 \, \dL} (3);
    \end{tikzpicture}
  \]
  \[
    \qquad
    \begin{tikzpicture}[x=45mm,y=20mm,scale=.7]
      \draw[0cell=.8] 
      (0,0) node (0) {XY}
      (1,0) node (1) {X\bigl(YS\bigr)}
      (.5,-1) node (2) {\bigl(XY\bigr)S};
      \draw[1cell=.8] 
      (0) edge node {1\, \dRdot} (1)
      (0) edge[swap] node {\dRdot} (2)
      (2) edge[swap] node {\dA} (1);
    \end{tikzpicture}
  \]
\end{lem}
\begin{proof}
  For each diagram, one verifies commutativity with respect to each of the objects and morphisms.
  For the pentagon, one uses the assignments described in \cref{defn:A} for $\dA$ and \cref{defn:FtensorG} for a product of functors.
  On objects, both composites around the pentagon give
  \[
    0 \mapsto 0, \quad
    0.w \mapsto 0, \quad
    (0.z).w \mapsto 0, \andspace
    ((x.y).z).w \mapsto x.(y.(z.w))
  \]
  for $x \in X$, $y \in Y$, $z \in Z$, and $w \in W$.
  Since both composites are strictly monoidal, this completes the verification on objects.

  Checking the pentagon on simple morphisms is the same as for simple objects.
  Then one must check each of fifteen types of adjoined morphism, five for each tensor product in $((XY)Z)W$.  For example, consider
  \[
    \de^L_{\ssum x_i.y_i}.1 \cn
    \Bigl( \bigl(\tsum x_i.y_i\bigr).z + \bigl(\tsum x_i.y_i\bigr).z' \Bigr).w
    \to
    \Bigl( \bigl(\tsum x_i.y_i\bigr).(z + z') \Bigr).w
  \]
  in $\Bigl( \bigl( XY \bigr) Z \Bigr) W$.
  The definitions of $\dA$, $1 \otimes \dA$, and $\dA \otimes 1$ give the following assignments around the top of the pentagon, where $p$ denotes each of the relevant perfect shuffles:
  \begin{align*}
    \dA 1 \cn \de^L_{\ssum x_i.y_i}.1
    & \mapsto
      \bigl( \tsum 1.\de^L_{y_i} \bigr).1 \;\circ\;
      \bigl( \tsum \de^L_{x_i} \bigr).1 \;\circ\;
      p.1\\
    \dA \cn \phantom{\de^L_{\ssum x_i.y_i}.1}
    & \mapsto
      \tsum 1.(\de^L_{y_i} .1) \;\circ\; 
      \tsum (1.\de^R_w) \;\circ\;
      \tsum \de^L_{x_i} \;\circ\;
      p\\
    1 \dA \cn \phantom{\de^L_{\ssum x_i.y_i}.1}
    & \mapsto
      \tsum (1.(1.\de^R_w)) \;\circ\;
      \tsum 1.\de^L_{y_i} \;\circ\; 
      \tsum \de^L_{x_i} \;\circ\;
      p.
  \end{align*}
  Likewise, we have the following assignments around the bottom of the pentagon:
  \begin{align*}
    \dA \cn \de^L_{\ssum x_i.y_i}.1
    & \mapsto
      1.\de^r_w \;\circ\;
      \de^L_{\ssum x_i.y_i}.1\\
    \dA \cn \phantom{\de^L_{\ssum x_i.y_i}.1}
    & \mapsto
      \tsum (1.(1.\de^R_w)) \;\circ\;
      \tsum 1.\de^L_{y_i} \;\circ\; 
      \tsum \de^L_{x_i} \;\circ\;
      p.
  \end{align*}
  Verification that the pentagon commutes on each of the other types of adjoined morphism is similar.

  For the other two diagrams, one verifies commutativity on each of the objects and morphisms of $XY$.  Recall from \cref{prop:Requiv} that $\dRdot = (-).1$: right multiplication by the natural number $1 \in S$.  Although $\dRdot$ is strong symmetric monoidal, the tensor products $\dRdot \otimes 1$ and $1 \otimes \dRdot$ are strict symmetric monoidal.  Likewise, the composite $\dA \circ \dRdot$ is strict symmetric monoidal because $\dA$ sends the morphisms $\de^R$ and $\ze^R$ to identities.
\end{proof}

\begin{lem}\label{lem:R-syl-triv}
  The following diagrams commute in $\permcat$.
  \[
    \begin{tikzpicture}[x=25mm,y=16mm,scale=.7]
      \draw[0cell=.8] 
      (.3,0) node (0) {\bigl(XY\bigr)Z}
      (1,1) node (1) {\bigl(YX\bigr)Z}
      (2,1) node (2) {Y\bigl(XZ\bigr)}
      (2.7,0) node (3) {Y\bigl(ZX\bigr)}
      (1,-1) node (4) {X\bigl(YZ\bigr)}
      (2,-1) node (5) {\bigl(YZ\bigr)X};
      \draw[1cell=.8] 
      (0) edge node {\dB \,1} (1)
      (1) edge node {\dA} (2)
      (2) edge node {1\,\dB} (3)
      (0) edge[swap] node {\dA} (4)
      (4) edge[swap] node {\dB} (5)
      (5) edge[swap] node {\dA} (3);
      \begin{scope}[shift={(3.5,0)}]
        \draw[0cell=.8] 
        (.3,0) node (0) {X\bigl(YZ\bigr)}
        (1,1) node (1) {X\bigl(ZY\bigr)}
        (2,1) node (2) {\bigl(XZ\bigr)Y}
        (2.7,0) node (3) {\bigl(ZX\bigr)Y}
        (1,-1) node (4) {\bigl(XY\bigr)Z}
        (2,-1) node (5) {Z\bigl(XY\bigr)};
        \draw[1cell=.8] 
        (0) edge node {1\,\dB} (1)
        (1) edge node {\dAdot} (2)
        (2) edge node {\dB\,1} (3)
        (0) edge[swap] node {\dAdot} (4)
        (4) edge[swap] node {\dB} (5)
        (5) edge[swap] node {\dAdot} (3);
      \end{scope} 
      \begin{scope}[shift={(3.25,-1.5)}]
        \draw[0cell=.8] 
        (-1,-1) node (0) {XY}
        (1,-1) node (1) {XY}
        (0,0) node (2) {YX};
        \draw[1cell=.8] 
        (0) edge[swap] node {1} (1)
        (0) edge node {\dB} (2)
        (2) edge node {\dB} (1);
      \end{scope}
    \end{tikzpicture}
  \]
\end{lem}
\begin{proof}
  The triangle $\dB^2 = 1$ is immediate from the definition.
  One verifies that the left hexagon commutes on each of the objects, simple morphisms, and adjoined morphisms of $\bigl(XY\bigr)Z$, as in the proof of \cref{lem:pi-mu-rho-triv}.  For example, consider
  \[
    \de^L_x . 1 \cn (x.y + x.y').z \to (x.(y+y')).z \inspace \bigl(XY\bigr)Z.
  \]
  The assignments around the top of the left hexagon are
  \begin{align*}
    \dB 1 \cn \de^L_{x}.1
    & \mapsto
      \de^R_x . 1 \\
    \dA \cn \phantom{\de^L_{x}.1}
    & \mapsto
      \de^R_{x.z} \\
    1 \dB \cn \phantom{\de^L_{x}.1}
    & \mapsto
      \bigl(1_{y+y'}.1_{z.x}\bigr) \circ \de^R_{z.x} = \de^R_{z.x},
  \end{align*}
  where the final assignment uses the formula for $\bigl(F \otimes G\bigr)\bigl(\de^R\bigr)$ with $F = \Id$ and $G = \dB$.
  The assignments around the bottom of the left hexagon are
  \begin{align*}
    \dA \cn \de^L_{x}.1
    & \mapsto
      1.\de^R_z \;\circ\;
      \de^L_y \\
    \dB \cn \phantom{\de^L_{x}.1}
    & \mapsto
      \de^R_z.1 \;\circ\;
      \de^R_y \\
    \dA \cn \phantom{\de^L_{x}.1}
    & \mapsto
      \de^R_{z.x} \;\circ\;
      1 = \de^R_{z.x}.
  \end{align*}
  The other verifications are similar.

  For the right hexagon, recall that $\dAdot$ is defined as the composite \cref{eq:Adot}.
  Using this composite, together with naturality of $\dB$ and the equality $\dB^2 = 1$, commutativity of the right hexagon follows from that of the left.
\end{proof}

\begin{thm}\label{thm:permcat-smb}
  The data
  \[
    \Bigl( \permcat, \otimes, S, \dA, \dL, \dR, \dB, \la \Bigr)
  \]
  defines a symmetric monoidal 2-category in which
  \begin{itemize}
  \item $\otimes$ is a 2-functor,
  \item each of $\pi$, $\mu$, $\rho$ is an identity, and
  \item each of $R_{- \mid --}$, $R_{-- \mid -}$, and $\syl$ is an identity.
  \end{itemize}
\end{thm}
\begin{proof}
  The 2-functoriality of the tensor product is \cref{prop:tensor-functor}.
  \cref{lem:pi-mu-rho-triv,lem:R-syl-triv} show that we can choose the indicated data to be trivial.
  For the remainder of this proof we check the seven axioms listed in \cref{sec:smb-defn}.

  For the left normalization axiom \cref{eq:left-norm-axiom}, the components of $\la$ are given by $\de^{-R}$ and $\ze^{-R}$.
  For the rest of the axioms, the only nontrivial cells are given by various mates with respect to the adjoint equivalence $\bigl(\dA,\dAdot\bigr)$.
  For example, \cref{eq:pi3} shows one mate of $\pi$.
  Thus, the relevant mates result in cells whose components are those of $\etaa=\bigl(\epz'\bigr)^\inv$ \cref{eq:etaa}, $\epza$ \cref{eq:epza}, or their inverses.
  Since these components are all given by instances of the adjoined morphisms $\de$ and $\ze$, it suffices by the Iterated Coherence Theorem~\ref{thm:iterated-tensor-coherence} to verify that both sides of each axiom have the same underlying permutation.

  For each of the left normalization axiom \cref{eq:left-norm-axiom}; the three crossing axioms \cref{eq:31-cross-axiom,eq:13-cross-axiom,eq:22-cross-axiom}; and the Yang-Baxter axiom \cref{eq:YB-axiom}, the two 1-cell composites around each side of each axiom are equal on objects.
  Thus, each side of each of these axioms is a monoidal transformation whose components are given by instances of $\de^{\pm 1}$ and $\ze^{\pm 1}$, and whose source and target are equal.
  Therefore, by \cref{thm:iterated-tensor-coherence}, each side of each of these first five axioms is an identity.
  
  For the (2,1)-syllepsis axiom \cref{eq:21-syl-axiom}, the source 1-cell composite, for each side, is equal to $\dA\dAdot$ on objects.  The target of each side is the identity.
  Therefore, the two sides of \cref{eq:21-syl-axiom} are nontrivial monoidal  transformations whose underlying permutations are identities.
  Similarly, the two sides of the (1,2)-syllepsis axiom \cref{eq:12-syl-axiom} have source that is equal to $\dAdot\dA$ on objects and target that is the identity.
  Therefore the syllepsis axioms also hold by \cref{thm:iterated-tensor-coherence}.
  This completes the proof.
\end{proof}

\section{The free 2-functor \texorpdfstring{$\bigP$}{P} is symmetric monoidal}\label{sec:coherence2}

In this section we complete the proof that the free permutative category construction
\[
  \bigP \cn (\Cat, \times) \to (\permcat, \otimes)
\]
is symmetric monoidal as a 2-functor between symmetric monoidal 2-categories.
Its unit constraint is the identity, and its monoidal constraint is the symmetric monoidal equivalence
\[
\Phi\cn \bigP X \otimes \bigP Y \to \bigP \bigl( X \times Y \bigr)
\]
of \cref{thm:free-tensor-adj-equiv}.
The main result is \cref{thm:P-sm2fun} below.
Throughout, we let $X$, $Y$ and $Z$ denote categories.

\begin{lem}\label{lem:P-2fun-om-gaL-gaR1}
  The following diagrams commute in $\permcat$, for each triple $X,Y,Z \in \Cat$.  
  Each unlabeled arrow is induced by the symmetric monoidal data of the cartesian product in $\Cat$.
  \begin{equation}\label{eq:om-gaL-gaR1}
    \begin{tikzpicture}[x=32mm,y=12mm,vcenter]
      \draw[0cell=.8] 
      (0,0) node (a) {\bigl( \bigP X \otimes \bigP Y \bigr) \otimes \bigP Z}
      (1,0) node (b) {\bigP X \otimes \bigl( \bigP Y \otimes \bigP Z \bigr)}
      (1,-1) node (c) {\bigP X \otimes \bigP \bigl( Y \times Z \bigr)}
      (1,-2) node (d) {\bigP \bigl( X \times \bigl( Y \times Z \bigr)\bigr)}
      (0,-1) node (b') {\bigP \bigl( X \times Y \bigr) \otimes \bigP Z}
      (0,-2) node (c') {\bigP \bigl( \bigl( X \times Y \bigr) \times Z\bigr)}
      ;
      \draw[1cell=.8] 
      (a) edge node {\dA} (b)
      (b) edge node {1 \otimes \Phi} (c)
      (c) edge node {\Phi} (d)
      (a) edge['] node {\Phi \otimes 1} (b')
      (b') edge['] node {\Phi} (c')
      (c') edge node {} (d)
      ;
    \end{tikzpicture}
    \qquad
    \begin{tikzpicture}[x=22mm,y=10mm,vcenter]
      \draw[0cell=.8] 
      (0,0) node (la) {S \otimes \bigP X}
      (1,0) node (lb) {\bigP X}
      (0,-1) node (la') {\bigP * \otimes \bigP X}
      (1,-1) node (lb') {\bigP\bigl(* \times X \bigr)}
      (la')+(0,-1) node (rda) {\bigP X \otimes S}
      (rda)+(1,0) node (rdb) {\bigP X}
      (rda)+(0,-1) node (rda') {\bigP X \otimes \bigP *}
      (rda')+(1,0) node (rdb') {\bigP \bigl( X \times * \bigr)}
      ;
      \draw[1cell=.8] 
      (la) edge node {\dL} (lb)
      (la') edge node {\Phi} (lb')
      (la) edge[equal] node {} (la')
      (lb') edge node {} (lb)
      (rdb) edge['] node {\dRdot} (rda)
      (rda') edge node {\Phi} (rdb')
      (rda) edge[equal] node {} (rda')
      (rdb) edge node {} (rdb')
      ;
    \end{tikzpicture}
  \end{equation} 
\end{lem}
\begin{proof}
  For each diagram, one verifies commutativity on (sums of) simple objects and morphisms.
  For example, the two composites around the associativity diagram are given by
  \[
    \Bigl( \tsum_{\ell} \bigl( \tsum_{i_{\ell}} x_{i_\ell} \bigr). \bigl( \tsum_{j_\ell} y_{j_\ell}\bigr) \Bigr). \bigl( \tsum_k z_k \bigr).
    \mapsto
    \tsum_\ell \tsum_{i_\ell} \tsum_{j_\ell} \tsum_{k} (x_{i_\ell}, (y_{j_\ell}, z_{k})).
  \]
  The adjoined morphisms $\de$, $\ze$, $\de.1$, and $\ze.1$ in the source are sent to identities in the target.
   All symmetry morphisms (both those adjoined in the construction of $\bigP$ and those adjoined in the construction of $\otimes$) are sent to the corresponding symmetry in the target.
  Commutativity of the other two diagrams is similar, but simpler.
  Note that the composite $\Phi \circ \dRdot$ is strict monoidal because the monoidal and unit constraints for $\dRdot = (-).1$ are $\de^R_1$ and $\ze^R_1$, which are among the morphisms that $\Phi$ sends to identities.
\end{proof}

\begin{defn}\label{defn:omdot-gaR-gaB}
  We define invertible modifications $\omdot$, $\ga^\dR$, and $\ga^\dB$ as in the following diagrams.
  Each unlabeled arrow is induced by the symmetric monoidal data of the cartesian product in $\Cat$.
\begin{equation}\label{eq:omdot-gaR-gaB}
  \begin{tikzpicture}[x=32mm,y=12mm,vcenter]
    \draw[0cell=.8] 
    (0,0) node (a) {\bigl( \bigP X \otimes \bigP Y \bigr) \otimes \bigP Z}
    (-1,0) node (b) {\bigP X \otimes \bigl( \bigP Y \otimes \bigP Z \bigr)}
    (-1,-1) node (c) {\bigP X \otimes \bigP \bigl( Y \times Z \bigr)}
    (-1,-2) node (d) {\bigP \bigl( X \times \bigl( Y \times Z \bigr)\bigr)}
    (0,-1) node (b') {\bigP \bigl( X \times Y \bigr) \otimes \bigP Z}
    (0,-2) node (c') {\bigP \bigl( \bigl( X \times Y \bigr) \times Z\bigr)}
    ;
    \draw[1cell=.8] 
    (b) edge node {\dAdot} (a)
    (b) edge['] node {1 \otimes \Phi} (c)
    (c) edge['] node {\Phi} (d)
    (a) edge node {\Phi \otimes 1} (b')
    (b') edge node {\Phi} (c')
    (d) edge node {} (c')
    ;
    \draw[2cell]
    node[between=b' and c at .45, rotate=45, 2label={above,\omdot}] {\Rightarrow}
    ;
  \end{tikzpicture}
  \qquad
  \begin{tikzpicture}[x=22mm,y=13mm,vcenter]
    \draw[0cell=.8] 
    (0,-2) node (ra) {\bigP X \otimes S}
    (1,-2) node (rb) {\bigP X}
    (0,-3) node (ra') {\bigP X \otimes \bigP *}
    (1,-3) node (rb') {\bigP \bigl( X \times * \bigr)}
    (0,-3.7) node (sa) {\bigP X \otimes \bigP Y}
    (1,-3.7) node (sb) {\bigP Y \otimes \bigP X}
    (0,-4.7) node (sa') {\bigP\bigl( X \times Y \bigr)}
    (1,-4.7) node (sb') {\bigP\bigl( Y \times X \bigr)}
    ;
    \draw[1cell=.8] 
    (ra) edge node {\dR} (rb)
    (ra') edge['] node {\Phi} (rb')
    (ra) edge[equal] node {} (ra')
    (rb') edge node {} (rb)
    (sa) edge node {\dB} (sb)
    (sa') edge node {} (sb')
    (sa) edge['] node {\Phi} (sa')
    (sb) edge node {\Phi} (sb')
    ;
    \draw[2cell]
    node[between=ra and rb' at .55, rotate=45, 2label={above,\ga^\dR}] {\Rightarrow}
    node[between=sa and sb' at .55, rotate=45, 2label={above,\ga^\dB}] {\Rightarrow}
    ;
  \end{tikzpicture}
\end{equation}
Each component is given by a symmetry isomorphism that interchanges order of summation.
\begin{itemize}
\item The component of $\omdot$ at an object
\[
  \bigl( \tsum_i x_i \bigr).\Bigl( \tsum_{\ell} \bigl( \tsum_{j_\ell} y_{j_\ell} \bigr).\bigl( \tsum_{k_\ell} z_{k_\ell} \bigr) \Bigr)
  \inspace \bigP X \otimes \bigl( \bigP Y \otimes \bigP Z\bigr)
\]
is the symmetry 
\[
  \tsum_i \tsum_\ell \tsum_{j_\ell} \tsum_{k_\ell} \bigl( x_i \,,\, (y_{j_\ell}, z_{k_\ell}) \bigr)
  \fto{\omdot = \beta}
  \tsum_\ell \tsum_i \tsum_{j_\ell} \tsum_{k_\ell} \bigl( x_i \,,\, (y_{j_\ell}, z_{k_\ell}) \bigr).
\]
\item The component of $\ga^\dR$ at an object
  \[
    \bigl( \tsum_i x_i \bigr).m = \bigl( \tsum_i x_i \bigr).\bigl( \tsum_{j=1}^m 1 \bigr) \inspace \bigP X \otimes S
  \]
  is the symmetry
  \[
    \tsum_i m x_i = \tsum_i \tsum_{j=1}^m x_i
    \fto{\ga^\dR = \beta}
    \tsum_{j=1}^m \tsum_i x_i = m \bigl( \tsum x_i \bigr).
  \]
\item The component of $\ga^\dB$ at an object
  \[
    \bigl( \tsum_i x_i \bigr).\bigl( \tsum_j y_j \bigr) \inspace \bigP X \otimes \bigP Y
  \]
  is the symmetry
  \[
    \tsum_i \tsum_j (y_j,x_i)
    \fto{\ga^\dB = \beta}
    \tsum_j \tsum_i (y_j,x_i).
  \]
\end{itemize}
For permutative categories $X$, $Y$, and $Z$, naturality of the symmetry isomorphisms $\beta$ implies that these components define monoidal transformations
\[
  \omdot_{X,Y,Z}, \quad \ga^\dR_{X}, \andspace \ga^\dB_X
\]
as in the diagrams above.
\end{defn}
\begin{lem}\label{lem:omdot-gaR-gaB}
  In the context of \cref{defn:omdot-gaR-gaB}, each of $\omdot$, $\ga^\dR$, and $\ga^\dB$ is an invertible modification.
\end{lem} 
\begin{proof}
  Each of $\Phi$, $\dB$, and $\dAdot$ is 2-natural by \cref{thm:free-tensor-adj-equiv}, \cref{defn:B}, and \cref{defn:Adot}, respectively.
  The tensor products $1 \otimes \Phi$ and $\Phi \otimes 1$ are also 2-natural, because $\otimes$ is 2-functorial (\cref{prop:tensor-functor}).
  For a functor of categories
  \[
    J\cn X \to X',
  \]
  the free construction $\bigP J$ is strict symmetric monoidal.
  This means that the pseudonaturality constraint $\dR_{\;\;\bigP J}$ is the identity by \cref{defn:R}.
  Therefore, the modification axioms for $\om$, $\ga^\dR$, and $\ga^\dB$ follow
  because the symmetry isomorphisms $\beta$ are preserved by symmetric monoidal functors.
\end{proof}

\begin{lem}\label{lem:mates-omdot-gaR1}
  \
  \begin{enumerate}
  \item The modification $\omdot$ in \cref{eq:omdot-gaR-gaB} is the mate of $\om = 1$ in the associativity diagram \cref{eq:om-gaL-gaR1}.
  \item The mate of $\ga^\dR$ in \cref{eq:omdot-gaR-gaB} is $\ga^\dR_1 = 1$, the identity in the corresponding diagram \cref{eq:om-gaL-gaR1}.
  \end{enumerate} 
\end{lem}
\begin{proof}
  Because the associativity in $(\Cat,\times)$ is an isomorphism, the mate of $\om = 1$ is given by whiskering of the inverse counit $\epzainv$ of \cref{prop:AAdot} with the composite
  \[
    \begin{tikzpicture}[x=20mm,y=15mm]
      \draw[0cell] 
      (0,0) node (a) {
        \bigP{X} \otimes \bigl( \bigP{Y} \otimes \bigP{Z} \bigr)
      }
      (a)+(1,-1) node (b) {
        \bigP{X} \otimes \bigP\bigl( Y \times Z \bigr)
      }
      (b)+(2,0) node (c) {
        \bigP\bigl( X \times \big( Y \times Z \bigr)\bigr)
      }
      (c)+(1,1) node (d) {
        \bigP\bigl( X \times \bigl( Y \times Z \bigr)\bigr).
      }
      ;
      \draw[1cell] 
      (a) edge node {1 \otimes \Phi} (b)
      (b) edge node {\Phi} (c)
      (c) edge['] node {} (d)
      ;
    \end{tikzpicture}
  \]
  Since the components of \cref{eq:epza} are given by $\de^L$ and $\ze^L$, the formula \cref{eq:Phi-deL} shows that this whiskering is equal to $\omdot$.

  Similarly, recall that $\dR \dRdot = 1$ and hence the counit $\epz^\dR$ is the identity.
  Thus, the mate $\ga^\dR_1$ is given by the whiskering $\ga^\dR * \dRdot$, which is the identity.
\end{proof}

\begin{thm}\label{thm:P-sm2fun}
  The free construction defines a symmetric monoidal 2-functor
  \[
    (\bigP, \Phi)\cn (\Cat, \times) \to (\permcat, \otimes)
  \]
  with the following properties:
  \begin{itemize}
  \item $\Phi$ is 2-natural,
  \item $\Psi = 1 \cn S \to \bigP *$,
  \item each of $\om$, $\ga^\dL$, and $\ga^\dR_1$ is an identity, and
  \item the components of $\omdot$, $\ga^\dR$, and $\ga^\dB$ are given by symmetry isomorphisms.
  \end{itemize}
\end{thm}
\begin{proof}
  The 2-naturality of $\Phi$ is verified in \cref{thm:free-tensor-adj-equiv}.
  The other claims about the data of $\bigP$ are verified in \cref{defn:omdot-gaR-gaB,lem:omdot-gaR-gaB} together with \cref{lem:P-2fun-om-gaL-gaR1,lem:mates-omdot-gaR1}.

  In $(\permcat,\otimes)$, recall from \cref{thm:permcat-smb}:
  \begin{itemize}
  \item $\dA$ and $\dB$ are 2-natural, and
  \item $\pi$, $\mu$, $R_{-|--}$, $R_{--|-}$, and $\syl$ are identities.
  \end{itemize}
  In $(\Cat,\times)$, the symmetric monoidal structure is that of a $\Cat$-enriched symmetric monoidal 1-category.
  Thus, all of the 2-dimensional symmetric monoidal data for $(\Cat,\times)$ are identities.

  Now we check the five axioms listed in \cref{sec:sm2fun-defn}, using the simplifications noted above.
  In \cref{eq:sm2fun-monoidal,eq:sm2fun-midunity}, all of the 2-cells are identities.
  In each of \cref{eq:sm2fun-lbraid,,eq:sm2fun-rbraid,eq:sm2fun-syll}, the only nontrivial 2-cells are given by $\ga^\dB$ and $\omdot$, with $\omdot$ appearing only in \cref{eq:sm2fun-rbraid}.
  Therefore, each of these last three axioms commutes by symmetric monoidal coherence.
\end{proof}

\section{The direct sum of permutative categories}
\label{sec:oplus}

This section defines a direct sum (sometimes called a biproduct) for permutative categories that is a model for both binary coproducts (\cref{cor:oplus=bicat_coprod}) and binary products (\cref{cor:oplus=bicat_prod}).
The main results are summarized in \cref{thm:oplus=bicat_directsum}.
Some treatment of products and coproducts has previously appeared in
\cite[Section~5]{Sch2014Ind} and \cite[Appendix~A]{FS2019Supplying}, but those authors do not relate their constructions to the Gray tensor product as we do here.

\begin{defn}\label{defn:perm-oplus}
  For permutative categories $(A,+^A,0^A)$ and $(B,+^B,0^B)$, define the \emph{direct sum} $(A \oplus B, +, 0)$ as follows.
  \begin{description}
  \item[Objects] The objects of $A \oplus B$ are finite formal sums
    \[
      \tsum [x_i] \forspace x_i \in \ob A \bincoprod \ob B,
    \]
    where we write $[x]$ for an object of $A$ or $B$, regarded as an object of $A \oplus B$.
    These sums are subject to the following relations.
    \begin{itemize}
    \item The empty sum is the monoidal unit $0$ and, in the case $x \in \ob A \bincoprod \ob B$ is either $0^A$ or $0^B$, we have
      \[
        [0^A] = 0 = [0^B].
      \]
    \item The formal sum of objects in $A$ or $B$ agrees with the monoidal sum in $A$ and $B$:
      \[
        [a] + [a'] = [a +^A a']
        \andspace
        [b] + [b'] = [b +^B b']
      \]
      for $a,a' \in \ob A$ and $b,b' \in \ob B$.
    \end{itemize}
  \item[Morphisms] The morphisms of $A \oplus B$ are generated under formal sums and composition by morphisms
    \[
      [f]\cn [a] \to [a'], \quad
      [g]\cn [b] \to [b'], \andspace
      \beta \cn [x] + [x'] \to [x'] + [x]
    \]
    for $a,a' \in \ob A$, $b,b' \in \ob B$, and $x,x' \in \ob A \bincoprod \ob B$.
    These are subject to the following relations.
    \begin{itemize}
    \item Addition of formal sums is a functor $\bigl( A \oplus B \bigr) \times \bigl( A \oplus B \bigr) \to \bigl( A \oplus B \bigr)$.
    \item For composable morphisms $f,f' \in A$ and $g,g'\in B$ we have
      \[
        [f'][f] = [f'f] \andspace [g'][g] = [g'g].
      \]
    \item For objects $a,a' \in A$ and $b,b' \in B$ we have
      \[
        \beta_{[a],[a']} = [\beta^A_{a,a'}]
        \andspace
        \beta_{[b],[b']} = [\beta^B_{b,b'}]
      \]
      where superscripts indicate the symmetry isomorphisms in $A$ and $B$.
    \item The morphisms $\beta$ are natural with respect to morphisms in $A \oplus B$ and satisfy the axioms of a symmetry isomorphism for the formal sum $+$.
    \end{itemize}
  \end{description}
  We will see that these generators and relations determine a well-defined permutative category in \cref{prop:oplus-Gray} below.
\end{defn}

\begin{rmk}\label{rmk:2fun-dirsum-incl}
  We note two key properties of the direct sum.
  \begin{enumerate}
  \item\label{it:dirsum-unique-rep} Using the relations for objects, each object of $A \oplus B$ can be uniquely represented as a formal sum
    \[
      \tsum [x_i] = [x_1] + \ldots + [x_n]
    \]
    where no two successive summands $x_i$ and $x_{i+1}$ are objects of the same category.
    So, either $x_i \in A$ and $x_{i+1} \in B$, or vice versa.
    Similarly, a sum of morphisms can be uniquely represented so that no successive summands come from the same category.
    All of our formulas will be applied to objects and morphisms presented in this way.

  \item\label{it:dirsum-can-incl} The assignments $x \mapsto [x]$ determine \emph{canonical inclusions}
    \begin{equation}\label{eq:can-incl}
      A \to A \oplus B \andspace B \to A \oplus B.
    \end{equation}
    The relations in $A \oplus B$ make the canonical inclusions into strict symmetric monoidal functors.
  \end{enumerate}
\end{rmk}

We use the (strong) Gray tensor product of 2-categories to show that $A \oplus B$ is well-defined.
The definition is due to Gray \cite{Gray80Closed} and we refer the reader to \cite[Section~3.1]{Gurski13Coherence} or \cite[Section~12.2]{JY212Dim} for textbook treatments.
To avoid confusion with the tensor for permutative categories in \cref{defn:AtensorB}, we use $\grtimes$ to denote the Gray tensor product.
For the following comparison, we let
\[
  \Si \cn \permcat \to \iicat
\]
denote the suspension 2-functor that sends each permutative category to a 1-object 2-category.

\begin{prop}\label{prop:oplus-Gray}
  For permutative categories $A$ and $B$, the direct sum is determined by an isomorphism of 1-object 2-categories
  \begin{align}\label{eq:oplus-Gray}
    \Si \bigl( A \oplus B \bigr) \iso \bigl( \Si A \bigr) \grtimes \bigl( \Si B \bigr).
  \end{align}
  In particular, $A \oplus B$ is a strict monoidal category. Via this isomorphism, the symmetry isomorphism $\beta_{[x],[y]}$ in $A \oplus B$ corresponds to the Gray structure 2-cell $\Si_{x,y}$ for $x,y \in \ob A \bincoprod \ob B$.
\end{prop}
\begin{proof}
  For $a \in A$ and $b \in B$, the generating objects $[a]$ and $[b]$ in $A \oplus B$ correspond to the generating 1-cells $a \grtimes 1$ and $1 \grtimes b$ in $\bigl( \Si A \bigr) \grtimes \bigl( \Si B \bigr)$.
  Likewise, the generating morphisms $[f]$ and $[g]$ correspond to the 2-cells $f \grtimes 1$ and $1 \grtimes g$.
  To complete the proof, one verifies that each of the relations for 1- and 2-cells in the Gray tensor product is either trivial in the 1-object case, or corresponds to a relation in \cref{defn:perm-oplus}.
\end{proof} 

\begin{prop}\label{prop:oplus-2fun}
  The direct sum defines a 2-functor
  \[
    \oplus \cn \permcat \times \permcat \to \permcat.
  \]
\end{prop}
\begin{proof}
  Let
  \[
    F \in \permcat\bigl(A,C\bigr) \andspace G \in \permcat\bigl(B,D\bigr)
  \]
  be symmetric monoidal functors.
  Define a symmetric monoidal functor
  \begin{equation}\label{eq:F-oplus-G}
    F \oplus G \cn A \oplus B \to C \oplus D
  \end{equation}
  by $\bigl(F \oplus G\bigr)0 = 0$, 
  \[
    \bigl(F \oplus G\bigr)[x] =
    \begin{cases}
      [Fx] & \ifspace x \in A\\
      [Gx] & \ifspace x \in B,
    \end{cases}
  \]
  and similarly for morphisms $[f]$ and $[g]$.
  Extend this to the symmetry isomorphism by
  \begin{equation}\label{eq:F-oplus-G-beta}
    \bigl(F \oplus G\bigr)\beta = \beta
  \end{equation} 
  and to sums of objects or morphisms using the unique representations of \cref{rmk:2fun-dirsum-incl}\,\cref{it:dirsum-unique-rep}.

  The unit constraint of $F \oplus G$ is the identity $1_0$.
  The monoidal constraints of $\bigl(F \oplus G\bigr)$ are determined by the following for $a,a' \in A$ and $b,b' \in B$.
  \begin{itemize}
 \item The source of $\bigl(\bigl(F \oplus G\bigr)_2\bigr)_{[a],[a']}$ is $[Fa + Fa']$, while the target is $[F(a+a')]$. Define this morphism to be $[F_2]$. 
 \item The source of $\bigl(\bigl(F \oplus G\bigr)_2\bigr)_{[b],[b']}$ is $[Gb + Gb']$, while the target is $[G(b+b')]$. Define this morphism to be $[G_2]$.
 \item The source and target of $\bigl(\bigl(F \oplus G\bigr)_2\bigr)_{[a],[b]}$ are both $[Fa] + [Gb]$, so define this morphism to be the identity.
 \item The source and target of $\bigl(\bigl(F \oplus G\bigr)_2\bigr)_{[b],[a]}$ are both $[Gb] + [Fa]$, so define this morphism to be the identity.
  \end{itemize}
  The symmetric monoidal functor axioms for $F$ and $G$, together with the relations in $A \oplus B$, show that $F \oplus G$ is a symmetric monoidal functor.

  Now let
  \[
    \phi \cn F \To F' \andspace \psi \cn G \To G'
  \]
  be monoidal transformations.
  Define a monoidal transformation
  \begin{equation}\label{eq:phi-oplus-psi}
    \phi \oplus \psi
  \end{equation}
  with components
  \begin{align*}
    \bigl(\phi \oplus \psi\bigr)_0\;\; & = 1, \\
    \bigl(\phi \oplus \psi\bigr)_{[a]} & = [\phi_a], \andspace \\
    \bigl(\phi \oplus \psi\bigr)_{[b]} & = [\psi_b]
  \end{align*}
  for $a \in A$ and $b \in B$.
  Then extend $\bigl(\phi \oplus \psi\bigr)$ to formal sums so that it is monoidal natural.

  One verifies directly from the formulas above that these definitions are 2-functorial.  This completes the proof.
\end{proof}

Checking composition with symmetric monoidal functors, one has the following.
\begin{cor}
  The isomorphism in \cref{prop:oplus-Gray} is 2-natural in both $A$ and $B$ with respect to symmetric monoidal functors, and therefore also strict symmetric monoidal functors.
\end{cor}

\begin{rmk}
  The definition of $\phi \oplus \psi$ in the proof of \cref{prop:oplus-2fun} is a special case of the Gray tensor product for icons in \cite[Proposition~3.7]{Gur2013monoidal}.
\end{rmk}

\subsection*{Bicategorical coproducts}
We now show that the construction above models bicategorical coproducts in $\permcat$ and 2-categorical coproducts in $\permcats$ (\cref{prop:oplus-coprod-adj-equiv,cor:oplus=bicat_coprod} below).

\begin{defn}\label{defn:M-Mdot}
  For permutative categories $A$, $B$, and $C$, we write
  \[
    M \cn \permcat\bigl(A \oplus B, C\bigr) \to
    \permcat\bigl(A, C\bigr) \times \permcat\bigl(B, C\bigr)
  \]
  for the functor induced by the canonical inclusions \cref{eq:can-incl}, given on objects by $x \mapsto [x]$.
  Define a reverse functor
  \[
    \Mdot \cn 
    \permcat\bigl(A, C\bigr) \times \permcat\bigl(B, C\bigr)
    \to
    \permcat\bigl(A \oplus B, C\bigr)
  \]
  as follows.

  For symmetric monoidal functors
  \[
    F \in \permcat\bigl(A,C\bigr) \andspace G \in \permcat\bigl(B,C\bigr),
  \]
  Let $\Mdot(F,G)$ denote the composite
  \begin{equation}\label{eq:Mdot-FG}
    A \oplus B \fto{F \oplus G} C \oplus C \fto{V} C,
  \end{equation}
  where the first morphism is given by \cref{eq:F-oplus-G} and $V$ is the strict symmetric monoidal functor that sends objects $[x]$ and morphisms $[f]$ in $C \oplus C$ to the corresponding objects and morphisms in $C$.
  For monoidal transformations
  \[
    \phi\cn F \To F'\cn A \to C \andspace \psi \cn G \To G'\cn B \to C, 
  \]
  use \cref{eq:phi-oplus-psi} and define
  \[
    \Mdot(\phi,\psi) = V * (\phi \oplus \psi).
  \]
  Functoriality of $\Mdot$ follows from 2-functoriality of $\oplus$ (\cref{prop:oplus-2fun}) and functoriality of the whiskering $V * -$.
\end{defn}

\begin{rmk}\label{rmk:Mdot-not-ssm}
  Note that $\Mdot$ is generally not a \emph{strict} symmetric monoidal functor with respect to the pointwise monoidal sums of its source and target.
  It is symmetric monoidal with nontrivial monoidal constraint determined by the adjoint equivalence \cref{eq:oplus-coprod} below.
\end{rmk}

\begin{prop}\label{prop:oplus-coprod-adj-equiv}
  Suppose $A$, $B$, and $C$ are permutative categories.
  The canonical inclusions induce an adjoint equivalence
  \begin{equation}\label{eq:oplus-coprod}
    M \cn \permcat\bigl(A \oplus B, C\bigr) \lradjequiv
    \permcat\bigl(A, C\bigr) \times \permcat\bigl(B, C\bigr) \bacn \Mdot.
  \end{equation}
  Moreover, the following statements hold.
  \begin{itemize}
  \item Both functors are 2-natural with respect to symmetric monoidal functors in $A$, $B$, and $C$.
  \item The functor $M$ is strict symmetric monoidal with respect to pointwise monoidal sums. 
  \item The functor $\Mdot$ is symmetric monoidal, with identity unit constraint and with monoidal constraint determined by the adjoint equivalence \eqref{eq:oplus-coprod}.
  \end{itemize}
  Furthermore, the restriction to subcategories of strict monoidal functors yields an isomorphism of categories
  \begin{equation}\label{eq:oplus-coprod-strict}
    \strict{A \oplus B, C} \fto{\iso} \strict{A,C} \times \strict{B,C}
  \end{equation}
  that is 2-natural with respect to strict symmetric monoidal functors.
\end{prop}
\begin{proof}
  Given a symmetric monoidal functor
  \[
    K \cn A \oplus B \to C,
  \]
  let
  \[
    K_A \cn A \to C
    \andspace
    K_B \cn B \to C
  \]
  denote the restrictions via the canonical inclusions.
  Likewise, we use subscripts $A$ and $B$ to denote the corresponding restrictions for a monoidal transformation $\phi \cn K \to K'$.
  Thus we have
  \[
    M(K) = (K_A,K_B) \andspace M(\phi) = (\phi_A,\phi_B).
  \]
Note that $M$ is strict symmetric monoidal since it is given by pre-composition with the canonical inclusions.

  One can verify by definition of $\Mdot$ that the composite $M\Mdot$ is the identity.
  For the other composite, define a monoidal isomorphism
  \[
    \al\cn \Mdot(K_A,K_B) \To K \forspace K \in \permcat\bigl( A \oplus B, C \bigr)
  \]
  as follows.  The components of $\al$ are the isomorphisms
  \begin{align*}
    \al_0 = K_0 & \cn 0 = 0 \to K0,\\
    \al_{[x]} = 1_{K[x]}\\
    \al_{[a]+[b]} = K_2 & \cn K[a] + K[b] \to K\bigl([a] + [b]\bigr),
          \andspace\\
    \al_{[b]+[a]} = K_2 & \cn K[b] + K[a] \to K\bigl([b] + [a]\bigr)
  \end{align*} 
  for $a \in A$, $b \in B$, and $x$ in $A$ or $B$.
  If $x$ and $x'$ are both objects of $A$, or both objects of $B$, then we have
  \[
    \Mdot(K_A,K_B)_2 = K_2\cn K[x] + K[x'] \to K\bigl([x] + [x']\bigr)
  \]
  and hence we take $\al_{[x]+[x']} = 1$ in such cases.
  The symmetric monoidal axioms for $(K,K_2,K_0)$ ensure that $\al$ is well-defined and monoidal natural.

  The whiskering $M * \al$ is the identity because the monoidal constraint of each $K_A$ or $K_B$ is given by restricting that of $K$ along the corresponding canonical inclusion.
  The whiskering $\al * \Mdot$ is the identity because, taking $K = \Mdot(F,G)$ above, the components $\al_{[x]+[x']} = (\Mdot(F,G)_2)_{[x],[x']}$ are identities when $x$ and $x'$ are in different categories, and are equal to $F_2$, respectively $G_2$, when $x$ and $x'$ are both in $A$, respectively $B$.
  This completes the proof that $\Mdot$ is adjoint inverse to $M$.

  If $K$ is strict symmetric monoidal, then all the components of $\al$ are identities and $\Mdot(K_A,K_B) = K$.
  Therefore, $\Mdot$ provides a strict inverse for $M$ in \cref{eq:oplus-coprod-strict}. In particular, this implies that $\Mdot$ is symmetric monoidal \cite[Section~2.1]{Kelly1974Doctrinal}.

  Next we verify 2-naturality of $M$.
  For
  \[
    \psi \cn P' \To P \cn \ol{A} \to A \in \permcat,
  \]
  the following two whiskerings are equal.
  \[
    \begin{tikzpicture}[x=30mm,y=18mm]
      \def\wl{1.3} 
      \def\wr{.1} 
      \def\h{.5} 
      \def\m{.4} 
      \draw[font=\Large] (0,0) node (eq) {=}; 
      \newcommand{\boundary}{
        \draw[0cell] 
        (0,0) node (a) {\ol{A}}
        (1,0) node (b) {\ol{A} \oplus B}
        (0,-1) node (a') {A}
        (1,-1) node (b') {A \oplus B}
        ;
        \draw[1cell] 
        (a) edge node {} (b)
        (a') edge node {} (b')
        (a) edge[',bend right=40] node {P'} (a')
        (b) edge[bend left=50] node[pos=.45] {P \oplus 1} (b')
        ;
      }
      \begin{scope}[shift={(-\wl-\m,\h)}]
        \boundary
        \draw[1cell] 
        (a) edge[bend left=40] node {P} (a')
        ;
        \draw[2cell] 
        node[between=a and a' at .6, rotate=0, 2label={above,\psi}] {\Rightarrow}
        ;
      \end{scope}
      \begin{scope}[shift={(\wr+\m,\h)}]
        \boundary
        \draw[1cell] 
        (b) edge[',bend right=50] node[pos=.45] {P' \oplus 1} (b')
        ;
        \draw[2cell] 
        node[between=b and b' at .6, rotate=0, 2label={above,\psi\oplus 1}] {\Rightarrow}
        ;
      \end{scope}
    \end{tikzpicture}
  \]
  A similar equality holds for the canonical inclusion of $B$.
  Together these imply 2-naturality of $M$ in \cref{eq:oplus-coprod} with respect to symmetric monoidal functors
  \[
    P\cn \ol{A} \to A, \quad Q\cn \ol{B} \to B, \andspace R\cn C \to \ol{C}
  \]
  and monoidal transformations between them.

 Next we verify 2-naturality of $\Mdot$ with respect to $P$, $Q$, and $R$ as above.
  For symmetric monoidal functors
  \[
    F\cn A \to C \andspace G\cn B \to C,
  \]
  one must prove that the composite $R\Mdot(F,G)(P \oplus Q)$
is equal to $\Mdot\bigl(\, RFP, RGQ \,\bigr)$. These symmetric monoidal functors are respectively given by the composites
   \begin{equation}\label{eq:2natM-1}
    \ol{A} \oplus \ol{B} \fto{P \oplus Q} A \oplus B \fto{F \oplus G} C \oplus C \fto{V}
    C \fto{R} \ol{C}
  \end{equation}
  and
  \begin{equation}\label{eq:2natM-2}
    \ol{A} \oplus \ol{B} \fto{RPF \oplus RQG} \ol{C} \oplus \ol{C} \fto{V}
     \ol{C}.
  \end{equation}
  Since $\oplus$ is a 2-functor, the latter is equal to 
   \[ \ol{A} \oplus \ol{B} \fto{P \oplus Q} A \oplus B \fto{F \oplus G} C \oplus C \fto{R \oplus R}
    \ol{C} \oplus \ol{C} \fto{V} \ol{C}.\]
    Thus the equality of the functors follows from 2-naturality of $V$, which can be easily verified. This 2-naturality of $V$ also implies the 2-naturality statement for 2-cells.
  \end{proof}

Recalling bicategorical coproducts from \cref{defn:bicat-2cat-prod}, we have the following.
\begin{cor}\label{cor:oplus=bicat_coprod}
  Let $A$ and $B$ be permutative categories.
  Then $A \oplus B$, with its canonical inclusions from \cref{eq:can-incl}, is a bicategorical coproduct of $A$ and $B$ in $\permcat$ and a 2-categorical coproduct of $A$ and $B$ in $\permcats$.
\end{cor}

\subsection*{Bicategorical products}
We now turn to products and show that the direct sum models bicategorical products of permutative categories.
Recall the comparison maps for bicategorical co/products from \cref{defn:bicat-direct-sum}.

\begin{prop}\label{prop:oplus=times}
  For permutative categories $A$ and $B$, the comparison map
  \[
    I \cn A \oplus B \to A \times B
  \]
  is an equivalence in $\permcat$.
  Moreover, $I$ is
  \begin{itemize}
  \item pseudonatural with respect to symmetric monoidal functors in both $A$ and $B$,
  \item 2-natural with respect to strict symmetric monoidal functors in both $A$ and $B$, and
  \item strict symmetric monoidal.
  \end{itemize} 
\end{prop}
\begin{proof}
  It is routine to verify that $I$ is uniquely determined by the requirement that it be strict symmetric monoidal, using that $\oplus$ is a 2-categorical coproduct in $\permcats$ (\cref{prop:oplus-coprod-adj-equiv}), and its restrictions to $A$ and $B$ separately.
  These restrictions are the inclusions
  \begin{align*}
    A \to A \times B & \qquad (a \mapsto (a,0)) \andspace\\
    B \to A \times B & \qquad (b \mapsto (0,b)),
  \end{align*}
  respectively.
  Note that this means $I$ is given on an object $\tsum [x_i]$ by
  \[
    I\bigl( \tsum [x_i] \bigr) = \bigl( \sum_{x_i \in A} x_i, \sum_{x_i \in B} x_i \bigr),
  \]
  and that $I$ maps the symmetry
  \[
    \beta_{[a],[b]}\cn [a] + [b] \iso [b] + [a], \forspace a \in A \andspace b \in B,
  \]
  to the identity on $(a,b)$ in $A \times B$.
  Since we have constructed the comparison map $I$ in $\permcats$ using its 2-categorical products and coproducts, it is necessarily a strict symmetric monoidal functor.

  Recall from \cref{prop:oplus-Gray} that the direct sum is determined by the Gray tensor product $\grtimes$ for 1-object 2-categories.
  For general 2-categories $X$ and $Y$, there is a 2-functor
  \[
    i\cn X \grtimes Y \to X \times Y
  \]
  that is the identity on objects and is a biequivalence \cite[Corollary~3.22]{Gurski13Coherence}.
  The formulation above, together with \cref{eq:oplus-Gray}, show that $\Si I$ is equal to the composite
  \[
    \Si \bigl( A \oplus B \bigr) \iso \bigl( \Si A \bigr) \grtimes \bigl( \Si B \bigr) \fto{i} \bigl( \Si A \bigr) \times \bigl( \Si B \bigr) \iso \Si \bigl( A \times B \bigr).
  \]
  Since this is a composite of two isomorphisms and a biequivalence, the entire composite is a biequivalence, and that implies that $I$ is an equivalence of categories.
  This verification completes the construction of the required strict symmetric monoidal equivalence $I \cn A \oplus B \to A \times B$.

We first prove the pseudonaturality of $I$ with respect to symmetric monoidal functors, and then the 2-naturality with respect to strict symmetric monoidal functors will follow immediately.
Let $F \cn A \to \ol{A}$ and $G \cn B \to \ol{B}$ be symmetric monoidal functors.
We will construct a monoidal natural isomorphism $\iota$ filling the square below.
 \[
    \begin{tikzpicture}[x=40mm,y=20mm]
      \draw[0cell] 
      (0,0) node (a) {A \oplus B}
      (1,0) node (b) {A \times B}
      (0,-1) node (a') {\ol{A} \oplus \ol{B}}
      (1,-1) node (b') {\ol{A} \times \ol{B}}
      (.5,-.5) node (iso) {\iso}
      ;
      \draw[1cell] 
      (a) edge node {I} (b)
      (b) edge node {F \times G} (b')
      (a) edge[swap] node {F \oplus G} (a')
      (a') edge[swap] node {I} (b')
      ;
    \end{tikzpicture}
  \]
  We compute the values of $\bigl( F \times G \bigr) \circ I$ and $I \circ \bigl( F \oplus G \bigr)$ on an object $\tsum [x_i]$ in $A \oplus B$.
  Recalling that 
  \[
  I \bigl( \tsum [x_i] \bigr) = \bigl( \sum_{x_i \in A} x_i, \sum_{x_i \in B} x_i \bigr),
  \]
  the top and right composite yields
  \[
  \bigl( F \times G \bigr) \circ I \bigl( \tsum [x_i] \bigr) = \bigl( F\bigl(\sum_{x_i \in A} x_i \bigr), G\bigl(\sum_{x_i \in B} x_i \bigr) \bigr)
  \]
  while the left and bottom composite yields
  \[
  I \circ \bigl( F \oplus G \bigr)\bigl( \tsum [x_i] \bigr) = \bigl( \sum_{x_i \in A} Fx_i, \sum_{x_i \in B} Gx_i \bigr).
  \]
  Recall from \cref{convention:F2} that we write $F_2$ and $G_2$ for any composite of sums of monoidal constraints.
  We define the isomorphism $\iota \cn \bigl( F \times G \bigr) \circ I \iso I \circ \bigl( F \oplus G \bigr)$ to have as its component at $\tsum [x_i]$ the morphism $(F_2^{-1}, G_2^{-1})$.
  These components satisfy the axioms for a monoidal transformation by the naturality of $F_2$ and $G_2$, the fact that these are symmetric monoidal functors, and the coherence theorem for monoidal functors \cite{JS1993Braided}.
  Furthermore, they assemble to give a pseudonatural transformation by a straightforward application of the formula
  \[
  (F' \circ F)_2 = F'_2 \circ F'(F_2)
  \]
  defining the monoidal constraint for a composite of monoidal functors. 
  Finally we note that if $F$ and $G$ are strict, then the isomorphism $\iota$ is an identity, verifying 2-naturality in that case.
\end{proof}

\begin{cor}\label{cor:oplus=bicat_prod}
Let $A$ and $B$ be permutative categories. 
Then $A \oplus B$, with its projections 
\[
A \oplus B \fto{I} A \times B \to A \andspace A \oplus B \fto{I} A \times B \to B,
\]
is a bicategorical product of $A$ and $B$ in $\permcat$.
\end{cor}
\begin{proof}
Bicategorical products are invariant under equivalence, and $I$ is an equivalence in $\permcat$.
Since $A \times B$ is a 2-categorical product in $\permcat$, it is also a bicategorical product (\cref{rmk:2cat-to-bicat}) and so $A \otimes B$ is as well.
\end{proof}

\begin{cor}\label{cor:oplus=times_univ}
Suppose $A$, $B$, and $C$ are permutative categories. The strict symmetric monoidal functor $I$ of \cref{prop:oplus=times} and the two projections induce an equivalence of categories
\begin{equation}\label{eq:oplus=times_univ}
    N\cn \permcat\bigl(A, B \oplus C\bigr) \fto{\hty} \permcat\bigl(A, B\bigr) \times \permcat\bigl(A, C\bigr).
  \end{equation}
  Moreover, the following statements hold.
  \begin{itemize}
  \item $N$ is 2-natural in $A$ with respect to symmetric monoidal functors, and pseudonatural in $B$ and $C$ with respect to symmetric monoidal functors.
  \item $N$ is 2-natural in $B$ and $C$ with respect to strict symmetric monoidal functors.
  \item $N$ is strict symmetric monoidal with respect to the pointwise monoidal sums.
\end{itemize}
\end{cor}
\begin{proof}
The only additional features to verify are the compatibility with the pointwise symmetric monoidal structures and the naturality statements.
The strict symmetric monoidal functor $I$ induces an equivalence of categories
\[
 \permcat\bigl(A, B \oplus C\bigr) \fto{\hty}  \permcat\bigl(A, B \times C\bigr) 
\]
by post-composition.
This is 2-natural in $A$ with respect to all symmetric monoidal functors, but the naturality in $B$ and $C$ matches that of $I$ as determined in \cref{prop:oplus=times}.
The product of permutative categories is a 2-categorical product with respect both strict symmetric monoidal functors and symmetric monoidal functors, giving a 2-natural isomorphism
\[
 \permcat\bigl(A, B \times C\bigr) \fto{\iso} \permcat\bigl(A, B\bigr) \times \permcat\bigl(A, C\bigr). 
 \]
These are induced by the projections from $B \times C$ onto $B$ and $C$ separately, and the projections are strict symmetric monoidal.

Now let $F \cn Y \to \ol{Y}$ be any strict symmetric monoidal functor between permutative categories.
It is routine to verify that post-composition with $F$,
\[
F_* \cn \permcat(X, Y) \to \permcat(X, \ol{Y}),
\]
is itself a strict symmetric monoidal functor with respect to the pointwise structures on the source and target.
Applying this fact to $I$ and the two projections yields the desired conclusions.
\end{proof}

\begin{rmk}\label{rmk:oplus=times_weakeq}
In general, there is no strict symmetric monoidal pseudo-inverse to $I$. 
This means that $B \oplus C$ is generally not equivalent to $B \times C$ in $\permcats$, and therefore that $\strict{A, B \oplus C}$ is generally not equivalent to $\strict{A, B} \times \strict{A, C}$.
We prove this in the case that $A = B = C = S$, the  monoidal unit from \cref{defn:S}.

Let $u$, respectively $v$, denote the image of the natural number $1 \in S$ in the left, respectively right, copy of $S$ using the two canonical inclusions $S \to S \oplus S$.
Assume that $I \cn S \oplus S \to S \times S$ admits a strict symmetric monoidal pseudo-inverse $J \cn S \times S \to S \oplus S$, equipped with monoidal isomorphisms $\epz \cn JI \iso 1_{S \oplus S}$ and $\eta \cn 1_{S \times S} \iso IJ$.
Let $L$ denote the isomorphism
\begin{equation}\label{eq:example_iso}
L\cn \strict{S \oplus S, S \oplus S} \iso \bigl( S \oplus S \bigr) \times \bigl( S \oplus S \bigr),
\end{equation}
given by the 2-categorical coproduct property of $\oplus$ (\cref{cor:oplus=bicat_coprod}) and the fact that homming out of $S$ is the underlying category functor (\cref{rmk:what-is-S}).
Thus, for a strict symmetric monoidal functor $F \cn S \oplus S \to S \oplus S$, the object $L(F)$ is the pair $\bigl( Fu, Fv \bigr)$.
In particular, $L(1_{S \oplus S}) = (u,v)$.
Applying $L$ to the isomorphism $\epz \cn JI \iso 1$ produces an isomorphism
\[
L(\epz) \cn \bigl( JIu, JIv \bigr) \iso \bigl( u, v \bigr),
\]
where each factor above is a sum in $S \oplus S$ with only a single term.

Any object of $S \oplus S$ has the form
\[
a_1u + b_1v + \cdots + a_nu + b_nv, \forspace a_i,b_i \in \mathbbm{N},
\]
by \cref{rmk:2fun-dirsum-incl}.
Such an object is isomorphic to $u$ if and only if it is equal to $u$.
Furthermore, one checks that the morphism set $\bigl(S \oplus S\bigr)(u, u)$ has only one element, which is the identity on $u$.
Analogous observations hold for $v$.
Therefore $JIu = u$ and $JIv = v$.

But in $S \times S$, we have the equality of objects $(1,0)(0,1) = (0,1)(1,0)$, so if $J$ is strict symmetric monoidal we would then have
\[
u + v = J(1,0) + J(0,1) = J(1,1) = J(0,1) + J(1,0) = v + u
\]
in $S \oplus S$.
This is false, so there cannot be a strict symmetric monoidal pseudo-inverse to $I \cn S \oplus S \to S \times S$.
\end{rmk}

The following theorem collects the key results of this section.
\begin{thm}\label{thm:oplus=bicat_directsum}
Let $A$ and $B$ be permutative categories. 
\begin{enumerate}
\item $A \oplus B$, with its canonical inclusions \cref{eq:can-incl}
\[
A \to A \oplus B \andspace B \to A \oplus B 
\]
is a bicategorical coproduct of $A$ and $B$ in $\permcat$ and a 2-categorical coproduct of $A$ and $B$ in $\permcats$.
\item $A \oplus B$, with its projections 
\[
A \oplus B \fto{I} A \times B \to A \andspace A \oplus B \fto{I} A \times B \to B,
\]
is a bicategorical product of $A$ and $B$ in $\permcat$.
\item The comparison map
\[
A \oplus B \to A \times B
\]
is a strict symmetric monoidal equivalence, and therefore $\permcat$ admits bicategorical direct sums.
\end{enumerate}
\end{thm}
\begin{proof}
The first statement appears as \cref{cor:oplus=bicat_coprod}, the second as \cref{cor:oplus=bicat_prod}, and the third as \cref{prop:oplus=times}.
\end{proof}

The results of \cref{thm:oplus=bicat_directsum} and those of \cref{sec:bilin-otimes} show that the tensor product distributes over direct sums.
\begin{cor}\label{cor:oplus-otimes-dist}
  Suppose $A$, $B$, and $C$ are permutative categories.
  There are adjoint equivalences
  \begin{align*}
    \pa^R_C\cn \bigl( A \otimes C \bigr) \oplus \bigl( B \otimes C \bigr)
    & \lradjequiv \bigl(A \oplus B\bigr) \otimes C \bacn\patil^R_C \andspace \\
    \pa^L_C\cn \bigl( C \otimes A \bigr) \oplus \bigl( C \otimes B \bigr)
    & \lradjequiv C \otimes \bigl(A \oplus B\bigr) \bacn\patil^L_C
  \end{align*}
  with the following properties.
  \begin{itemize}
  \item Each of $\pa^R$, $\patil^R$, $\pa^L$, and $\patil^L$ is 2-natural with respect to symmetric monoidal functors in $A$, $B$, and $C$.
  \item Each of $\pa^R$, $\patil^R$, $\pa^L$, and $\patil^L$ is strict symmetric monoidal.
  \end{itemize}
\end{cor}
\begin{proof}
  We use the bicategorical Yoneda Lemma \cite{Str96Cat,JY212Dim} as follows.
  Abbreviating
  \[
    \zP \bigl( -,- \bigr) = \permcat \bigl( -,- \bigr) \andspace
    \strict{ -, -} = \permcats \bigl( -,- \bigr),
  \]
  we have the following for each permutative category $D$:
  \begin{align*}
    \strict{ \bigl(A \otimes C\bigr) \oplus \bigl(B \otimes C\bigr) , D }
    & \iso \strict{ A \otimes C , D }
      \times \strict{ B \otimes C , D }
    & & \ref{eq:oplus-coprod-strict}\\
    & \iso \zP \bigl( A , \zP(C,D) \bigr)
      \times \zP \bigl( B , \zP(C,D) \bigr)
    & & \ref{eq:otimes-curry}\\
    & \hty \zP \bigl( A \oplus B, \zP(C,D) \bigr)
    & & \ref{prop:oplus-coprod-adj-equiv}\\
    & \iso \strict{\bigl( A \oplus B \bigr) \otimes C, D}
    & & \ref{eq:otimes-curry}
  \end{align*}
  Letting $D=\bigl(A \oplus B\bigr) \otimes C$, the identity map in the final step above corresponds to a strict symmetric monoidal functor $\pa^R_C$ as desired. 
  A similar computation yields $\pa^L_C$.
  The reverse functors $\patil^R$ and $\patil^L$ are given in the final step of the Yoneda argument, by taking $D = \bigl( A \otimes C \bigr) \oplus \bigl( B \otimes C \bigr)$ and the identity in the first step.

  Given symmetric monoidal functors
  \[
    P\cn \ol{A} \to A, \quad
    Q\cn \ol{B} \to B, \andspace
    R\cn C \to \ol{C},
  \]
  note that the tensor product in \cref{defn:FtensorG} results in strict symmetric monoidal functors.
  Also, recall that the sum of two strict monoidal functors in \cref{prop:oplus-2fun} will be strict monoidal.
  Therefore, each of the induced
  \[
    \bigl( P \oplus Q \bigr) \otimes R, \quad
    \bigl( P \otimes R \bigr) \oplus \bigl( Q \otimes R \bigr), \quad
    R \otimes \bigl( P \oplus Q \bigr), \andspace
    \bigl( R \otimes P \bigr) \oplus \bigl( R \otimes Q \bigr)
  \]
  is strict symmetric monoidal.
  Hence, 2-naturality of $\pa^R$, $\patil^R$, $\pa^L$, and $\patil^L$ follows from the 2-naturality of each equivalence and isomorphism above.
\end{proof}

\begin{rmk}\label{rmk:bicolim-otimes}
  The standard Yoneda argument as in \cref{cor:oplus-otimes-dist} shows, more generally, that $\otimes$ distributes, via a strict symmetric monoidal equivalence, over small bicolimits in $\permcat$.
\end{rmk}

Using that $\oplus$ is both a product and a coproduct, we can conclude the following corollary. 
\begin{cor}\label{cor:oplus_smb}
  \ 
\begin{enumerate}
\item The 2-category $\permcats$ has the structure of a symmetric monoidal 2-category using $\oplus$.
\item The 2-category $\permcat$ has the structure of a symmetric monoidal 2-category using $\oplus$, in two distinct ways. 
\end{enumerate}
\end{cor}
\begin{proof}
For $\permcats$, first note that every category with coproducts is symmetric monoidal. Adding a $\cat$-enrichment to that statement yields a proof that every 2-category with 2-categorical coproducts is a $\cat$-enriched version of a symmetric monoidal category. (This is the symmetric version of what was called a strongly monoidal 2-category in \cite{Gur2013monoidal}.) One then applies \cref{cor:oplus=bicat_coprod}.

For $\permcat$, the two distinct symmetric monoidal structures arise from the coproduct and product properties for $\oplus$.
\begin{description}
\item[As a coproduct] We note that it is possible to use the data from the proof of the symmetric monoidal structure on $\permcats$ above. The axioms still hold, and the associativity and unit isomorphisms inherited from $\permcats$ are still 2-natural in all variables even when extended from strict symmetric monoidal functors to symmetric monoidal functors.
\item[As a product] By \cref{cor:oplus=bicat_prod}, $\oplus$ is a bicategorical product.
  This gives the structure of a symmetric monoidal bicategory by \cite[Theorem~2.15]{CKWW07Cartesian}.
\end{description}
\end{proof}

\begin{rmk}\label{rmk:same_smb}
The identity functor is a symmetric monoidal biequivalence between the two symmetric monoidal structures on $\permcat$ given above. The reader can prove this by writing out the explicit product-induced structure, which has associator given as the composite
\[
(A \oplus B) \oplus C \fto{(I \times 1) \circ I} 
(A \times B) \times C \cong A \times (B \times C) \fto{(1 \oplus I^{\bullet}) \circ I^{\bullet}} 
A \oplus (B \oplus C),
\]
where $I^\bullet$ is a chosen pseudoinverse for $I$ in $\permcat$. This composite is isomorphic, but not equal, to the coproduct-induced associator for $\oplus$, and such an isomorphism is the data needed to construct the modification $\omega$ in \cref{defn:sm2fun}.
\end{rmk}

\appendix
\section{Symmetric monoidal 2-categories}
\label{sec:smb-defn}

Below, we summarize the definition of symmetric monoidal bicategory, as given in \cite[Section~12.1]{JY212Dim}.
See \cite[Section~2]{SP2011Classification} and  \cite[Section~2]{CG2014Iterated} for equivalent presentations.
We list the data in \cref{defn:sm2cat} below.
After that, we discuss several simplifications that occur in our application.
This leads to a simplification in the symmetric monoidal bicategory axioms, and we give the simplified axioms at the end of this section.

\begin{defn}\label{defn:sm2cat}
A \textit{symmetric monoidal 2-category} consists of the following.
\begin{description}
\item[Monoidal Data]\ 
  \begin{itemize}
  \item It has a 2-category $\B$.
  \item It has a tensor product pseudofunctor
    $\otimes\cn \B \times \B \to \B$, denoted by concatenation below.
  \item It has a unit object $S$ determined by a pseudofunctor $* \to \B$.
  \item It has an associativity pseudonatural equivalence with components $\dA\colon (XY)Z \fto{\simeq} X(YZ)$.
  \item It has unit pseudonatural equivalences with components
    \[
      \dL\colon SX \fto{\simeq} X \andspace
      \dR\colon SX \fto{\simeq} X.
    \]
  \item It has invertible modifications $\pi$, $\mu$, $\la$, $\rho$ as follows.
    
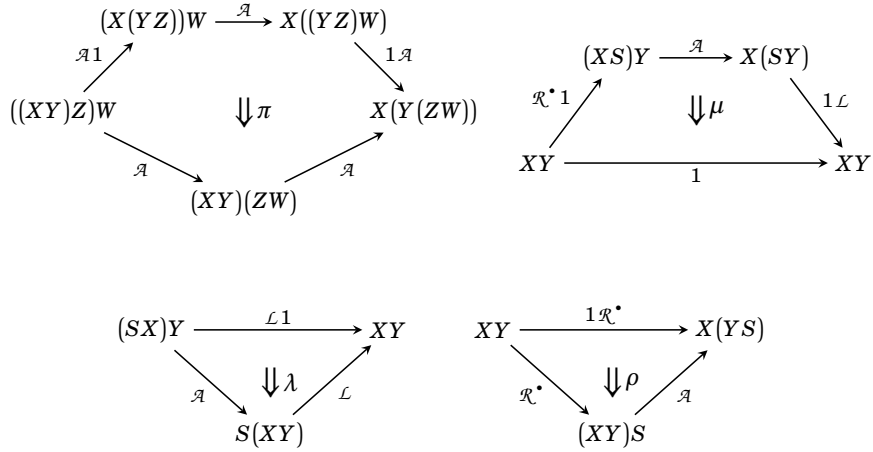
\[
      \begin{tikzpicture}[x=20mm,y=20mm,scale=.6,vcenter]
        \draw[0cell=.8] 
        (0,0) node (0) {\bigl(\bigl(XY\bigr)Z\bigr)W}
        (1,1) node (1) {\bigl(X\bigl(YZ\bigr)\bigr)W}
        (3,1) node (2) {X\bigl(\bigl(YZ\bigr)W\bigr)}
        (4,0) node (3) {X\bigl(Y\bigl(ZW\bigr)\bigr)}
        (2,-1) node (4) {\bigl(XY\bigr)\bigl(ZW\bigr)};
        \draw[1cell=.8] 
        (0) edge node {\dA \,1} (1)
        (1) edge node {\dA} (2)
        (2) edge node {1\,\dA} (3)
        (0) edge[swap] node {\dA} (4)
        (4) edge[swap] node {\dA} (3);
        \draw[2cell]
        (2,0) node[rotate=270, 2label={above,\,\pi}] {\Rightarrow}
        ;
      \end{tikzpicture}
      \quad
      \begin{tikzpicture}[x=30mm,y=20mm,scale=.7,vcenter]
        \draw[0cell=.8] 
        (0,0) node (0) {\bigl(XS\bigr)Y}
        (1,0) node (1) {X\bigl(SY\bigr)}
        (-.5,-1) node (2) {XY}
        (1.5,-1) node (3) {XY};
        \draw[1cell=.8] 
        (0) edge node {\dA} (1)
        (2) edge node {\dRdot\,1} (0)
        (2) edge[swap] node {1} (3)
        (1) edge node {1 \, \dL} (3);
        \draw[2cell]
        (.5,-.5) node[rotate=270, 2label={above,\,\mu}] {\Rightarrow}
        ;
      \end{tikzpicture}
    \]
    \[
      \begin{tikzpicture}[x=45mm,y=20mm,scale=.7]
        \draw[0cell=.8] 
        (0,0) node (0) {\bigl(SX\bigr)Y}
        (1,0) node (1) {XY}
        (.5,-1) node (2) {S\bigl(XY\bigr)};
        \draw[1cell=.8] 
        (0) edge node {\dL\,1} (1)
        (0) edge[swap] node {\dA} (2)
        (2) edge[swap] node {\dL} (1);
        \draw[2cell]
        (.5,-.5) node[rotate=270, 2label={above,\,\la}] {\Rightarrow}
        ;
      \end{tikzpicture}
      \qquad
      \begin{tikzpicture}[x=45mm,y=20mm,scale=.7]
        \draw[0cell=.8] 
        (0,0) node (0) {XY}
        (1,0) node (1) {X\bigl(YS\bigr)}
        (.5,-1) node (2) {\bigl(XY\bigr)S};
        \draw[1cell=.8] 
        (0) edge node {1\, \dRdot} (1)
        (0) edge[swap] node {\dRdot} (2)
        (2) edge[swap] node {\dA} (1);
        \draw[2cell]
        (.5,-.5) node[rotate=270, 2label={above,\,\rho}] {\Rightarrow}
        ;
      \end{tikzpicture}
    \]
  \end{itemize}
\item[Braiding and Sylleptic Data]\ 
  \begin{itemize}
  \item It has a pseudonatural braid equivalence with components
    \[
      \dB\colon XY \fto{\simeq} YX.
    \] 
  \item It has two invertible modifications, denoted $R_{-|--}$ and $R_{--|-}$, and
  \item an invertible modification $\syl$ as follows.
    
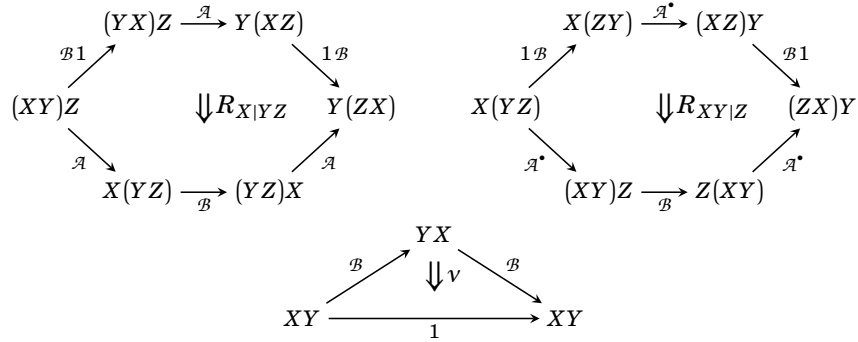
\[
      \begin{tikzpicture}[x=25mm,y=16mm,scale=.7]
        \draw[0cell=.8] 
        (.3,0) node (0) {\bigl(XY\bigr)Z}
        (1,1) node (1) {\bigl(YX\bigr)Z}
        (2,1) node (2) {Y\bigl(XZ\bigr)}
        (2.7,0) node (3) {Y\bigl(ZX\bigr)}
        (1,-1) node (4) {X\bigl(YZ\bigr)}
        (2,-1) node (5) {\bigl(YZ\bigr)X};
        \draw[1cell=.8] 
        (0) edge node {\dB \,1} (1)
        (1) edge node {\dA} (2)
        (2) edge node {1\,\dB} (3)
        (0) edge[swap] node {\dA} (4)
        (4) edge[swap] node {\dB} (5)
        (5) edge[swap] node {\dA} (3);
        \draw[2cell]
        (1.5,0) node[rotate=270, 2label={above,\,R_{X \mid YZ}}] {\Rightarrow}
        ;
        \begin{scope}[shift={(3.5,0)}]
          \draw[0cell=.8] 
          (.3,0) node (0) {X\bigl(YZ\bigr)}
          (1,1) node (1) {X\bigl(ZY\bigr)}
          (2,1) node (2) {\bigl(XZ\bigr)Y}
          (2.7,0) node (3) {\bigl(ZX\bigr)Y}
          (1,-1) node (4) {\bigl(XY\bigr)Z}
          (2,-1) node (5) {Z\bigl(XY\bigr)};
          \draw[1cell=.8] 
          (0) edge node {1\,\dB} (1)
          (1) edge node {\dAdot} (2)
          (2) edge node {\dB\,1} (3)
          (0) edge[swap] node {\dAdot} (4)
          (4) edge[swap] node {\dB} (5)
          (5) edge[swap] node {\dAdot} (3);
          \draw[2cell]
          (1.5,0) node[rotate=270, 2label={above,\,R_{XY \mid Z}}] {\Rightarrow}
          ;
        \end{scope} 
        \begin{scope}[shift={(3.25,-1.5)}]
          \draw[0cell=.8] 
          (-1,-1) node (0) {XY}
          (1,-1) node (1) {XY}
          (0,0) node (2) {YX};
          \draw[1cell=.8] 
          (0) edge[swap] node {1} (1)
          (0) edge node {\dB} (2)
          (2) edge node {\dB} (1);
          \draw[2cell]
          (0,-.5) node[rotate=270, 2label={above,\,\syl}] {\Rightarrow}
          ;
        \end{scope}
      \end{tikzpicture}
    \]
  \end{itemize}
\end{description} 
These data satisfy three axioms for the monoidal structure, four axioms for
the braided structure, two axioms for the sylleptic structure, and one
final axiom for the symmetric structure.
\end{defn}

In our application, $\otimes$ is a 2-functor and each of $\pi$, $\mu$, $\rho$, $R_{- \mid --}$, $R_{-- \mid -}$, and $\syl$ is an identity (see \cref{sec:2Ddata}).
Thus, the relevant axioms simplify considerably.
The only nontrivial data in the axioms is that of $\la$, together with the following mates with respect to the adjoint equivalence $(\dA,\dAdot)$.
\begin{itemize}
\item For the pentagonator $\pi$, the axioms use ten mates with respect to $(\dA,\dAdot)$, denoted $\pi_i$ for $1 \leq i \leq 10$.
  See \cite[12.1.4 and~12.5.1]{JY212Dim} for a complete description.
\item For both hexagonators, $R^1$, $R^2$, and $R^3$ denote various mates with respect to $(\dA,\dAdot)$.
  See \cite[12.1.8 and~12.1.11]{JY212Dim} for a complete description.
\end{itemize}

Recall the unit and counit
\[
  \etaa = (\epz')^\inv \cn 1 \To \dAdot \dA \andspace
  \epza \cn \dA \dAdot \To 1
\]
from \cref{eq:etaa} and \cref{eq:epza}, respectively.
Each of the mates above is obtained by a suitable pasting with one of the unit, counit, or their inverses. 
For example, $\pi_3$ at $X,Y,W,Z$ is the following mate.
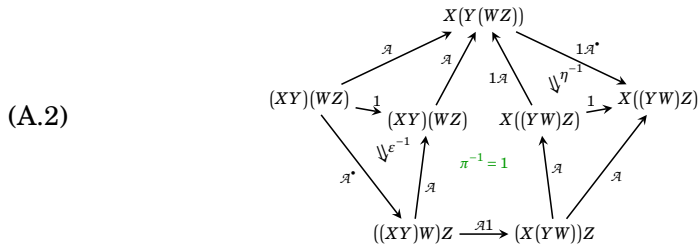
\begin{equation}\label{eq:pi3}
  \begin{tikzpicture}[x=20mm,y=20mm,scale=.8,vcenter]
    \smallnodes
    \draw[0cell] 
    (180-18:1)+(-.5,0) node (a) {
      \bigl( XY \bigr)\bigl( WZ \bigr)
    }
    (180+1.0*72-18:1) node (b) {
      \bigl(\bigl( XY \bigr) W \bigr)Z
    }
    (180+2.0*72-18:1) node (c) {
      \bigl(X\bigl( YW \bigr)  \bigr)Z
    }
    (180+3.0*72-18:1)+(0.5,0) node (d) {
      X\bigl(\bigl( YW \bigr) Z \bigr)
    }
    (180+4.0*72-18:1) node (e) {
      X\bigl(Y\bigl( WZ \bigr)  \bigr)
    }
    (d)+(190:1) node[scale=1] (y) {
      X\bigl(\bigl( YW \bigr) Z \bigr)
    }
    (a)+(-10:1) node[scale=1] (x) {
      \bigl( XY \bigr)\bigl( WZ \bigr)
    }
    ;
    \draw[1cell] 
    (a) edge['] node {\dAdot} (b)
    (b) edge node {\dA 1} (c)
    (c) edge['] node {\dA} (d)
    (a) edge node {\dA} (e)
    (e) edge node {1 \dAdot} (d)
    (y) edge node (y1) {1} (d)
    (c) edge['] node {\dA} (y)
    (y) edge node {1\dA} (e)
    (a) edge node (x1) {1} (x)
    (b) edge['] node {\dA} (x)
    (x) edge node {\dA} (e)
    ;
    \draw[2cell] 
    (e)+(0,-1.2) node[anno] {\pi^\inv = 1}
    node[between=x1 and b at .4, shift={(-.1,0)}, rotate=-70, 2label={above,\epzainv}] {\Rightarrow}
    node[between=y1 and e at .3, shift={(-.05,-.1)}, rotate=-70, 2label={above,\etaainv}] {\Rightarrow}
    ;
  \end{tikzpicture}
\end{equation}
The other mates are similar.
Since each mate involves only unit or counit components $\etaa$ and $\epza$, the proof of \cref{thm:permcat-smb} uses the Coherence Theorem~\ref{thm:iterated-tensor-coherence} for iterated tensor products to verify that the relevant pastings are equal.

Now we list the axioms for a symmetric monoidal bicategory, in the special case that
\begin{itemize}
\item $\B$ is a 2-category, 
\item $\otimes$ is a 2-functor,
\item each of $\dA$, $\dL$, $\dR$, $\dB$ is 2-natural, and
\item $\pi$, $\mu$, $\rho$, $R_{--|-}$, $R_{-|--}$, and $\syl$ are all identities.
\end{itemize}
We mark diagram regions with \annoarg{$=$}, \annoarg{$R=1$}, \annoarg{$\natA$}, etc., to indicate identities that follow from the above assumptions.
Of the ten general axioms, two of the monoidal axioms are completely trivial (involving only $\pi$, $\rho$, and $\mu$), as is the symmetry axiom (involving only $\syl$).
There are seven remaining axioms.

\ \\
\noindent\textbf{Left Normalization Axiom}
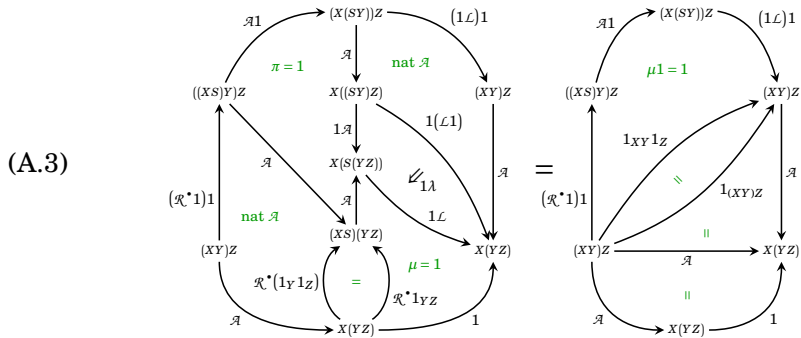
\begin{equation}\label{eq:left-norm-axiom}
\begin{tikzpicture}[scale=.7,vcenter]
\smallnodes
\def\m{.05} 
\def\a{.5} \def\b{1} \def\bp{1.04} \def\c{.3}
\def\t{1.4} \def\v{1} \def\u{.5} \def\q{.95} \def\m{.5} \def\z{60}
\draw[font=\huge] (0,0) node [rotate=0] (eq) {$=$};
\newcommand{\boundary}{
\draw[0cell=\smallzero] 
(0,0) node (x11) {\bigl(X\bigl(SY\bigr)\bigr)Z}
($(x11)+(-\b,-\u)$) node (x21) {\bigl(\bigl(XS\bigr)Y\bigr)Z}  
($(x11)+(\b,-\u)$) node (x22) {\bigl(XY\bigr)Z}
($(x11)+(-\b,-\v-\u)$) node (x31) {\bigl(XY\bigr)Z} 
($(x11)+(\b,-\v-\u)$) node (x32) {X\bigl(YZ\bigr)}
($(x11)+(0,-2*\v)$) node (x41) {X\bigl(YZ\bigr)} 
;
\draw[1cell] 
(x31) edge node[pos=.3] {\bigl(\dRdot 1\bigr)1} (x21)
(x21) edge[bend left] node {\dA1} (x11)
(x11) edge[bend left] node {\bigl(1\dL\bigr)1} (x22)
(x22) edge node (as) {\dA} (x32)
(x31) edge[out=-90,in=180] node[swap,pos=.4] {\dA} (x41) 
(x41) edge[out=0,in=-90] node[swap,pos=.6] {1} (x32)
;
}
\begin{scope}[xscale=2.6,yscale=3,shift={(-\b-3*\m/4,\v)}]
\boundary
\draw[0cell=\smallzero] 
($(x11)+(0,-\u)$) node (A) {X\bigl(\bigl(SY\bigr)Z\bigr)}
($(x11)+(0,-\q)$) node (B) {X\bigl(S\bigl(YZ\bigr)\bigr)}
($(x11)+(0,-\t)$) node (C) {\bigl(XS\bigr)\bigl(YZ\bigr)}
;
\draw[1cell] 
(x11) edge node[swap] {\dA} (A)
(A) edge node[swap] {1\dA} (B)
(A) edge[bend left=20] node[pos=.3] (onelone) {1\bigl(\dL 1\bigr)} (x32)
(C) edge node {\dA} (B) 
(B) edge[bend right=15] node[pos=.6] {1\dL} (x32)
(x21) edge node[swap,pos=.4] {\dA} (C) 
(x41) edge[bend left=\z] node {\dRdot \bigl(1_Y 1_Z\bigr)} (C)
(x41) edge[bend right=\z] node[swap,pos=.45] {\dRdot 1_{YZ}} (C)
;
\draw[2cell] 
node[between=x22 and A at .6, shift={(0,.7)}, anno] {\natA}
node[between=B and as at .4, shift={(0,-.2)}, rotate=-135, 2label={above,1\lambda}] {\Rightarrow}
node[between=x21 and A at .5, shift={(0,.6)}, anno] {\pi = 1}
node[between=x31 and C at .2, shift={(.3,.7)}, anno] {\natA}
node[between=C and x41 at .5, anno] {=}
node[between=C and x32 at .5, shift={(0,-.5)}, anno] {\mu = 1}
;
\end{scope}  
\begin{scope}[xscale=1.8,yscale=3,shift={(\b+\m,\v)}]
\boundary
\draw[1cell]
(x31) edge[bend left=15] node {1_{XY}1_Z} (x22)
(x31) edge[bend right=15] node[swap,pos=.6] {1_{(XY)Z}} (x22)
(x31) edge node[swap] (ass) {\dA} (x32)
;
\draw[2cell] 
node[between=x21 and x22 at .4, shift={(0,.5)}, anno] {\mu 1 = 1}
node[between=x31 and x22 at .45, rotate=-45, anno] {=}
node[between=x31 and x32 at .6, shift={(0,.4)}, rotate=-90, anno] {=}
node[between=ass and x41 at .5, shift={(-.0,0)}, rotate=-90, anno] {=}
;
\end{scope}  
\end{tikzpicture}
\end{equation}

\noindent\textbf{(3,1)-Crossing Axiom}
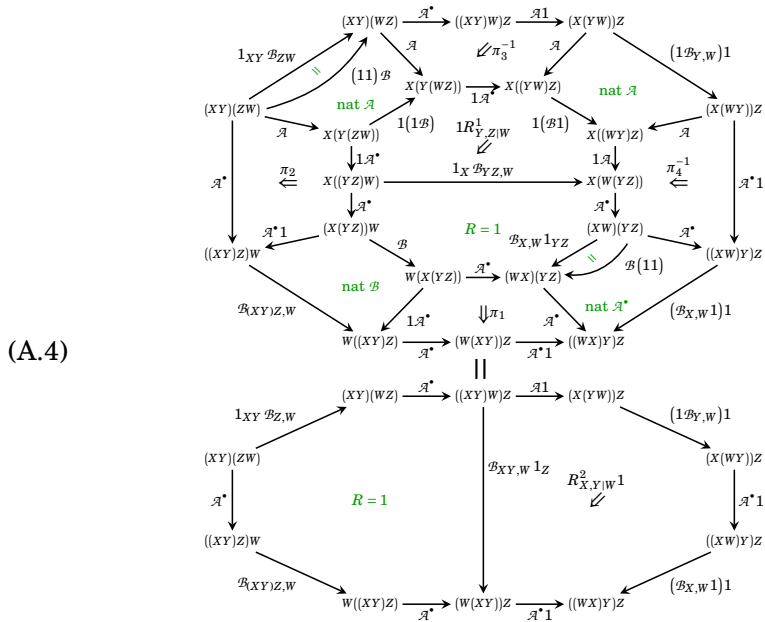
\begin{equation}\label{eq:31-cross-axiom}
\begin{tikzpicture}[x=23mm,y=13.6mm,vcenter,scale=.61]
  \smallnodes
  \def\r{1.2}
  \def\s{1.1}
  \def\a{40}
  \def\e{20}
  \def\v{.75}
  \def\h{1.25}
  \draw[0cell=\smallzero] 
  (-\h,0) node (l) {X\bigl(\bigl(YZ\bigr)W\bigr)}
  (l) ++(0,\v) node (lt) {X\bigl(Y\bigl(ZW\bigr)\bigr)}
  (lt) ++(45:\s) node (ltt) {X\bigl(Y\bigl(WZ\bigr)\bigr)}
  (l) ++(0,-\v) node (lb) {\bigl(X\bigl(YZ\bigr)\bigr)W}
  (lb) ++(-45:\s) node (lbb) {W\bigl(X\bigl(YZ\bigr)\bigr)}
  (\h,0) node (r) {X\bigl(W\bigl(YZ\bigr)\bigr)}
  (r) ++(0,\v) node (rt) {X\bigl(\bigl(WY\bigr)Z\bigr)}
  (rt) ++(135:\s) node (rtt) {X\bigl(\bigl(YW\bigr)Z\bigr)}
  (r) ++(0,-\v) node (rb) {\bigl(XW\bigr)\bigl(YZ\bigr)}
  (rb) ++(-135:\s) node (rbb) {\bigl(WX\bigr)\bigl(YZ\bigr)}  
  (lt) ++(180-\a+\e:\r) node (lt2) {\bigl(XY\bigr)\bigl(ZW\bigr)}
  (ltt) ++(180-\a-\e:\r) node (ltt2) {\bigl(XY\bigr)\bigl(WZ\bigr)}
  (rtt) ++(\a+\e:\r) node (rtt2) {\bigl(X\bigl(YW\bigr)\bigr)Z}
  (rt) ++(\a-\e:\r) node (rt2) {\bigl(X\bigl(WY\bigr)\bigr)Z}
  (rb) ++(-\a+\e:\r) node (rb2) {\bigl(\bigl(XW\bigr)Y\bigr)Z}
  (rbb) ++(-\a-\e:\r) node (rbb2) {\bigl(\bigl(WX\bigr)Y\bigr)Z}
  (lbb) ++(180+\a+\e:\r) node (lbb2) {W\bigl(\bigl(XY\bigr)Z\bigr)}
  (lb) ++(180+\a-\e:\r) node (lb2) {\bigl(\bigl(XY\bigr)Z\bigr)W}
  node[between=ltt2 and rtt2 at .5] (mt) {\bigl(\bigl(XY\bigr)W\bigr)Z}
  node[between=lbb2 and rbb2 at .5] (mb) {\bigl(W\bigl(XY\bigr)\bigr)Z}
  (lt2.east) ++(0:-.02) node (lt2') {}
  (ltt2.south) ++(-90:.00) node (ltt2') {}
  (rb.south) ++(0:.14) node (rb') {}
  (rbb.east) ++(-.05,.05) node (rbb') {}
  ;
  \draw[1cell] 
  (lt) edge['] node {1\bigl(1\dB\bigr)} (ltt)
  (ltt) edge['] node {1\dAdot} (rtt)
  (rtt) edge['] node {1\bigl(\dB 1\bigr)} (rt)
  (rt) edge['] node {1\dA} (r)
  (lt) edge node {1\dAdot} (l)
  (l) edge node {1_X\,\dB_{YZ,W}} (r)
  (r) edge['] node {\dAdot} (rb)
  (rb) edge['] node {\dB_{X,W} 1_{YZ}} (rbb)
  (l) edge node {\dAdot} (lb)
  (lb) edge node {\dB} (lbb)
  (lbb) edge node {\dAdot} (rbb)
  (lt2) edge node (1be) {1_{XY}\,\dB_{ZW}} (ltt2)
  (ltt2) edge node {\dAdot} (mt)
  (mt) edge node {\dA1} (rtt2)
  (rtt2) edge node {\bigl(1\dB_{Y,W}\bigr)1} (rt2)
  (rt2) edge node {\dAdot 1} (rb2)
  (rb2) edge node (be11) {\bigl(\dB_{X,W} 1\bigr)1} (rbb2)
  (lt2) edge['] node {\dAdot} (lb2)
  (lb2) edge['] node {\dB_{(XY)Z,W}} (lbb2)
  (lbb2) edge['] node {\dAdot} (mb)
  (mb) edge['] node {\dAdot 1} (rbb2)
  (lt2) edge['] node {\dA} (lt)
  (ltt2) edge node {\dA} (ltt)
  (rtt2) edge['] node {\dA} (rtt)
  (rt2) edge node {\dA} (rt)
  (lb) edge['] node {\dAdot 1} (lb2)
  (lbb) edge node {1\dAdot} (lbb2)
  (rbb) edge['] node {\dAdot} (rbb2)
  (rb) edge node {\dAdot} (rb2)
  (lt2') edge[',bend right=20,looseness=1] node[pos=.75] (11be) {\bigl(11\bigr)\,\dB} (ltt2')
  (rb') edge[bend left=30,looseness=1] node[pos=.15] (be11) {\dB\,\bigl(11\bigr)} (rbb')
  ;
  \draw[2cell] 
  (lt2) ++(38:1) node[shift={(-45:.0)}, rotate=-45, anno] (T1) {=}
  (rb) ++(245:.5) node[shift={(0,0)}, rotate=135, anno] (T2) {=}
  node[between=lt and rt at .5, shift={(0,-.2)}, rotate=225, 2label={below left,1R^1_{Y,Z|W}}] {\Rightarrow}
  node[between=lb and rb at .5, anno] {R = 1}
  (mt) ++(0,-.45) node[rotate=-135, 2label={above left,\pi_3^\inv}] {\Rightarrow}
  (mb) ++(0,.45) node[rotate=-90, 2label={above,\pi_1}] {\Rightarrow}
  (l) ++(-.6,0) node[rotate=180, 2label={below,\pi_2}] {\Rightarrow}
  (r) ++(.6,0) node[rotate=180, 2label={below,\pi_4^\inv}] {\Rightarrow}
  node[between=rt and rtt2 at .4, shift={(.1,0)}, anno] {\natA}
  node[between=lb and lbb2 at .5, shift={(0,0)}, anno] {\natB}
  (lt) ++(85:.5) node[anno] {\natA}
  (rbb2) ++(80:.6) node[anno] {\natAdot}
  ;
  \draw[font=\Huge] (0,-3) node [rotate=90] (eq) {$=$};
  \begin{scope}[yscale=1]
  \def\sh{++(0,-6)}
  \def\S{++(0,1)}  
  \def\T{++(0,.4)}  
  \draw[0cell=\smallzero] 
  (lt2)\sh\T node (lt2B) {\bigl(XY\bigr)\bigl(ZW\bigr)}
  (ltt2)\sh node (ltt2B) {\bigl(XY\bigr)\bigl(WZ\bigr)}
  (rtt2)\sh node (rtt2B) {\bigl(X\bigl(YW\bigr)\bigr)Z}
  (rt2)\sh\T node (rt2B) {\bigl(X\bigl(WY\bigr)\bigr)Z}
  (rb2)\sh\T\S node (rb2B) {\bigl(\bigl(XW\bigr)Y\bigr)Z}
  (rbb2)\sh\T\S\T node (rbb2B) {\bigl(\bigl(WX\bigr)Y\bigr)Z}
  (lbb2)\sh\T\S\T node (lbb2B) {W\bigl(\bigl(XY\bigr)Z\bigr)}
  (lb2)\sh\T\S node (lb2B) {\bigl(\bigl(XY\bigr)Z\bigr)W}
  (mt)\sh node (mtB) {\bigl(\bigl(XY\bigr)W\bigr)Z}
  (mb)\sh\T\S\T node (mbB) {\bigl(W\bigl(XY\bigr)\bigr)Z}
  ;
  \draw[1cell] 
  (lt2B) edge node (1be) {1_{XY}\,\dB_{Z,W}} (ltt2B)
  (ltt2B) edge node {\dAdot} (mtB)
  (mtB) edge node {\dA1} (rtt2B)
  (rtt2B) edge node {\bigl(1\dB_{Y,W}\bigr)1} (rt2B)
  (rt2B) edge node {\dAdot 1} (rb2B)
  (rb2B) edge node (be11) {\bigl(\dB_{X,W} 1\bigr)1} (rbb2B)
  (lt2B) edge['] node {\dAdot} (lb2B)
  (lb2B) edge['] node {\dB_{(XY)Z,W}} (lbb2B)
  (lbb2B) edge['] node {\dAdot} (mbB)
  (mbB) edge['] node {\dAdot 1} (rbb2B)
  (mtB) edge node[pos=.33] {\dB_{XY,W}\,1_Z} (mbB)
  ;
  \draw[2cell] 
  node[between=ltt2B and lbb2B at .5, anno] {R = 1}
  node[between=rtt2B and rbb2B at .5, shift={(0,0)}, rotate=225,
    2label={below left,R^2_{X,Y|W}1}] {\Rightarrow}
  ;  
  \end{scope}
\end{tikzpicture}
\end{equation}

\noindent\textbf{(1,3)-Crossing Axiom}
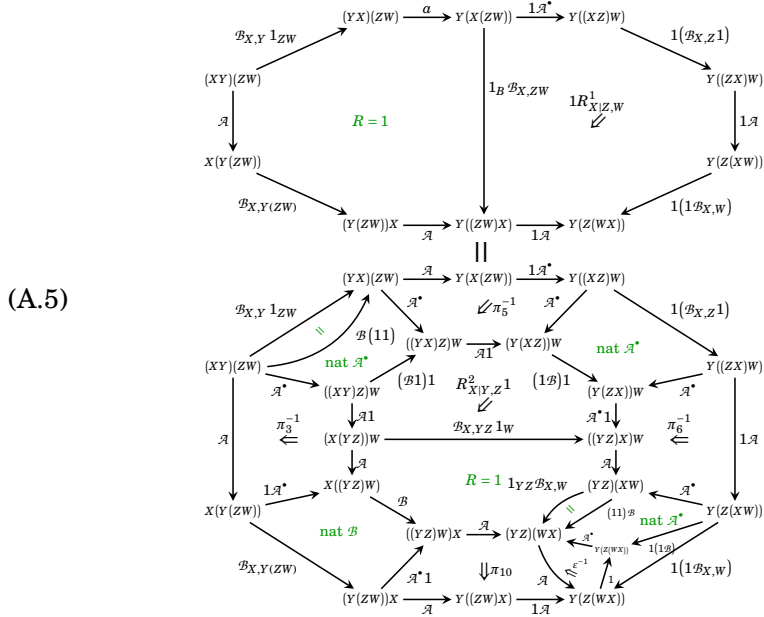
\begin{equation}\label{eq:13-cross-axiom}
\begin{tikzpicture}[x=23mm,y=13.6mm,vcenter,scale=.61]
  \smallnodes
  \def\r{1.2}
  \def\s{1.1}
  \def\a{40}
  \def\e{20}
  \def\v{.75}
  \def\h{1.25}
  \draw[0cell=\smallzero] 
  (-\h,0) node (l) {\bigl(X\bigl(YZ\bigr)\bigr)W}
  (l) ++(0,\v) node (lt) {\bigl(\bigl(XY\bigr)Z\bigr)W}
  (lt) ++(45:\s) node (ltt) {\bigl(\bigl(YX\bigr)Z\bigr)W}
  (l) ++(0,-\v) node (lb) {X\bigl(\bigl(YZ\bigr)W\bigr)}
  (lb) ++(-45:\s) node (lbb) {\bigl(\bigl(YZ\bigr)W\bigr)X}
  (\h,0) node (r) {\bigl(\bigl(YZ\bigr)X\bigr)W}
  (r) ++(0,\v) node (rt) {\bigl(Y\bigl(ZX\bigr)\bigr)W}
  (rt) ++(135:\s) node (rtt) {\bigl(Y\bigl(XZ\bigr)\bigr)W}
  (r) ++(0,-\v) node (rb) {\bigl(YZ\bigr)\bigl(XW\bigr)}
  (rb) ++(-135:\s) node (rbb) {\bigl(YZ\bigr)\bigl(WX\bigr)}
  (lt) ++(180-\a+\e:\r) node (lt2) {\bigl(XY\bigr)\bigl(ZW\bigr)}
  (ltt) ++(180-\a-\e:\r) node (ltt2) {\bigl(YX\bigr)\bigl(ZW\bigr)}
  (rtt) ++(\a+\e:\r) node (rtt2) {Y\bigl(\bigl(XZ\bigr)W\bigr)}
  (rt) ++(\a-\e:\r) node (rt2) {Y\bigl(\bigl(ZX\bigr)W\bigr)}
  (rb) ++(-\a+\e:\r) node (rb2) {Y\bigl(Z\bigl(XW\bigr)\bigr)}
  (rbb) ++(-\a-\e:\r) node (rbb2) {Y\bigl(Z\bigl(WX\bigr)\bigr)}
  (lbb) ++(180+\a+\e:\r) node (lbb2) {\bigl(Y\bigl(ZW\bigr)\bigr)X}
  (lb) ++(180+\a-\e:\r) node (lb2) {X\bigl(Y\bigl(ZW\bigr)\bigr)}
  node[between=ltt2 and rtt2 at .5] (mt) {Y\bigl(X\bigl(ZW\bigr)\bigr)}
  node[between=lbb2 and rbb2 at .5] (mb) {Y\bigl(\bigl(ZW\bigr)X\bigr)}
  (lt2.east) ++(0:-.02) node (lt2') {}
  (ltt2.south) ++(-90:.00) node (ltt2') {}
  (rb.south) ++(0:.0) node (rb') {}
  (rbb.east) ++(-.05,.05) node (rbb') {}  
  ;
  \draw[1cell] 
  (lt) edge['] node {\bigl(\dB 1\bigr)1} (ltt)
  (ltt) edge['] node {\dA1} (rtt)
  (rtt) edge['] node {\bigl(1\dB\bigr)1} (rt)
  (rt) edge['] node {\dAdot 1} (r)
  (lt) edge node {\dA1} (l)
  (l) edge node {\dB_{X,YZ}\,1_W} (r)
  (r) edge['] node {\dA} (rb)
  (rb) edge[',bend right=20] node[pos=.2] {1_{YZ} \dB_{X,W}} (rbb)
  (l) edge node {\dA} (lb)
  (lb) edge node {\dB} (lbb)
  (lbb) edge node {\dA} (rbb)
  (lt2) edge node (1be) {\dB_{X,Y}\,1_{ZW}} (ltt2)
  (ltt2) edge node {\dA} (mt)
  (mt) edge node {1\dAdot} (rtt2)
  (rtt2) edge node {1\bigl(\dB_{X,Z}1\bigr)} (rt2)
  (rt2) edge node {1\dA} (rb2)
  (rb2) edge node (be11) {1\bigl(1\dB_{X,W}\bigr)} (rbb2)
  (lt2) edge['] node {\dA} (lb2)
  (lb2) edge['] node {\dB_{X,Y(ZW)}} (lbb2)
  (lbb2) edge['] node {\dA} (mb)
  (mb) edge['] node {1\dA} (rbb2)
  (lt2) edge['] node {\dAdot} (lt)
  (ltt2) edge node {\dAdot} (ltt)
  (rtt2) edge['] node {\dAdot} (rtt)
  (rt2) edge node {\dAdot} (rt)
  (lb2) edge node {1\dAdot} (lb)
  (lbb2) edge['] node {\dAdot 1} (lbb)
  (rbb) edge[',bend right=20] node {\dA} (rbb2)
  (rb2) edge['] node {\dAdot} (rb)
  (lt2') edge[',bend right=30,looseness=1] node[pos=.75] (be11) {\dB\,\bigl(11\bigr)} (ltt2')
  (rb') edge[bend left=0,looseness=1] node[pos=.15,scale=.7] (11be) {\bigl(11\bigr)\,\dB} (rbb')
  ;
  \draw[2cell] 
  (lt2) ++(35:1) node[shift={(0,0)}, rotate=-45, anno] (T1) {=}
  (rb) ++(220:.55) node[shift={(0,0)}, rotate=135, anno] (T2) {=}
  node[between=lt and rt at .5, shift={(0,-.2)}, rotate=225,
    2label={below left,R^2_{X|Y,Z}1}] {\Rightarrow}
  node[between=lb and rb at .5, shift={(0,.1)}, anno] {R = 1}
  (mt) ++(0,-.45) node[rotate=-135, 2label={above left,\pi_5^\inv}] {\Rightarrow}
  (mb) ++(0,.45) node[rotate=-90, 2label={above,\pi_{10}}] {\Rightarrow}
  (l) ++(-.6,0) node[rotate=180, 2label={below,\pi_3^\inv}] {\Rightarrow}
  (r) ++(.6,0) node[rotate=180, 2label={below,\pi_6^\inv}] {\Rightarrow}
  node[between=rt and rtt2 at .4, shift={(.1,0)}, anno] {\natAdot}
  node[between=lb and lbb2 at .4, shift={(-.2,0)}, anno] {\natB}
  (lt) ++(95:.5) node[anno] {\natAdot}
  (rb2) ++(190:.7) node[anno] {\natAdot}
  ;
  \draw[0cell=\smallzero]
  (rbb2) ++(80:.8) node[scale=.6] (P) {Y \bigl(Z \bigl(WX\bigr)\bigr)}
  ;
  \draw[1cell=.7]
  (rbb2) edge['] node {1} (P)
  (P) edge['] node {\dAdot} (rbb)
  (rb2) edge node[pos=.8] {1\bigl(1\dB\bigr)} (P)
  ;
  \draw[2cell=.7]
  (rbb2)+(120:.5) node[rotate=155, 2label={below,\epzainv}] {\Rightarrow}
  ;
  \draw[font=\Huge] (0,3) node [rotate=90] (eq) {$=$};
  \begin{scope}[yscale=1]
    \def\sh{++(0,6)}
    \def\shT{++(0,5.6)}
    \def\shTS{++(0,4.6)}
    \def\shTST{++(0,4.2)}
  \draw[0cell=\smallzero] 
  (lt2)\shTS node (lt2B) {\bigl(XY\bigr)\bigl(ZW\bigr)}
  (ltt2)\shTST node (ltt2B) {\bigl(YX\bigr)\bigl(ZW\bigr)}
  (rtt2)\shTST node (rtt2B) {Y\bigl(\bigl(XZ\bigr)W\bigr)}
  (rt2)\shTS node (rt2B) {Y\bigl(\bigl(ZX\bigr)W\bigr)}
  (rb2)\shT node (rb2B) {Y\bigl(Z\bigl(XW\bigr)\bigr)}
  (rbb2)\sh node (rbb2B) {Y\bigl(Z\bigl(WX\bigr)\bigr)}
  (lbb2)\sh node (lbb2B) {\bigl(Y\bigl(ZW\bigr)\bigr)X}
  (lb2)\shT node (lb2B) {X\bigl(Y\bigl(ZW\bigr)\bigr)}
  (mt)\shTST node (mtB) {Y\bigl(X\bigl(ZW\bigr)\bigr)}
  (mb)\sh node (mbB) {Y\bigl(\bigl(ZW\bigr)X\bigr)}
  ;
  \draw[1cell] 
  (lt2B) edge node (1be) {\dB_{X,Y}\,1_{ZW}} (ltt2B)
  (ltt2B) edge node {a} (mtB)
  (mtB) edge node {1\dAdot} (rtt2B)
  (rtt2B) edge node {1\bigl(\dB_{X,Z}1\bigr)} (rt2B)
  (rt2B) edge node {1\dA} (rb2B)
  (rb2B) edge node (be11) {1\bigl(1\dB_{X,W}\bigr)} (rbb2B)
  (lt2B) edge['] node {\dA} (lb2B)
  (lb2B) edge['] node {\dB_{X,Y(ZW)}} (lbb2B)
  (lbb2B) edge['] node {\dA} (mbB)
  (mbB) edge['] node {1\dA} (rbb2B)
  (mtB) edge node[pos=.33] {1_B\,\dB_{X,ZW}} (mbB)
  ;
  \draw[2cell] 
  node[between=ltt2B and lbb2B at .5, shift={(0,0)}, anno] {R = 1}
  node[between=rtt2B and rbb2B at .5, shift={(0,0)}, rotate=225, 2label={below left,1R^1_{X|Z,W}}] {\Rightarrow}
  ;
  \end{scope}
\end{tikzpicture}
\end{equation}

\noindent\textbf{(2,2)-Crossing Axiom}
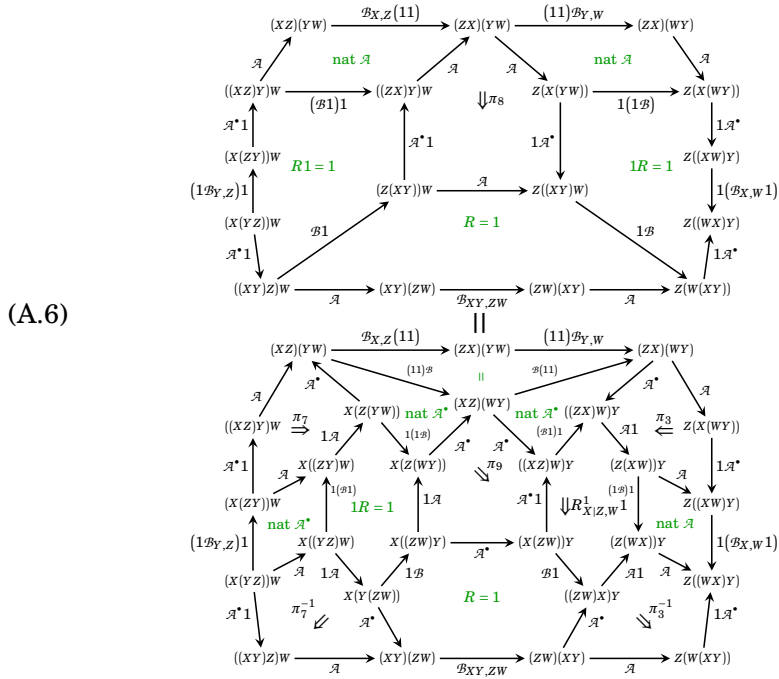
\begin{equation}\label{eq:22-cross-axiom}
  \begin{tikzpicture}[xscale=10,yscale=1.2,baseline={(eq.base)},scale=.61]
    \smallnodes
\def\xs{\tiny}
\def\m{.6} \def\v{1} \def\u{1.5} \def\t{-4.7}
\def\c{.2} \def\d{.1} \def\e{.7} \def\f{.26} \def\g{.02} \def\h{1} \def\i{.1}
\draw[font=\Huge] (0,0) node [rotate=90] (eq) {$=$};
\newcommand{\boundary}{
\draw[0cell=\smallzero] 
(0,0) node (x1) {\bigl(X\bigl(YZ\bigr)\bigr)W}
($(x1)+(0,\v)$) node (x2) {\bigl(X\bigl(ZY\bigr)\bigr)W} 
($(x2)+(0,\v)$) node (x3) {\bigl(\bigl(XZ\bigr)Y\bigr)W}
($(x3)+(\i,\v)$) node (x4) {\bigl(XZ\bigr)\bigl(YW\bigr)}
($(x1)+(\h,0)$) node (x9) {Z\bigl(\bigl(WX\bigr)Y\bigr)}
($(x9)+(0,\v)$) node (x8) {Z\bigl(\bigl(XW\bigr)Y\bigr)}
($(x8)+(0,\v)$) node (x7) {Z\bigl(X\bigl(WY\bigr)\bigr)} 
($(x7)+(-\i,\v)$) node (x6) {\bigl(ZX\bigr)\bigl(WY\bigr)}
($(x4)!.5!(x6)$) node (x5) {\bigl(ZX\bigr)\bigl(YW\bigr)} 
($(x1)+(\g,-\v)$) node (y1) {\bigl(\bigl(XY\bigr)Z\bigr)W} 
($(x9)+(-\g,-\v)$) node (y4) {Z\bigl(W\bigl(XY\bigr)\bigr)}
($(y1)!.33!(y4)$) node (y2) {\bigl(XY\bigr)\bigl(ZW\bigr)}
($(y1)!.67!(y4)$) node (y3) {\bigl(ZW\bigr)\bigl(XY\bigr)} 
;
\draw[1cell] 
(x1) edge node {\bigl(1\dB_{Y,Z}\bigr)1} (x2)
(x2) edge node {\dAdot 1} (x3)
(x3) edge node[pos=.2] {\dA} (x4)
(x4) edge node {\dB_{X,Z}\bigl(11\bigr)} (x5)
(x5) edge node {\bigl(11\bigr)\dB_{Y,W}} (x6)
(x6) edge node[pos=.7] {\dA} (x7)
(x7) edge node {1\dAdot} (x8)
(x8) edge node {1\bigl(\dB_{X,W}1\bigr)} (x9)
(x1) edge node[swap,pos=.2] {\dAdot 1} (y1) 
(y1) edge node[swap] {\dA} (y2)
(y2) edge node[swap] {\dB_{XY,ZW}} (y3)
(y3) edge node[swap] {\dA} (y4)
(y4) edge node[swap,pos=.8] {1\dAdot} (x9)
;}
\begin{scope}[shift={(-\h/2,1.2*\u)},yscale=1.2]
\smallnodes
\boundary 
\draw[0cell=\smallzero]
($(x1)+(.33,\v/2)$) node (z1) {\bigl(Z\bigl(XY\bigr)\bigr)W}
($(x1)+(.67,\v/2)$) node (z4) {Z\bigl(\bigl(XY\bigr)W\bigr)} 
($(x3)!.33!(x7)$) node (z2) {\bigl(\bigl(ZX\bigr)Y\bigr)W}
($(x3)!.67!(x7)$) node (z3) {Z\bigl(X\bigl(YW\bigr)\bigr)} 
;
\draw[1cell] 
(y1) edge node {\dB 1} (z1)
(z1) edge node {\dA} (z4)
(z1) edge node[swap] {\dAdot 1} (z2)
(x3) edge node[swap] {\bigl(\dB 1\bigr)1} (z2)
(z2) edge node[swap] {\dA} (x5)
(x5) edge node[swap] {\dA} (z3)
(z3) edge node[swap] {1\bigl(1\dB\bigr)} (x7)
(z3) edge node[swap] {1\dAdot} (z4)
(z4) edge node {1\dB} (y4)
;
\draw[2cell]
node[between=x4 and z2 at .5, shift={(0,0)}, anno] {\natA}
node[between=z3 and x6 at .5, shift={(0,0)}, anno] {\natA}
node[between=x1 and z1 at .4, shift={(0,1)}, anno] {R1 = 1}
node[between=z4 and x9 at .6, shift={(0,1)}, anno] {1R = 1}
node[between=z1 and z4 at .5, shift={(0,2)}, rotate=-90, 2label={above,\pi_8}] {\Rightarrow}
node[between=y2 and y3 at .5, shift={(0,1.5)}, anno] {R = 1}
;
\end{scope}  
\begin{scope}[shift={(-\h/2,\t)},yscale=1.4]
\boundary 
\draw[0cell=\smallzero]
($(x5)+(0,-.7)$) node (z0) {\bigl(XZ\bigr)\bigl(WY\bigr)} 
($(x3)+(\f,\c)$) node (za1) {X\bigl(Z\bigl(YW\bigr)\bigr)} 
($(za1)+(-\d,-\e)$) node (za2) {X\bigl(\bigl(ZY\bigr)W\bigr)} 
($(za1)+(\d,-\e)$) node (za3) {X\bigl(Z\bigl(WY\bigr)\bigr)}
($(za2)+(0,-\v)$) node (za4) {X\bigl(\bigl(YZ\bigr)W\bigr)}
($(za3)+(0,-\v)$) node (za5) {X\bigl(\bigl(ZW\bigr)Y\bigr)}
($(za4)+(\d,-\e)$) node (za6) {X\bigl(Y\bigl(ZW\bigr)\bigr)}
($(x3)+(1-\f,\c)$) node (zb1) {\bigl(\bigl(ZX\bigr)W\bigr)Y} 
($(zb1)+(-\d,-\e)$) node (zb2) {\bigl(\bigl(XZ\bigr)W\bigr)Y} 
($(zb1)+(\d,-\e)$) node (zb3) {\bigl(Z\bigl(XW\bigr)\bigr)Y}
($(zb2)+(0,-\v)$) node (zb4) {\bigl(X\bigl(ZW\bigr)\bigr)Y}
($(zb3)+(0,-\v)$) node (zb5) {\bigl(Z\bigl(WX\bigr)\bigr)Y}
($(zb4)+(\d,-\e)$) node (zb6) {\bigl(\bigl(ZW\bigr)X\bigr)Y}
;
\draw[1cell] 
(x4) edge node[font=\xs,pos=.6] {\bigl(11\bigr)\dB} (z0) 
(z0) edge node[font=\xs,pos=.4] {\dB\bigl(11\bigr)} (x6)
(za4) edge node[',font=\xs,pos=.75] {1\bigl(\dB 1\bigr)} (za2) 
(za2) edge node[pos=.3] {1\dA} (za1)
(za1) edge node[font=\xs,pos=.7] {1\bigl(1\dB\bigr)} (za3)
(za4) edge node[swap,pos=.3] {1\dA} (za6) 
(za6) edge node[swap,pos=.7] {1\dB} (za5)
(za5) edge node[swap,pos=.6] {1\dA} (za3)
(x1) edge node[swap] {\dA} (za4) 
(x2) edge node[pos=.6] {\dA} (za2) 
(za1) edge node[pos=.7] {\dAdot} (x4) 
(za3) edge node[swap] {\dAdot} (z0) 
(za5) edge node[swap] {\dAdot} (zb4) 
(za6) edge node[swap,pos=.1] {\dAdot} (y2) 
(zb4) edge node[pos=.6] {\dAdot 1} (zb2) 
(zb2) edge node[font=\xs,pos=.4] {\bigl(\dB 1\bigr)1} (zb1)
(zb1) edge node {\dA1} (zb3)
(zb3) edge node[font=\xs,swap,pos=.25] {\bigl(1\dB\bigr)1} (zb5)
(zb4) edge node[swap,pos=.3] {\dB 1} (zb6) 
(zb6) edge node[swap,pos=.7] {\dA1} (zb5)
(z0) edge node[swap] {\dAdot} (zb2) 
(x6) edge node[pos=.3] {\dAdot} (zb1)
(zb3) edge node[pos=.4] {\dA} (x8)
(zb5) edge node[swap] {\dA} (x9)
(y3) edge node[swap,pos=.9] {\dAdot} (zb6)
;
\draw[2cell]
node[between=x5 and z0 at .5, shift={(-.0,0)}, rotate=90, anno] {=}
node[between=x3 and za1 at .4, shift={(0,-.2)}, 2label={above,\pi_7}] {\Rightarrow}
node[between=x7 and zb1 at .4, shift={(0,-.2)}, rotate=180, 2label={below,\pi_3}] {\Rightarrow}
node[between=za1 and z0 at .5, shift={(0,-.1)}, anno] {\natAdot}
node[between=zb1 and z0 at .5, shift={(0,-.1)}, anno] {\natAdot}
node[between=x2 and za4 at .5, shift={(0,0)}, anno] {\natAdot}
node[between=x8 and zb5 at .5, shift={(0,0)}, anno] {\natA}
node[between=za4 and za5 at .5, shift={(0,.8)}, anno] {1R = 1}
node[between=za5 and zb4 at .5, shift={(0,1.5)}, rotate=-45, 2label={above,\pi_9}] {\Rightarrow}
node[between=zb4 and zb5 at .3, shift={(-.2,.8)}, rotate=-90, 2label={above,R^1_{X|Z,W}1}] {\Rightarrow}
node[between=za6 and zb6 at .5, shift={(0,0)}, anno] {R = 1}
node[between=y1 and y2 at .4, shift={(0,.8)}, rotate=225, 2label={below,\pi_7^{-1}}] {\Rightarrow}
node[between=y3 and y4 at .6, shift={(0,.8)}, rotate=-45, 2label={above,\pi_3^{-1}}] {\Rightarrow}
;
\end{scope}     
\end{tikzpicture}
\end{equation}

\noindent\textbf{Yang-Baxter Axiom}
\begin{equation}\label{eq:YB-axiom}
  \begin{tikzpicture}[x=17mm,y=17mm,baseline={(eq.base)},scale=.7]
    \smallnodes
\def\m{1.4} \def\h{1}
\def\a{45} \def\d{1.2} \def\e{.75}
\draw[font=\Huge] (0,-.707-\d-.25*\e) node [rotate=0] (eq) {$=$};
\newcommand{\boundary}{
\draw[0cell=\smallzero] 
(0,0) node (x1) {\bigl(XY\bigr)Z}
($(x1)+(-3*\a:1)$) node (x2) {\bigl(YX\bigr)Z} 
($(x2)+(-2*\a:\d)$) node (x3) {Y\bigl(XZ\bigr)} 
($(x3)+(-2*\a:\e)$) node (x4) {Y\bigl(ZX\bigr)}
($(x4)+(-2*\a:\d)$) node (x5) {\bigl(YZ\bigr)X}
($(x5)+(-1*\a:1)$) node (x6) {\bigl(ZY\bigr)X}
($(x6)+(-0*\a:1)$) node (x7) {Z\bigl(YX\bigr)}
($(x1)+(0*\a:1)$) node (y1) {X\bigl(YZ\bigr)} 
($(y1)+(-1*\a:1)$) node (y2) {X\bigl(ZY\bigr)}
($(y2)+(-2*\a:\d)$) node (y3) {\bigl(XZ\bigr)Y}
($(y3)+(-2*\a:\e)$) node (y4) {\bigl(ZX\bigr)Y}
($(y4)+(-2*\a:\d)$) node (y5) {Z\bigl(XY\bigr)}
;
\draw[1cell] 
(x1) edge['] node {\dB_{X,Y}1} (x2)
(x2) edge['] node {\dA} (x3)
(x3) edge['] node[pos=.5] {1\dB_{X,Z}} (x4)
(x4) edge['] node[pos=.5] {\dAdot} (x5)
(x5) edge['] node {\dB_{Y,Z}1} (x6)
(x6) edge['] node {\dA} (x7)
(x1) edge node {\dA} (y1) 
(y1) edge node {1\dB_{Y,Z}} (y2)
(y2) edge node {\dAdot} (y3)  
(y3) edge node[pos=.5] {\dB_{X,Z}1} (y4) 
(y4) edge node {\dA} (y5) 
(y5) edge node {1\dB_{X,Y}} (x7) 
;}
\begin{scope}[shift={(-\h-\m,0)},rotate=0*\a]
\boundary 
\draw[0cell=\smallzero]
node[between=x5 and y1 at .4,shift={(.4,0)}] (z) {\bigl(YZ\bigr)X} 
;
\draw[1cell] 
(y1) edge node[swap,pos=.7] {\dB} (z)
(z) edge node[pos=.5] {\dA} (x4)
(z) edge node[pos=.5] (one) {1} (x5)
(z) edge[bend right=0] node[pos=.35] {\dB 1} (x6)
(y2) edge[bend right=0] node[pos=.4] {\dB} (x6)
;
\draw[2cell]
node[between=x1 and z at .5, shift={(-.2,0)}, anno] {R = 1}
node[between=x4 and one at .5, shift={(0,0)}, rotate=-0, 2label={below,\etaainv}] {\Rightarrow}
node[between=one and x6 at .25, shift={(0,0)}, rotate=0, anno] {=}
node[between=y2 and z at .5, anno] {\natB}
node[between=z and y5 at .55, shift={(0,-.5)}, rotate=-0, 2label={above,\bigl(R^1_{X|Z,Y}\bigr)^{-1}}] {\Rightarrow}
;
\end{scope}  
\begin{scope}[shift={(\m,0)}]
\boundary 
\draw[0cell=\smallzero]
node[between=x2 and x7 at .4,shift={(.4,0)}] (z) {\bigl(YX\bigr)Z}
;
\draw[1cell] 
(x1) edge node[pos=.65] {\dB 1} (z)
(x2) edge node[pos=.5] (one) {1} (z)
(x3) edge node[',pos=.5] {\dAdot} (z)
(z) edge node[',pos=.5] {\dB} (x7)
(x1) edge node[pos=.5] {\dB} (y5)
;
\draw[2cell]
node[between=z and x6 at .5, shift={(-.2,0)}, rotate=0, 2label={above,R^1_{Y,X|Z}}] {\Rightarrow}
node[between=x3 and one at .5, shift={(-30:0)}, rotate=0, 2label={above,\hspace{-2ex}\etaainv}] {\Rightarrow}
node[between=x1 and one at .6, shift={(-.1,0)}, anno] {=}
node[between=z and y5 at .45, shift={(0,0)}, anno] {\natB}
node[between=y2 and z at .45, shift={(0,.1)}, rotate=0, 2label={above,\bigl(R^3_{X,Y|Z}\bigr)^{-1}}] {\Rightarrow}
;
\end{scope}  
\end{tikzpicture}
\end{equation}

\noindent\textbf{(2,1)-Syllepsis Axiom}
\begin{equation}\label{eq:21-syl-axiom}
  \begin{tikzpicture}[x=30mm,y=30mm,scale=.5,baseline={(eq.base)}]
    \smallnodes
\def\w{1.3}
\def\a{360/7}
\draw[font=\Huge] (0,0) node [rotate=0] (eq) {$=$};
\newcommand{\boundary}{
\draw[0cell=\smallzero] 
(90:1) node (x1) {X\bigl(YZ\bigr)}
(90+\a:1) node (x2) {X\bigl(ZY\bigr)} 
(90+2*\a:1) node (x3) {\bigl(XZ\bigr)Y} 
(90+3*\a:1) node (x4) {\bigl(ZX\bigr)Y}
(90+4*\a:1) node (x5) {Z\bigl(XY\bigr)}
(90+5*\a:1) node (x6) {\bigl(XY\bigr)Z}
(90+6*\a:1) node (x7) {X\bigl(YZ\bigr)}
;
\draw[1cell] 
(x1) edge node (one) {1} (x7) 
(x1) edge['] node {1\dB_{Y,Z}} (x2)
(x2) edge node[swap] {\dAdot} (x3)
(x3) edge node[swap] (bAC1) {\dB_{X,Z}1} (x4)
(x4) edge node[swap] (a) {\dA} (x5)
(x5) edge['] node[pos=.5] {\dB_{Z,XY}} (x6)
(x6) edge['] node[pos=.5] {\dA} (x7)
;}
\begin{scope}[shift={(-\w,0)}]
\boundary 
\draw[0cell=\smallzero]
($(x2)+(-18:.9)$) node (z1) {\bigl(XY\bigr)Z} 
($(z1)+(180+3*180/7:.7)$) node (z2) {Z\bigl(XY\bigr)} 
;
\draw[1cell] 
(x1) edge node[swap,pos=.4] {\dAdot} (z1)
(z1) edge node[swap,pos=.5] (aABC) {\dA} (x7)
(z1) edge node[pos=.5,swap] {1} (x6)
(z1) edge node[swap,pos=.6] {\dB} (z2)
(z2) edge node[pos=.5] (B) {\dB} (x6)
(z2) edge node[swap] {\dAdot} (x4)
(z2) edge node[pos=.6] {1} (x5)
;
\draw[2cell]
node[between=x3 and z1 at .5, shift={(-51:0)}, anno] {R = 1}
node[between=z2 and a at .55, shift={(0,0)}, rotate=20, 2label={above,\epza}] {\Rightarrow}
node[between=x5 and B at .5, shift={(25:.1)}, rotate=110, anno] {=}
node[between=z1 and B at .6, shift={(231:.05)}, anno] {\syl = 1}
node[between=aABC and x6 at .3, shift={(0,0)}, rotate=110, anno] {=}
node[between=x1 and aABC at .6, shift={(0,0)}, rotate=70, 2label={above,\epza}] {\Rightarrow}
;
\end{scope}  
\begin{scope}[shift={(\w,0)}]
\boundary 
\draw[0cell=\smallzero]
($(x5)+(180-40:.9)$) node (z2) {\bigl(XZ\bigr)Y} 
($(z2)+(2*180/7:.7)$) node (z1) {X\bigl(ZY\bigr)} 
;
\draw[1cell] 
(x1) edge node[swap,pos=.4] (1B) {1\dB} (z1)
(z1) edge node[swap,pos=.4] {1\dB} (x7)
(x2) edge node[swap,pos=.5] {1} (z1)
(x2) edge['] node[pos=.4] (abdot) {\dAdot} (z2)
(x3) edge node[pos=.5] {1} (z2)
(x4) edge node[swap,pos=.5] {\dB 1} (z2)
(z2) edge node[swap,pos=.5] {\dA} (z1)
;
\draw[2cell]
node[between=x6 and z2 at .5, shift={(180-51:0)}, anno] {R = 1}
node[between=bAC1 and z2 at .55, shift={(-60:.05)}, anno] {\syl 1 = 1}
node[between=x3 and abdot at .6, shift={(-60:.05)}, rotate=30, anno] {=}
node[between=abdot and z1 at .55, shift={(-.02,-.075)}, rotate=80, 2label={above,\epza}] {\Rightarrow}
node[between=x2 and 1B at .65, shift={(0,.0)}, rotate=10, anno] {=}
node[between=1B and x7 at .45, shift={(0,-.0)}, anno] {1\syl = 1}
;
\end{scope}  
\end{tikzpicture}
\end{equation}

\noindent\textbf{(1,2)-Syllepsis Axiom}

\begin{equation}\label{eq:12-syl-axiom}
  \begin{tikzpicture}[x=30mm,y=30mm,scale=.5,baseline={(eq.base)}]
  \smallnodes
\def\w{1.3}
\def\a{360/7}
\draw[font=\Huge] (0,0) node [rotate=0] (eq) {$=$};
\newcommand{\boundary}{
\draw[0cell=\smallzero] 
(90:1) node (x1) {Y\bigl(XZ\bigr)}
(90+\a:1) node (x2) {\bigl(YX\bigr)Z} 
(90+2*\a:1) node (x3) {\bigl(XY\bigr)Z} 
(90+3*\a:1) node (x4) {\bigl(XY\bigr)Z}
(90+4*\a:1) node (x5) {X\bigl(YZ\bigr)}
(90+5*\a:1) node (x6) {\bigl(YZ\bigr)X}
(90+6*\a:1) node (x7) {Y\bigl(ZX\bigr)}
;
\draw[1cell] 
(x7) edge['] node (one) {1\dB_{Z,X}} (x1)
(x2) edge['] node {\dB_{Y,X}1} (x3)
(x1) edge node[swap] {\dAdot} (x2)
(x4) edge['] node[swap] (bAC1) {1} (x3)
(x4) edge node[swap] (a) {\dA} (x5)
(x5) edge['] node[pos=.5] {\dB_{X,YZ}} (x6)
(x6) edge['] node[pos=.5] {\dA} (x7)
;}
\begin{scope}[shift={(\w,0)},rotate=-3*360/7]
\boundary 
\draw[0cell=\smallzero]
($(x1)+(-18-360/7:.9)$) node (z1) {\bigl(YZ\bigr)X} 
($(z1)+(8*180/7:.7)$) node (z2) {X\bigl(YZ\bigr)} 
;
\draw[1cell] 
(x7) edge node[',pos=.5] (aABC) {\dAdot} (z1)
(z1) edge node[swap,pos=.5] {\dB} (z2)
(z2) edge node[swap] {\dAdot} (x3)
(x6) edge node[pos=.4] {1} (z1)
(x5) edge node[pos=.5] (B) {\dB} (z1)
(x5) edge node[',pos=.5] {1} (z2)
(x4) edge node[',pos=.6] {\dA} (z2)
;
\draw[2cell]
node[between=x2 and z1 at .5, shift={(-51:0)}, anno] {R = 1}
node[between=bAC1 and z2 at .45, shift={(0,0)}, rotate=80, 2label={above,\etaainv\hspace{-.8ex}}] {\Rightarrow}
node[between=z2 and a at .55, shift={(0,0)}, rotate=20, anno] {=}
node[between=z2 and B at .6, shift={(-90:0.09)}, anno] {\syl = 1}
node[between=x6 and B at .6, shift={(231:.05)}, rotate=30, anno] {=}
node[between=aABC and x6 at .3, shift={(90:0.08)}, rotate=90, 2label={above,\etaainv}] {\Rightarrow}
;
\end{scope}  
\begin{scope}[shift={(-\w,0)}]
\tikzset{rotate=-3*360/7}
\boundary 
\draw[0cell=\smallzero]
($(x5)+(180-40:.9)$) node (z2) {\bigl(YX\bigr)Z} 
($(z2)+(2*180/7:.7)$) node (z1) {Y\bigl(XZ\bigr)} 
;
\draw[1cell] 
(z1) edge node[pos=.5] (1B) {1} (x1)
(z1) edge node[swap,pos=.4] {1\dB} (x7)
(z1) edge node[pos=.5] {\dAdot} (x2)
(z2) edge['] node[pos=.4] (abdot) {1} (x2)
(z2) edge node[',pos=.5] {\dB 1} (x3)
(x4) edge node[swap,pos=.5] {\dB 1} (z2)
(z2) edge node[swap,pos=.5] {\dA} (z1)
;
\draw[2cell]
node[between=x6 and z2 at .5, shift={(180-51:0)}, anno] {R = 1}
node[between=bAC1 and z2 at .4, shift={(-60:.05)}, anno] {\syl 1 = 1}
node[between=x3 and abdot at .6, shift={(0:.05)}, rotate=110, anno] {=}
node[between=abdot and z1 at .35, shift={(0:0)}, rotate=60, 2label={above,\etaainv\hspace{-1.2ex}}] {\Rightarrow}
node[between=x2 and 1B at .7, shift={(0,.0)}, rotate=100, anno] {=}
node[between=1B and x7 at .4, shift={(0,-.0)}, anno] {1\syl = 1}
;
\end{scope}  
\end{tikzpicture}
\end{equation}

\section{Symmetric monoidal 2-functors}
\label{sec:sm2fun-defn}

In this section we give the definition of symmetric monoidal 2-functor
\[
  H \cn (\B,\otimes) \to (\C,\otimes)
\]
under the simplifying assumption that each of the monoidal products $\otimes$ is 2-functorial.
We refer the reader to  \cite{GPS95Coherence,McCru00Balanced,SP2011Classification} for further discussion of the general structures.
Also see \cite{gurski11,GO2012Infinite} for an equivalent presentation and general coherence theorems.

\begin{defn}\label{defn:sm2fun}
  Suppose $(\B,\otimes)$ and $(\C,\otimes)$ are symmetric monoidal 2-categories.
  A \emph{symmetric monoidal 2-functor}
  \[
    H \cn \B \to \C
  \]
  consists of a 2-functor $H$ together with the following data.  We abbreviate the monoidal products as juxtaposition and we write $H(X) = \ol{X}$. 
  \begin{description}
  \item[Pseudonatural equivalences]\
    \begin{itemize}
    \item It has a monoidal constraint $\Phi\cn \ol{X}\ \ol{Y} \to \ol{XY}$.
    \item It has a unit constraint $\Psi\cn S \to \ol{S}$.
    \end{itemize}
  \item[Invertible modifications]\
    \begin{itemize}
    \item It has an associativity modification $\om$ with
    \item left and right unity modifications $\ga^\dL$ and $\ga^\dR$ as follows.
      \[
        \begin{tikzpicture}[x=30mm,y=15mm,vcenter]
          \draw[0cell=.8] 
          (0,0) node (a) {\bigl(  \ol{X} \  \ol{Y} \bigr) \;  \ol{Z}}
          (1,0) node (b) {\ol{X} \; \bigl(  \ol{Y} \  \ol{Z} \bigr)}
          (1,-1) node (c) {\ol{X} \  \ol{Y  Z} }
          (1,-2) node (d) {\ol{X  \bigl( Y  Z \bigr)}}
          (0,-1) node (b') {\ol{X  Y}  \  \ol{Z}}
          (0,-2) node (c') {\ol{\bigl( X  Y \bigr)  Z}}
          ;
          \draw[1cell=.8] 
          (a) edge node {\dA} (b)
          (b) edge node {1 \Phi} (c)
          (c) edge node {\Phi} (d)
          (a) edge['] node {\Phi 1} (b')
          (b') edge['] node {\Phi} (c')
          (c') edge['] node {\ol{\dA}} (d)
          ;
          \draw[2cell]
          node[between=b' and c at .5, rotate=45, 2label={above,\om}] {\Rightarrow}
          ;
        \end{tikzpicture}
        \qquad
        \begin{tikzpicture}[x=25mm,y=15mm,vcenter]
          \draw[0cell=.8] 
          (0,0) node (la) {S \  \ol{X}}
          (1,0) node (lb) { \ol{X}}
          (0,-1) node (la') { \ol{S} \   \ol{X}}
          (1,-1) node (lb') {\ol{S X}}
          (0,-1.75) node (ra) { \ol{X} \  S}
          (1,-1.75) node (rb) { \ol{X}}
          (0,-2.75) node (ra') { \ol{X} \  \ol{S}}
          (1,-2.75) node (rb') {\ol{X S}}
          ;
          \draw[1cell=.8] 
          (la) edge node {\dL} (lb)
          (la') edge['] node {\Phi} (lb')
          (la) edge['] node {\Psi 1} (la')
          (lb') edge['] node {\ol{\dL}} (lb)
          (ra) edge node {\dR} (rb)
          (ra') edge['] node {\Phi} (rb')
          (ra) edge['] node {1 \Psi} (ra')
          (rb') edge['] node {\ol{\dR}} (rb)
          ;
          \draw[2cell] 
          node[between=la and lb' at .5, rotate=-90, 2label={above,\;\ga^\dL}] {\Rightarrow}
          node[between=ra and rb' at .5, rotate=90, 2label={below,\;\ga^\dR}] {\Rightarrow}
          ;
        \end{tikzpicture}
      \]
    \item a symmetry modification $\ga^\dB$ as follows.
      \[
        \begin{tikzpicture}[x=25mm,y=15mm]
          \draw[0cell=.8] 
          (0,0) node (sa) { \ol{X} \  \ol{Y}}
          (1,0) node (sb) { \ol{Y} \  \ol{X}}
          (0,-1) node (sa') {\ol{X Y}}
          (1,-1) node (sb') {\ol{Y X}}
          ;
          \draw[1cell=.8] 
          (sa) edge node {\dB} (sb)
          (sa') edge['] node {\ol{\dB}} (sb')
          (sa) edge['] node {\Phi} (sa')
          (sb) edge node {\Phi} (sb')
          ;
          \draw[2cell] 
          node[between=sa and sb' at .5, rotate=45, 2label={above,\ga^\dB}] {\Rightarrow}
          ;
        \end{tikzpicture}
      \]
    \end{itemize}
  \end{description}
  These data satisfy two monoidal axioms \cite[(HTA1) and (HTA2), pp.~17--18]{GPS95Coherence}, two braid axioms \cite[(BHA1) and (BHA2), pp.~141--142]{McCru00Balanced}, and one sylleptic axiom \cite[(SHA1), p.~145]{McCru00Balanced}.
\end{defn}

We give the axioms for a symmetric monoidal 2-functor under the simplifying assumption that the monoidal products on its source and target are 2-functors.
With this assumption, modifications denoted $\om1$ etc. do not require implicit uses of the pseudofunctoriality constraints $\otimes_2$ or $\otimes_0$ in order to have the sources and targets as indicated below.
Moreover, 2-functoriality of $H = \ol{(\ )}$ implies that we have, for example,
\[
  1_{\ol{X}} = \ol{1_X} \andspace \ol{\dA \circ \dA} = \ol{\dA} \circ \ol{\dA}
\]
in the diagrams below.

In the axioms below, we denote the pseudonaturality constraints of $\dA$ and $\Phi$ with subscripts such as $\dA_{1, \Phi, 1}$ or $\Phi_{\dB,1}$.
These constraints will be identities in our application, but we include these, and other data, for completeness.
We will also need the following two mates associated to the data above.
\begin{equation}\label{eq:sm2fun-mates}
  \begin{tikzpicture}[x=30mm,y=15mm,vcenter]
    \draw[0cell=.8] 
    (0,0) node (a) {\bigl(  \ol{X} \   \ol{Y} \bigr) \;  \ol{Z}}
    (-1,0) node (b) { \ol{X} \; \bigl(  \ol{Y} \   \ol{Z} \bigr)}
    (-1,-1) node (c) { \ol{X} \   \ol{Y  Z} }
    (-1,-2) node (d) {\ol{X  \bigl( Y  Z \bigr)}}
    (0,-1) node (b') {\ol{X  Y} \   \ol{Z}}
    (0,-2) node (c') {\ol{\bigl( X  Y \bigr)  Z}}
    ;
    \draw[1cell=.8] 
    (b) edge node {\dAdot} (a)
    (b) edge['] node {1 \; \Phi} (c)
    (c) edge['] node {\Phi} (d)
    (a) edge node {\Phi \; 1} (b')
    (b') edge node {\Phi} (c')
    (d) edge node {\ol{\dAdot}} (c')
    ;
    \draw[2cell]
    node[between=b' and c at .5, rotate=45, 2label={above,\omdot}] {\Rightarrow}
    ;
  \end{tikzpicture}
  \qquad
  \begin{tikzpicture}[x=25mm,y=15mm,vcenter]
    \draw[0cell=.8] 
    (0,-1.75) node (ra) { \ol{X} \  S}
    (1,-1.75) node (rb) { \ol{X}}
    (0,-2.75) node (ra') { \ol{X} \  \ol{S}}
    (1,-2.75) node (rb') {\ol{X S}}
    ;
    \draw[1cell=.8] 
    (rb) edge['] node {\dRdot} (ra)
    (ra') edge['] node {\Phi} (rb')
    (ra) edge['] node {1 \; \Psi} (ra')
    (rb) edge node {\ol{\dRdot}} (rb')
    ;
    \draw[2cell] 
    node[between=ra and rb' at .5, shift={(0,-.1)}, rotate=180, 2label={below,\;\ga^\dR_1}] {\Rightarrow}
    ;
  \end{tikzpicture}
\end{equation}

The five axioms for a symmetric monoidal 2-functor
\[
  H \cn (\B, \otimes) \to (\C, \otimes),
\]
in which both monoidal products $\otimes$ are 2-functorial, are as follows.

\noindent\textbf{Monoidal Axiom}
\begin{equation}\label{eq:sm2fun-monoidal}
  \begin{tikzpicture}[x=30mm,y=30mm,scale=.5,baseline={(eq.base)}]
    \smallnodes
    \def\wl{3} 
    \def\wr{0} 
    \def\h{-.5} 
    \def\m{.8} 
    \draw[font=\Huge] (0,0) node [rotate=0] (eq) {$=$};
    \newcommand{\boundary}{
      \draw[0cell=\smallzero] 
      (0,0) node (a) {
        \bigl( \bigl( \ol{X} \ \ol{Y} \bigr) \; \ol{Z} \bigr) \; \ol{W}
      }
      (a)+(0,.8) node (b) {
        \bigl(\ol{XY} \ \ol{Z}\bigr) \; \ol{W}
      }
      (b)+(.5,.8) node (c) {
        \ol{\bigl(XY\bigr)Z} \ \ol{W}
      }
      (c)+(1,.5) node (d) {
        \ol{\bigl(\bigl(XY\bigr)Z\bigr)W}
      }
      (d)+(1,-.5) node (e) {
        \ol{\bigl(X\bigl(YZ\bigr)\bigr)W}
      }
      (e)+(.5,-.8) node (f) {
        \ol{X\bigl(\bigl(YZ\bigr)W\bigr)}
      }
      (f)+(0,-.8) node (g) {
        \ol{X\bigl(Y\bigl(ZW\bigr)\bigr)}
      }
      (a)+(.3,-.5) node (x) {
        \bigl(\ol{X} \ \ol{Y}\bigr) \; \bigl(\ol{Z} \ \ol{W}\bigr) 
      }
      (x)+(.5,-.5) node (y) {
        \ol{X} \; \bigl(\ol{Y} \; \bigl(\ol{Z} \ \ol{W}\bigr)\bigr)
      }
      (g)+(-.3,-.5) node (w) {
        \ol{X} \; \bigl(\ol{Y \bigl(ZW\bigr)}\bigr)
      }
      (w)+(-.5,-.5) node (z) {
        \ol{X} \; \bigl(\ol{Y} \ \ol{ZW}\bigr)
      }
      ;
      \draw[1cell]
      (a) edge node {\bigl(\Phi   1\bigr)  1} (b)
      (b) edge node {\Phi   1} (c)
      (c) edge node {\Phi} (d)
      (d) edge node {\ol{\dA 1}} (e)
      (e) edge node {\ol{\dA}} (f)
      (f) edge node {\ol{1 \dA}} (g)
      (a) edge['] node[pos=.3] {\dA} (x)
      (x) edge['] node[pos=.4] {\dA} (y)
      (y) edge['] node {1 \bigl(1 \Phi\bigr)} (z)
      (z) edge['] node[pos=.6] {1 \Phi} (w)
      (w) edge['] node[pos=.7] {\Phi} (g)
      ;
    }
    \begin{scope}[shift={(-\wl-\m,\h)}]
      \boundary 
      \draw[0cell=\smallzero]
      (d)+(0,-.7) node (t) {
        \ol{X \bigl(YZ\bigr)} \; \ol{W}
      }
      (t)+(-.25,-.4) node (q) {
        \bigl(\ol{X} \ \ol{YZ}\bigr) \; \ol{W}
      }
      (q)+(0,-1.2) node (r) {
        \ol{X} \; \bigl(\bigl(\ol{Y} \ \ol{Z}\bigr) \; \ol{W}\bigr)
      }
      (q)+(-.6,-.6) node (p) {
        \bigl(\ol{X} \; \bigl(\ol{Y} \ \ol{Z}\bigr)\bigr) \; \ol{W}
      }
      (q)+(.4,-.6) node (s) {
        \ol{X}\;\bigl(\ol{YZ} \  \ol{W}\bigr)
      }
      (s)+(.6,-.3) node (v) {
        \ol{X} \; \ol{\bigl(YZ\bigr)W}
      }
      ;
      \draw[1cell] 
      (a) edge node {\dA1} (p)
      (p) edge node[pos=.4] {\bigl(1\Phi\bigr)1} (q)
      (q) edge node {\dA} (s)
      (s) edge node {1\Phi} (v)
      (v) edge['] node[pos=.3] {1\ol{\dA}} (w)
      (p) edge['] node[pos=.3] {\dA} (r)
      (r) edge['] node {1\bigl( \Phi1 \bigr)} (s)
      (r) edge['] node {1 \dA} (y)
      (v) edge node {\Phi} (f)
      (q) edge['] node[pos=.7] {\Phi 1} (t)
      (c) edge['] node {\ol{\dA} 1} (t)
      (t) edge['] node {\Phi} (e)
      ;
      \draw[2cell]
      node[between=d and t at .5, rotate=-90, 2label={above,\Phi_{\dA,1}}] {\Rightarrow}
      node[between=c and p at .4, rotate=-90, 2label={above,\,\om 1}] {\Rightarrow}
      node[between=q and f at .5, rotate=-90, 2label={above,\,\om}] {\Rightarrow}
      node[between=p and s at .5, shift={(-.1,.15)}, rotate=-90, 2label={above,\dA_{1,\Phi,1}}] {\Rightarrow}
      node[between=a and r at .5, shift={(-.05,-.05)}, rotate=-90, 2label={above,\,\pi}] {\Rightarrow}
      node[between=y and v at .5, rotate=-90, 2label={above,\,1\om}] {\Rightarrow}
      node[between=v and g at .5, shift={(0,.15)}, rotate=-90, 2label={above,\Phi_{1,\dA}}] {\Rightarrow}
      ;
    \end{scope}  
    \begin{scope}[shift={(\wr+\m,\h)}]
      \boundary 
      \draw[0cell=\smallzero]
      (x)+(60:1) node (p') {
        \ol{XY} \; \bigl( \ol{Z} \ \ol{W}\bigr)
      }
      (x)+(0:1) node (r') {
        \bigl( \ol{X} \ \ol{Y} \bigr)\; \ol{ZW}
      }
      (x)+(30:1.6) node (q') {
        \ol{XY} \ \ol{ZW}
      }
      (e)+(225:.8) node (s') {
        \ol{\bigl( XY\bigr) \bigl(ZW\bigr)}
      }
      ;
      \draw[1cell]
      (b) edge node {\dA} (p')
      (x) edge['] node[pos=.7] {\Phi\bigl(11\bigr)} (p')
      (p') edge node[pos=.2] {1\Phi} (q')
      (x) edge node {\bigl(11\bigr)\Phi} (r')
      (r') edge['] node {\Phi 1} (q')
      (q') edge node {\Phi} (s')
      (s') edge node {\ol{\dA}} (g)
      (d) edge node {\ol{\dA}} (s')
      (r') edge node {\dA} (z)
      ;
      \draw[2cell]
      node[between=s' and e at .5, rotate=-90, 2label={above,\,\ol{\pi}}] {\Rightarrow}
      node[between=d and p' at .5, rotate=-90, 2label={above,\,\om}] {\Rightarrow}
      node[between=s' and z at .5, rotate=-90, 2label={above,\,\om}] {\Rightarrow}
      node[between=a and p' at .5, shift={(-.15,0.05)}, rotate=-90, 2label={above,\dA_{\Phi,1,1}}] {\Rightarrow}
      node[between=r' and y at .5, rotate=-90, 2label={above,\dA_{1,1,\Phi}}] {\Rightarrow}
      ;
    \end{scope}  
  \end{tikzpicture}
\end{equation}

\noindent\textbf{Middle Unity Axiom}
\begin{equation}\label{eq:sm2fun-midunity}
  \begin{tikzpicture}[x=30mm,y=30mm,scale=.5,baseline={(eq.base)}]
    \smallnodes
    \def\wl{3} 
    \def\wr{0} 
    \def\h{-1.25} 
    \def\m{.6} 
    \draw[font=\Huge] (0,0) node [rotate=0] (eq) {$=$};
    \newcommand{\boundary}{
      \draw[0cell=\smallzero] 
      (0,0) node (a) {
        \ol{X} \ \ol{Y}
      }
      (a)+(0,2) node (b) {
        \ol{XY}
      }
      (b)+(1,1) node (c) {
        \ol{\bigl(XS\bigr)Y}
      }
      (c)+(1,0) node (d) {
        \ol{X\bigl(SY\bigr)}
      }
      (d)+(1,-1) node (e) {
        \ol{XY}
      }
      (e)+(0,-2) node (x) {
        \ol{X} \ \ol{Y}
      }
      ;
      \draw[1cell]
      (a) edge node {\Phi} (b)
      (b) edge node {\ol{\dRdot 1}} (c)
      (c) edge node {\ol{\dA}} (d)
      (d) edge node {\ol{1\dL}} (e)
      (a) edge['] node {1} (x)
      (x) edge['] node {\Phi} (e)
      ;
    }
    \begin{scope}[shift={(-\wl-\m,\h)}]
      \boundary 
      \draw[0cell=\smallzero]
      (c)+(-.3,-1) node (r) {
        \ol{\bigl(XS\bigr)} \ \ol{Y}
      }
      (d)+(.3,-1) node (u) {
        \ol{X} \ \ol{SY}
      }
      (r)+(.4,-.7) node (s) {
        \bigl(\ol{X} \ \ol{S}\bigr) \; \ol{Y}
      }
      (u)+(-.4,-.7) node (t) {
        \ol{X} \; \bigl(\ol{S} \ \ol{Y}\bigr)
      }
      (s)+(0,-.7) node (p) {
        \bigl(\ol{X} \ S\bigr) \; \ol{Y}
      }
      (t)+(0,-.7) node (q) {
        \ol{X} \; \bigl(S \ \ol{Y}\bigr)
      }
      ;
      \draw[1cell] 
      (a) edge node[pos=.7] {\ol{\dRdot}\ 1} (r)
      (r) edge['] node {\Phi} (c)
      (a) edge['] node {\dRdot 1} (p)
      (p) edge node {\bigl(1 \Psi \bigr) 1} (s)
      (s) edge node {\Phi 1} (r)
      (s) edge node {\dA} (t)
      (t) edge node {1 \Phi} (u)
      (u) edge node {\Phi} (d)
      (u) edge node[pos=.3] {1 \ol{\dL}} (x)
      (p) edge node {\dA} (q)
      (q) edge node {1\dL} (x)
      (q) edge['] node {1\bigl(\Psi 1\bigr)} (t)
      ;
      \draw[2cell]
      node[between=b and r at .5, shift={(-.15,-.15)}, rotate=-90, 2label={above,\Phi_{\dRdot,1}}] {\Rightarrow}
      node[between=u and e at .5, shift={(-.05,-.15)}, rotate=-90, 2label={above,\Phi_{1,\dL}}] {\Rightarrow}
      node[between=r and u at .5, shift={(0,.2)}, rotate=-90, 2label={above,\,\om}] {\Rightarrow}
      node[between=p and t at .45, rotate=-45, 2label={above,\dA_{1,\dL,1}}] {\Rightarrow}
      (p)+(180:.5) node[rotate=0, 2label={above,\ga^{\dR}_1 1}] {\Rightarrow}
      (q)+(10:.45) node[rotate=-90, 2label={above,1 \ga^{\dL}}] {\Rightarrow}
      node[between=p and q at .5, shift={(0,-.25)}, rotate=-90, 2label={above,\,\mu}] {\Rightarrow}
      ;
    \end{scope}  
    \begin{scope}[shift={(\wr+\m,\h)}]
      \boundary 
      \draw[1cell] 
      (b) edge node {1} (e)
      ;
      \draw[2cell]
      node[between=a and e at .5, rotate=-90, 2label={above,\Phi_{1,1}}] {\Rightarrow}
      node[between=c and d at .5, shift={(0,-.35)}, rotate=-90, 2label={above,\ol{\mu}}] {\Rightarrow}
      ;
    \end{scope}
  \end{tikzpicture}
\end{equation}

\noindent\textbf{Left Braid Axiom}
\begin{equation}\label{eq:sm2fun-lbraid}
  \begin{tikzpicture}[x=27mm,y=30mm,scale=.5,baseline={(eq.base)}]
    \smallnodes
    \def\wl{3} 
    \def\wr{0} 
    \def\h{.7} 
    \def\m{.6} 
    \draw[font=\Huge] (0,0) node [rotate=0] (eq) {$=$};
    \newcommand{\boundary}{
      \draw[0cell=\smallzero] 
      (0,0) node (a) {
        \bigl(\ol{X}\ \ol{Y}\bigr)\; \ol{Z}
      }
      (a)+(1,.5) node (b) {
        \ol{X} \; \bigl(\ol{Y} \ \ol{Z}\bigr)
      }
      (b)+(1,0) node (c) {
        \bigl(\ol{Y} \ \ol{Z}\bigr) \; \ol{X}
      }
      (c)+(1,-.5) node (d) {
        \ol{Y} \; \bigl(\ol{Z} \ \ol{X}\bigr)
      }
      (d)+(0,-.8) node (e) {
        \ol{Y} \ \ol{ZX}
      }
      (e)+(0,-.7) node (f) {
        \ol{Y \bigl( ZX \bigr)}
      }
      (f)+(-1,-.5) node (w) {
        \ol{Y \bigl( XZ \bigr)}
      }
      (a)+(0,-.8) node (x) {
        \ol{XY} \ \ol{Z}
      }
      (x)+(0,-.7) node (y) {
        \ol{\bigl(XY\bigr) Z}
      }
      (y)+(1,-.5) node (z) {
        \ol{\bigl(YX\bigr)Z}
      }
      ;
      \draw[1cell] 
      (a) edge node {\dA} (b)
      (b) edge node {\dB} (c)
      (c) edge node {\dA} (d)
      (d) edge node {1\Phi} (e)
      (e) edge node {\Phi} (f)
      (a) edge['] node {\Phi 1} (x)
      (x) edge['] node {\Phi} (y)
      (y) edge['] node {\ol{\dB 1}} (z)
      (z) edge['] node {\ol{\dA}} (w)
      (w) edge['] node {\ol{1\dB}} (f)
      ;
    }
    \begin{scope}[shift={(-\wl-\m,\h)}]
      \boundary 
      \draw[0cell=\smallzero]
      (b)+(0,-.75) node (p) {
        \ol{X} \ \ol{YZ}
      }
      (c)+(0,-.75) node (q) {
        \ol{YZ} \ \ol{X}
      }
      (p)+(0,-.75) node (r) {
        \ol{X\bigl(YZ\bigr)}
      }
      (q)+(0,-.75) node (s) {
        \ol{\bigl(YZ\bigr)X}
      }
      ;
      \draw[1cell] 
      (b) edge['] node {1\Phi} (p)
      (p) edge['] node {\Phi} (r)
      (c) edge node {\Phi 1} (q)
      (q) edge node {\Phi} (s)
      (p) edge node {\dB} (q)
      (y) edge node {\ol{\dA}} (r)
      (r) edge node {\ol{\dB}} (s)
      (s) edge node {\ol{\dA}} (f)
      ;
      \draw[2cell]
      node[between=b and q at .55, rotate=45, 2label={above,\dB_{1,\Phi}}] {\Rightarrow}
      node[between=p and s at .55, rotate=45, 2label={above,\ga^\dB}] {\Rightarrow}
      node[between=a and r at .55, rotate=45, 2label={above,\om}] {\Rightarrow}
      node[between=q and e at .55, rotate=45, 2label={above,\om}] {\Rightarrow}
      node[between=r and w at .5, shift={(-.06,0)}, rotate=90, 2label={below,\,\ol{R}_{X|YZ}}] {\Rightarrow}
      ;
    \end{scope}  
    \begin{scope}[shift={(\wr+\m,\h)}]
      \boundary 
      \draw[0cell=\smallzero]
      (b)+(0,-1) node (p') {
        \bigl(\ol{Y} \ \ol{X}\bigr) \; \ol{Z} 
      }
      (c)+(0,-1) node (q') {
        \ol{Y} \; \bigl(\ol{X} \ \ol{Z}\bigr)
      }
      (p')+(0,-.7) node (r') {
        \ol{YX} \ \ol{Z}
      }
      (q')+(0,-.7) node (s') {
        \ol{Y} \ \ol{XZ}
      }
      ;
      \draw[1cell] 
      (a) edge node {\dB 1} (p')
      (p') edge node {\dA} (q')
      (q') edge node {1 \dB} (d)
      (x) edge['] node {\ol{\dB} 1} (r')
      (s') edge['] node {1 \ol{\dB}} (e)
      (p') edge['] node {\Phi 1} (r')
      (r') edge['] node {\Phi} (z)
      (q') edge node {1 \Phi} (s')
      (s') edge node {\Phi} (w)
      ;
      \draw[2cell]
      node[between=b and q' at .5, rotate=90, 2label={below,\,R_{\ol{X}|\ol{Y}\ol{Z}}}] {\Rightarrow}
      node[between=a and r' at .55, rotate=45, 2label={above,\ga^\dB 1}] {\Rightarrow}
      node[between=q' and e at .6, rotate=45, 2label={above,1 \ga^\dB}] {\Rightarrow}
      node[between=x and z at .55, rotate=45, 2label={above,\Phi_{\dB,1}}] {\Rightarrow}
      node[between=s' and f at .6, rotate=45, 2label={above,\Phi_{1,\dB}}] {\Rightarrow}
      node[between=r' and s' at .5, shift={(0.05,0)}, rotate=45, 2label={above,\om}] {\Rightarrow}
      ;
    \end{scope}  
  \end{tikzpicture}
\end{equation}

\noindent\textbf{Right Braid Axiom}
\begin{equation}\label{eq:sm2fun-rbraid}
  \begin{tikzpicture}[x=27mm,y=30mm,scale=.5,baseline={(eq.base)}]
    \smallnodes
    \def\wl{3} 
    \def\wr{0} 
    \def\h{.7} 
    \def\m{.6} 
    \draw[font=\Huge] (0,0) node [rotate=0] (eq) {$=$};
    \newcommand{\boundary}{
      \draw[0cell=\smallzero] 
      (0,0) node (a) {
        \ol{X} \; \bigl(\ol{Y}\ \ol{Z}\bigr)
      }
      (a)+(1,.5) node (b) {
        (\ol{X} \ \ol{Y}) \; \ol{Z}
      }
      (b)+(1,0) node (c) {
        \ol{Z} \; (\ol{X} \ \ol{Y})
      }
      (c)+(1,-.5) node (d) {
        (\ol{Z} \ \ol{X} )\; \ol{Y}
      }
      (d)+(0,-.8) node (e) {
        \ol{ZX} \ \ol{Y}
      }
      (e)+(0,-.7) node (f) {
        \ol{(ZX)Y}
      }
      (f)+(-1,-.5) node (w) {
        \ol{(XZ)Y}
      }
      (a)+(0,-.8) node (x) {
        \ol{X} \ \ol{YZ}
      }
      (x)+(0,-.7) node (y) {
        \ol{X ( Y Z )}
      }
      (y)+(1,-.5) node (z) {
        \ol{X ( Z Y )}
      }
      ;
      \draw[1cell] 
      (a) edge node {\dAdot} (b)
      (b) edge node {\dB} (c)
      (c) edge node {\dAdot} (d)
      (d) edge node {\Phi1} (e)
      (e) edge node {\Phi} (f)
      (a) edge['] node {1\Phi} (x)
      (x) edge['] node {\Phi} (y)
      (y) edge['] node {\ol{1\dB}} (z)
      (z) edge['] node {\ol{\dAdot}} (w)
      (w) edge['] node {\ol{\dB1}} (f)
      ;
    }
    \begin{scope}[shift={(-\wl-\m,\h)}]
      \boundary 
      \draw[0cell=\smallzero]
      (b)+(0,-.75) node (p) {
        \ol{XY} \ \ol{Z}
      }
      (c)+(0,-.75) node (q) {
        \ol{Z} \ \ol{XY}
      }
      (p)+(0,-.7) node (r) {
        \ol{\bigl(XY\bigr) Z}
      }
      (q)+(0,-.7) node (s) {
        \ol{Z \bigl(XY\bigr)}
      }
      ;
      \draw[1cell] 
      (b) edge['] node {\Phi 1} (p)
      (p) edge['] node {\Phi} (r)
      (c) edge node {1\Phi} (q)
      (q) edge node {\Phi} (s)
      (p) edge node {\dB} (q)
      (y) edge node {\ol{\dAdot}} (r)
      (r) edge node {\ol{\dB}} (s)
      (s) edge node {\ol{\dAdot}} (f)
      ;
      \draw[2cell]
      node[between=b and q at .55, rotate=45, 2label={above,\dB_{\Phi,1}}] {\Rightarrow}
      node[between=p and s at .55, rotate=45, 2label={above,\ga^\dB}] {\Rightarrow}
      node[between=a and r at .55, rotate=45, 2label={above,\omdot}] {\Rightarrow}
      node[between=q and e at .55, rotate=45, 2label={above,\omdot}] {\Rightarrow}
      node[between=r and w at .5, shift={(-.06,0)}, rotate=90, 2label={below,\,\ol{R}_{XY|Z}}] {\Rightarrow}
      ;
    \end{scope}  
    \begin{scope}[shift={(\wr+\m,\h)}]
      \boundary 
      \draw[0cell=\smallzero]
      (b)+(0,-1) node (p') {
        \ol{X} \; \bigl(\ol{Z} \ \ol{Y}\bigr)
      }
      (c)+(0,-1) node (q') {
        \bigl(\ol{X} \ \ol{Z}\bigr) \; \ol{Y}
      }
      (p')+(0,-.7) node (r') {
        \ol{X} \ \ol{ZY}
      }
      (q')+(0,-.7) node (s') {
        \ol{XZ} \ \ol{Y}
      }
      ;
      \draw[1cell] 
      (a) edge node {1 \dB} (p')
      (p') edge node {\dAdot} (q')
      (q') edge node {\dB 1} (d)
      (x) edge['] node {1 \ol{\dB}} (r')
      (s') edge['] node {\ol{\dB} 1} (e)
      (p') edge['] node {1\Phi} (r')
      (r') edge['] node {\Phi} (z)
      (q') edge node {\Phi 1} (s')
      (s') edge node {\Phi} (w)
      ;
      \draw[2cell]
      node[between=b and q' at .5, rotate=90, 2label={below,\,R_{\ol{X}\ol{Y}|\ol{Z}}}] {\Rightarrow}
      node[between=a and r' at .55, rotate=45, 2label={above,1\ga^\dB}] {\Rightarrow}
      node[between=q' and e at .6, rotate=45, 2label={above,\ga^\dB1}] {\Rightarrow}
      node[between=x and z at .55, rotate=45, 2label={above,\Phi_{1,\dB}}] {\Rightarrow}
      node[between=s' and f at .6, rotate=45, 2label={above,\Phi_{\dB,1}}] {\Rightarrow}
      node[between=r' and s' at .5, shift={(0.05,0)}, rotate=45, 2label={above,\omdot}] {\Rightarrow}
      ;
    \end{scope}  
  \end{tikzpicture}
\end{equation}

\noindent\textbf{Sylleptic Axiom}
\begin{equation}\label{eq:sm2fun-syll}
  \begin{tikzpicture}[x=30mm,y=20mm,scale=.5,baseline={(eq.base)}]
    \smallnodes
    \def\wl{2} 
    \def\wr{0} 
    \def\h{0.2} 
    \def\m{.6} 
    \draw[font=\Huge] (0,0) node [rotate=0] (eq) {$=$};
    \newcommand{\boundary}{
      \draw[0cell=\smallzero] 
      (0,0) node (a) {\ol{X} \ \ol{Y}}
      (a)+(1,1) node (b) {\ol{Y} \ \ol{X}}
      (b)+(1,-1) node (c) {\ol{X} \ \ol{Y}}
      (c)+(0,-1) node (d) {\ol{XY}}
      (a)+(0,-1) node (x) {\ol{XY}}
      ;
      \draw[1cell]
      (a) edge node {\dB} (b)
      (b) edge node {\dB} (c)
      (c) edge node {\Phi} (d)
      (a) edge['] node {\Phi} (x)
      (x) edge['] node {1} (d)
      ;
    }
    \begin{scope}[shift={(-\wl-\m,\h)}]
      \boundary 
      \draw[1cell] 
      (a) edge['] node {1} (c)
      ;
      \draw[2cell]
      (b)+(0,-.6) node[rotate=90, 2label={above,\syl}] {\Rightarrow}
      node[between=a and d at .55, rotate=45, 2label={above,\Phi_{1,1}}] {\Rightarrow}
      ;
    \end{scope}  
    \begin{scope}[shift={(\wr+\m,\h)}]
      \boundary 
      \draw[0cell=\smallzero]
      (b)+(0,-1) node (p) {\ol{YX}}
      ;
      \draw[1cell]
      (x) edge node {\ol{\dB}} (p)
      (p) edge node {\ol{\dB}} (d)
      (b) edge node {\Phi} (p)
      ;
      \draw[2cell]
      (p)+(0,-.6) node[rotate=90, 2label={above,\syl}] {\Rightarrow}
      node[between=a and p at .55, rotate=45, 2label={above,\ga^\dB}] {\Rightarrow}
      node[between=p and c at .55, rotate=45, 2label={above,\ga^\dB}] {\Rightarrow}
      ;
    \end{scope}  
  \end{tikzpicture}
\end{equation}

\bibliographystyle{sty/amsalpha2}
\bibliography{sty/Refs}%

\providecommand{\bysame}{\leavevmode\hbox to3em{\hrulefill}\thinspace}
\providecommand{\MR}{\relax\ifhmode\unskip\space\fi MR }
\providecommand{\MRhref}[2]{%
  \href{http://www.ams.org/mathscinet-getitem?mr=#1}{#2}
}
\providecommand{\doi}[1]{%
  doi:\href{https://dx.doi.org/#1}{\nolinkurl{#1}}}
\providecommand{\arxiv}[1]{%
  arXiv:\href{https://arxiv.org/abs/#1}{#1}}
\begin{thebibliography}{CKWW07}

\bibitem[BKP89]{BKP1989Two}
R.~{Blackwell}, G.~{Kelly}, and A.~{Power}, \emph{Two-dimensional monad
  theory}, {J. Pure Appl. Algebra} \textbf{59} (1989), no.~1, 1--41.
  \doi{10.1016/0022-4049(89)90160-6}

\bibitem[Bou17]{Bou17Skew}
J.~Bourke, \emph{Skew structures in 2-category theory and homotopy theory}, J.
  Homotopy Relat. Struct. \textbf{12} (2017), no.~1, 31--81 (English).
  \doi{10.1007/s40062-015-0121-z}

\bibitem[CKWW07]{CKWW07Cartesian}
A.~Carboni, G.~M. Kelly, R.~F.~C. Walters, and R.~J. Wood, \emph{Cartesian
  bicategories {II}}, Theory Appl. Categ. \textbf{19} (2007), no.~6, 93--124.

\bibitem[CG14]{CG2014Iterated}
E.~Cheng and N.~Gurski, \emph{Iterated icons}, Theory and Applications of
  Categories \textbf{29} (2014), no.~32, 929--977.

\bibitem[EM06]{EM2006Rings}
A.~D. Elmendorf and M.~A. Mandell, \emph{Rings, modules, and algebras in
  infinite loop space theory}, Adv. Math. \textbf{205} (2006), no.~1, 163--228.
  \doi{10.1016/j.aim.2005.07.007}

\bibitem[Fio06]{Fio06Pseudo}
T.~M. Fiore, \emph{{P}seudo {L}imits, {B}iadjoints, and {P}seudo {A}lgebras:
  {C}ategorical {F}oundations of {C}onformal {F}ield {T}heory}, Mem. Amer.
  Math. Soc. \textbf{182} (2006), no.~860, x+171. \doi{10.1090/memo/0860}

\bibitem[FS]{FS2019Supplying}
B.~Fong and D.~I. Spivak, \emph{Supplying bells and whistles in symmetric
  monoidal categories}, 2019. \arxiv{1908.02633}

\bibitem[GGN15]{GGN15Universality}
D.~Gepner, M.~Groth, and T.~Nikolaus, \emph{Universality of multiplicative
  infinite loop space machines}, Algebr. Geom. Topol. \textbf{15} (2015),
  no.~6, 3107--3153 (English). \doi{10.2140/agt.2015.15.3107}

\bibitem[GPS95]{GPS95Coherence}
R.~Gordon, A.~J. Power, and R.~Street, \emph{Coherence for tricategories}, vol.
  117, Mem. Amer. Math. Soc., no. 558, Amer. Math. Soc., 1995.
  \doi{10.1090/memo/0558}

\bibitem[Gra74]{Gra74Formal}
J.~W. Gray, \emph{Formal category theory: adjointness for {$2$}-categories},
  Lecture Notes in Mathematics, vol. 391, Springer-Verlag, Berlin-New York,
  1974. \doi{10.1007/BFb0061280}

\bibitem[Gra80]{Gray80Closed}
\bysame, \emph{Closed categories, lax limits and homotopy limits}, J. Pure
  Appl. Algebra \textbf{19} (1980), 127--158.
  \doi{10.1016/0022-4049(80)90098-5}

\bibitem[Gur11]{gurski11}
N.~Gurski, \emph{Loop spaces, and coherence for monoidal and braided monoidal
  bicategories}, Adv. Math. \textbf{226} (2011), no.~5, 4225--4265.
  \doi{10.1016/j.aim.2010.12.007}

\bibitem[Gur12]{Gur2012Biequivalences}
\bysame, \emph{Biequivalences in tricategories}, Theory Appl. Categ.
  \textbf{26} (2012), no.~14, 349--384.

\bibitem[Gur13a]{Gurski13Coherence}
\bysame, \emph{Coherence in three-dimensional category theory}, Cambridge
  Tracts in Mathematics, vol. 201, Cambridge University Press, Cambridge, 2013.
  \doi{10.1017/CBO9781139542333}

\bibitem[Gur13b]{Gur2013monoidal}
\bysame, \emph{The monoidal structure of strictification}, Theory Appl. Categ.
  \textbf{28} (2013), no.~1, 1--23.

\bibitem[GO13]{GO2012Infinite}
N.~Gurski and A.~M. Osorno, \emph{Infinite loop spaces, and coherence for
  symmetric monoidal bicategories}, Adv. Math. \textbf{246} (2013), 1--32.
  \doi{10.1016/j.aim.2013.06.028}

\bibitem[HP02]{HP02Pseudo}
M.~Hyland and J.~Power, \emph{Pseudo-commutative monads and pseudo-closed
  2-categories}, J. Pure Appl. Algebra \textbf{175} (2002), no.~1-3, 141--185
  (English). \doi{10.1016/S0022-4049(02)00133-0}

\bibitem[JY]{JYringIII}
N.~Johnson and D.~Yau, \emph{{B}imonoidal {C}ategories, {$E_n$}-{M}onoidal
  {C}ategories, and {A}lgebraic {$K$}-{T}heory. {V}olume {III}: {F}rom
  {C}ategories to {S}tructured {R}ing {S}pectra}, available at
  \url{https://nilesjohnson.net}, 2021. \arxiv{2107.10526}

\bibitem[JY21]{JY212Dim}
\bysame, \emph{{2}-{D}imensional {C}ategories}, Oxford University Press, New
  York, 2021. \doi{10.1093/oso/9780198871378.001.0001} \arxiv{2002.06055}

\bibitem[JS93]{JS1993Braided}
A.~Joyal and R.~Street, \emph{Braided tensor categories}, Adv. Math.
  \textbf{102} (1993), no.~1, 20--78. \doi{10.1006/aima.1993.1055}

\bibitem[Kel74]{Kelly1974Doctrinal}
G.~M. Kelly, \emph{Doctrinal adjunction}, Category {S}eminar ({P}roc. {S}em.,
  {S}ydney, 1972/1973), Lecture Notes in Mathematics, vol. 420, Springer,
  Berlin, 1974, pp.~257--280 (English). \doi{10.1007/BFb0063105}

\bibitem[Kel89]{Kel89Elementary}
\bysame, \emph{Elementary observations on {$2$}-categorical limits}, Bull.
  Austral. Math. Soc. \textbf{39} (1989), no.~2, 301--317.
  \doi{10.1017/S0004972700002781}

\bibitem[LS12]{LS2012Enhanced}
S.~Lack and M.~Shulman, \emph{Enhanced 2-categories and limits for lax
  morphisms}, Adv. Math. \textbf{229} (2012), no.~1, 294--356.
  \doi{10.1016/j.aim.2011.08.014}

\bibitem[ML98]{ML98Categories}
S.~Mac~Lane, \emph{Categories for the working mathematician}, 2nd ed., Grad.
  Texts Math., vol.~5, New York, NY: Springer, 1998 (English).
  \doi{10.1007/978-1-4757-4721-8}

\bibitem[May74]{May1974Einfty}
J.~P. May, \emph{{$E_{\infty }$} spaces, group completions, and permutative
  categories}, New developments in topology ({P}roc. {S}ympos. {A}lgebraic
  {T}opology, {O}xford, 1972), London Math. Soc. Lecture Note Ser, vol.~11,
  Cambridge Univ. Press, London, 1974, pp.~61--93.
  \doi{10.1017/CBO9780511662607.008}

\bibitem[May80]{may1980pairings}
J.~P. May, \emph{Pairings of categories and spectra}, Journal of Pure and
  Applied Algebra \textbf{19} (1980), 299--346.

\bibitem[McC00]{McCru00Balanced}
P.~McCrudden, \emph{Balanced coalgebroids}, Theory and Applications of
  Categories \textbf{7} (2000), no.~6, 71--147.

\bibitem[Pst22]{Pst22Dualizable}
P.~Pstrągowski, \emph{On dualizable objects in monoidal bicategories}, Theory
  Appl. Categ. \textbf{38} (2022), no.~9, 257--310.

\bibitem[Sch14]{Sch2014Ind}
D.~Sch{\"a}ppi, \emph{Ind-abelian categories and quasi-coherent sheaves},
  Mathematical Proceedings of the Cambridge Philosophical Society, vol. 157,
  Cambridge University Press, 2014, pp.~391--423.

\bibitem[Sch]{Sch07Tensor}
V.~Schmitt, \emph{Tensor product for symmetric monoidal categories}, 2007.
  \arxiv{0711.0324}

\bibitem[SP]{SP2011Classification}
C.~Schommer-Pries, \emph{The classification of two-dimensional extended
  topological field theories}, 2011. \arxiv{1112.1000v2}

\bibitem[Str80]{street1980fb}
R.~Street, \emph{{Fibrations in bicategories}}, Cahiers de Topologie et
  Geometrie Differentielle \textbf{21} (1980), no.~2, 111--160.

\bibitem[Str74]{Str74Cosmoi}
R.~Street, \emph{Elementary cosmoi. {I}}, Category {S}eminar ({P}roc. {S}em.,
  {S}ydney, 1972/1973), Lecture Notes in Math., Vol. 420, Springer, Berlin,
  1974, pp.~134--180.

\bibitem[Str96]{Str96Cat}
\bysame, \emph{Categorical structures}, Handbook of algebra, {V}ol. 1, Handb.
  Algebr., vol.~1, Elsevier/North-Holland, Amsterdam, 1996, pp.~529--577.
  \doi{10.1016/S1570-7954(96)80019-2}

\bibitem[Szl12]{Szl2012Skew}
K.~Szlach{\'a}nyi, \emph{Skew-monoidal categories and bialgebroids}, Advances
  in Mathematics \textbf{231} (2012), no.~3-4, 1694--1730.

\end{thebibliography}

\end{document}